\documentclass{amsart}
\usepackage{amssymb, graphics, color, enumitem}
\usepackage[all]{xy}

 \usepackage{tikz}

 \usepackage[applemac]{inputenc}
\usepackage[cyr]{aeguill}

\usepackage[margin = 1in]{geometry}

\newtheorem{lemma}{Lemma}[section]
\newtheorem{proposition}[lemma]{Proposition}
\newtheorem{corollary}[lemma]{Corollary}
\newtheorem{theorem}[lemma]{Theorem}

\newtheorem{example}[lemma]{Example}
\newtheorem{definition}[lemma]{Definition}
\newtheorem{remark}[lemma]{Remark}

\newtheorem*{Acknowledgement}{Acknowledgements}

\newtheorem{assumption}[lemma]{Assumption}

\newcommand\ie{i\@.e\@. }

\newcommand\pa{ \partial}

\newcommand\bbC{\mathbb C}

\newcommand\bbI{\mathbb I}

\newcommand\bbN{\mathbb N}
\newcommand\bbP{\mathbb P}

\newcommand\bbR{\mathbb R}
\newcommand\bbS{\mathbb S}

\newcommand\bbZ{\mathbb Z}
\newcommand\Crit{\operatorname{Crit}}

\usepackage{color}

\newcommand{\lrp}[1]{\left( {#1} \right)}

\newcommand\tX{\widetilde{X}}

\newcommand\hX{\widehat{X}}
\newcommand\hH{\widehat{H}}

\newcommand\hW{\widehat{W}}
\newcommand\CI{\mathcal{C}^{\infty}}

\newcommand\cC{\mathcal{C}}
\newcommand\cJ{\mathcal{J}}
\newcommand\cA{\mathcal{A}}

\newcommand\cD{\mathcal{D}}

\newcommand\cL{\mathcal{L}}
\newcommand\cR{\mathcal{R}}
\newcommand\db{\overline{\pa}}
\newcommand\tH{\widetilde{H}}
\newcommand\cV{\mathcal{V}}
\newcommand\cU{\mathcal{U}}
\newcommand\cI{\mathcal{I}}

\newcommand\End{\operatorname{End}}

\newcommand\tV{\widetilde{V}}

\newcommand\phg{\operatorname{phg}}
\newcommand\Spec{\operatorname{Spec}}
\newcommand\Id{\operatorname{Id}}

\newcommand\hG{\widehat{G}}

\newcommand\cH{\mathcal{H}}

\newcommand\hC{\widehat{C}}

\newcommand\cM{\mathcal{M}}

\newcommand\Vol{\operatorname{Vol}}

\newcommand\AC{\operatorname{AC}}

\newcommand\SU{\operatorname{SU}}

\newcommand\tr{\operatorname{tr}}

\newcommand{\hx}{\widehat{x}}

\newcommand\hZ{\widehat{Z}}
\newcommand\hV{\widehat{V}}
\newcommand\cP{\mathcal{P}}

\newcommand\QAC{\operatorname{QAC}}
\newcommand\nQAC{\mathfrak{n}\operatorname{QAC}}
\newcommand\nQFB{\mathfrak{n}\operatorname{QFB}}

\newcommand\Qb{\operatorname{Qb}}

\newcommand\nQb{\mathfrak{n}\operatorname{Qb}}
\newcommand\QFB{\operatorname{QFB}}

\newcommand\QALE{\operatorname{QALE}}

\newcommand\bV{\overline{V}}

\newcommand\sing{\operatorname{sing}}
\newcommand\zero{\operatorname{zero}}

\newcommand\cK{\mathcal{K}}

\newcommand\cB{\mathcal{B}}

\newcommand\cO{\mathcal{O}}
\newcommand\cS{\mathcal{S}}

\newcommand\nm{\operatorname{nm}}
\newcommand\wX{\widetilde{X}}
\newcommand\Diam{\operatorname{Diam}}
\newcommand\area{\operatorname{area}}

\begin{document}
\title[warped QAC Calabi-Yau metrics]
{ Warped quasi-asymptotically conical Calabi-Yau metrics}

\author{Ronan J.~Conlon}
\address{Department of Mathematical Sciences, The University of Texas at Dallas, Richardson, TX 75080, USA}
\email{ronan.conlon@utdallas.edu}

\author{Fr\'ed\'eric Rochon}
\address{D\'epartement de Math\'ematiques, Universit\'e du Qu\'ebec \`a Montr\'eal}
\email{rochon.frederic@uqam.ca}

\maketitle

\begin{abstract}
We construct many new examples of complete Calabi-Yau metrics of maximal volume growth on certain smoothings of Cartesian products of Calabi-Yau cones with smooth cross-sections, in particular with tangent cone at infinity having singularities of arbitrary depth.  A detailed description of the geometry at infinity of these metrics is given in terms of a compactification by a manifold with corners obtained through the notion of weighted blow-up.   A key analytical step in the construction of these Calabi-Yau metrics is to derive good mapping properties of the Laplacian on some suitable weighted H\"older spaces.  Our methods also produce singular Calabi-Yau metrics with an isolated conical singularity modelled on a Calabi-Yau cone distinct from the tangent cone at infinity, in particular yielding a transition behavior between different Calabi-Yau cones as conjectured by Yang Li.  This is used to exhibit many examples where the tangent cone at infinity does not uniquely specify a complete Calabi-Yau metric with exact K\"ahler form. 
\end{abstract}

\tableofcontents

\numberwithin{equation}{section}

\section{Introduction}

A complete K\"ahler manifold $(X,g,J)$ is Calabi-Yau if it is Ricci-flat and has a nowhere vanishing parallel holomorphic volume form $\Omega_X\in H^0(X;K_X)$.  This latter condition forces the holonomy of $g$ to be contained in $\SU(m)$ with $m=\dim_{\bbC}X$.  By the result of Yau \cite{Yau1978}, we know that a compact K\"ahler manifold admits a Calabi-Yau metric if and only if its canonical bundle is trivial, in which case a unique Calabi-Yau metric can be obtained in each K\"ahler class by solving a complex Monge-Amp\`ere equation.  On non-compact complete K\"ahler manifolds, the triviality of the canonical bundle is also a necessary condition for the existence of a Calabi-Yau metric, but one needs also to take into account the geometry at infinity.  For instance, on $\bbC^2$, the flat metric and the Taub-NUT metric are two complete Calabi-Yau metrics in the same K\"ahler class, but with quite distinct geometry at infinity, the volume growth of the latter being only cubic instead of order $4$.  

In this paper, we will focus on Calabi-Yau metrics of maximal volume growth, that is, such that the volume of a ball of radius $r$ is comparable to $r^{2m}$ for $r$ large with $m$ the complex dimension of the manifold.  A tangent cone at infinity of such a metric is then of the same dimension.  When such a tangent cone at infinity has a smooth cross-section, or equivalently when the Calabi-Yau metric has quadratic curvature decay, then by \cite{Colding-Minicozzi}, this is in fact the unique tangent cone at infinity of the metric.  By \cite{Sun-Zhang},  such Calabi-Yau metrics are asymptotically conical ($\AC$-metrics for short), that is, Calabi-Yau metrics converging smoothly at infinity at a rate $\mathcal{O}(r^{-\epsilon})$ for some $\epsilon>0$ to a Calabi-Yau cone with smooth cross-section.  Various examples have been obtained over the years by solving a complex Monge-Amp\`ere equation, notably in \cite{Joyce, vanC, vanC2011, Goto, CH2013,Ch2015}.  For a fixed Calabi-Yau cone at infinity $(C,g_C)$ with smooth cross-section, a complete classification of asymptotically conical Calabi-Yau manifolds with tangent cone at infinity $(C,g_C)$ was obtained in \cite{CH3}, generalizing in particular Kronheimer's classification \cite{Kronheimer1989} of asymptotically locally Euclidean hyperK\"ahler $4$-manifolds.  The upshot is that all such asymptotically conical Calabi-Yau metrics are obtained by considering a K\"ahler crepant resolution of a deformation of $C$ seen as an affine variety.  

Allowing the tangent cone at infinity to have a singular cross-section greatly opens up the possibilities of examples that can occur.  When $C= \bbC^m/\Gamma$ with $\Gamma$ a finite subgroup of $\SU(m)$ and $g_C$ is the metric induced by the Euclidean metric on $\bbC^m$, Joyce \cite{Joyce} constructed examples of Calabi-Yau metrics on K\"ahler crepant resolutions of $C$, the so called quasi-asymptotically locally Euclidean metrics ($\QALE$-metrics for short).  This was extended in \cite{CDR2016} to obtain Calabi-Yau quasi-asymptotically conical metrics ($\QAC$-metrics for short) in the sense of \cite{DM2014}.  If one considers instead smooth deformations of Calabi-Yau cones with singular cross-sections, then already on $\bbC^n$, many examples were obtained in \cite{YangLi, CR2021,Szekelyhidi}, providing in particular counter-examples to a conjecture of Tian \cite{Tian2006}.  As pointed out in \cite{CR2021}, these examples are not quite $\QAC$-metrics, but they are very close to being so in that they are in some sense warped $\QAC$-metrics.  The Calabi-Yau metrics of \cite{Biquard-Delcroix} constructed on complex symmetric spaces seem to have a similar behavior at infinity.  More recently, adapting the strategy of \cite{Szekelyhidi}, new examples of Calabi-Yau metrics were obtained  in \cite{Firester} by smoothing $\bbC\times C$ when $C$ is a complete intersection  with smooth cross-section equipped with a Calabi-Yau cone metric $g_C$.  

On $\bbC^n$, motivated by the conjecture of Tian \cite{Tian2006}, one could hope that a complete Calabi-Yau metric of maximal volume growth is completely determined up to scale and isometry by its tangent cone at infinity.  When $C=\bbC\times A_1$ with $A_1$ the $(n-1)$-dimensional Stenzel cone with Calabi-Yau cone metric $g_{A_1}$, Sz\'{e}kelyhidi \cite{Szekelyhidi2} showed that this is indeed the case.  However, already on $\bbC^3$, if one takes instead $C=\bbC\times A_2$ with $A_2$ the singular hypersurface 
$$
  \{ z_1^2+z_2^2+z_3^3=0\} \subset \bbC^3,
$$
the Calabi-Yau metric with tangent cone at infinity $\bbC\times A_2$ constructed in \cite{Szekelyhidi} is not unique.  Indeed, a $1$-parameter family of distinct such metrics was recently constructed by Chiu \cite{Chiu}.

In the present paper, we generalize the approach of \cite{CR2021} to construct new examples of complete Calabi-Yau metrics of maximal volume growth having a tangent cone at infinity with singular cross-section.    To state our results, consider $N$ Calabi-Yau cones $(W_1,g_{1}), \ldots, (W_N,g_N)$ with singular apex (\ie not corresponding to the Euclidean space) but with smooth cross-sections.  For each $q\in \{1,\ldots,N\}$, suppose that $W_q$ is a complete intersection in $\bbC^{m_q+n_q}$,
$$
        W_q= \{z_q\in \bbC^{m_q+n_q} \; | \; P_{q,1}(z_q)=\cdots = P_{q,n_q}(z_q)=0\},
$$
for $n_q$ polynomials $P_{q,1},\ldots, P_{q,n_q}$, where $m_q=\dim_{\bbC}W_q$.  Suppose that the natural $\bbR^+$-action on $W_q$ is induced by a diagonal action 
$$
      \begin{array}{lccl}
    \bbR^+\ni  t: & \bbC^{m_q+n_q} & \to & \bbC^{m_q+n_q}  \\
          & z_q  & \mapsto & t\cdot z_q= (t^{w_{q,1}}z_{q,1},\cdots,t^{w_{q,m_q+n_q}}z_{q,m_q+n_q})
      
      \end{array}
$$
for some positive weights $w_{q,1},\ldots, w_{q,m_q+n_q}$.  We will not assume that the cone is quasi-regular, so these weights are not necessarily rational.  We will assume  that each polynomial $P_{q,i}$ is homogeneous of some degree $d_{q,i}$ with respect to the $\bbR^+$-action,
$$
          P_{q,i}(t\cdot z_q)= t^{d_{q,i}} P_{q,i}(z) \quad \forall q\in \{1,\ldots,N\}, \; \forall i\in \{1,\ldots,n_q\},
$$
and that 
$$
    d_{q,1}\le\cdots\le d_{q,n_q}.
$$
Furthermore, we will assume that there is a $d>1$ such that for each $q\in\{1,\ldots,N\}$, there exists $k_q\in\{1,\ldots,n_q\}$ such that $d_{q,1}=\cdots= d_{q,k_q}=d$.  

On $W_q$, the K\"ahler form of the metric $g_{q}$ is given by $\omega_{q}= \frac{\sqrt{-1}}{2}\pa \db r_q^{2}$ with $r_q$ the radial distance to the origin with respect to the Calabi-Yau metric $g_{q}$.
On $W_q$, there is a holomorphic volume form $\Omega^{m_{q}}_q$ defined implicitly by 
\begin{equation}
  dz_{q,1}\wedge \cdots \wedge dz_{q,m_q+n_q}|_{W_q}= \Omega_q^{m_{q}}\wedge dP_{q,1}|_{W_q}\wedge\cdots \wedge dP_{q,n_q}|_{W_q}.
\label{cyc.2}\end{equation}
The fact that $(W_q,g_{q})$ is Calabi-Yau means that there is a constant $c_{m_q}\in\bbC\setminus \{0\}$ depending only on $m_q$ such that
\begin{equation}
  \omega_{q}^{m_q}= c_{m_q}\Omega^{m_q}_{q}\wedge \overline{\Omega}^{m_q}_{q}.
\label{cyc.3}\end{equation}
By \eqref{cyc.2}, the holomorphic volume form $\Omega^{m_q}_{q}$ is homogeneous of degree $\sum_{j=1}^{m_q+n_q}w_{q,j}-\sum_{j=1}^{n_q}d_{q,j}$ with respect to the $\bbR^+$-action, while $\omega_{q}$ is of degree $2$, so we deduce from \eqref{cyc.3} that
\begin{equation}
  m_q = \lrp{\sum_{j=1}^{m_q+n_q} w_{q,j}} -\sum_{j=1}^{n_q} d_{q,j}.
\label{cyc.4}\end{equation}
Using the convention that $m_0=1$ and $n_0=0$, consider the Calabi-Yau cone $W_0:=\bbC^{m_0+n_0}=\bbC$ with canonical Euclidean metric $g_0$, radial function $r_0=|z_0|$ for $z_0\in\bbC$ and holomorphic volume form $\Omega_0^{m_0}=dz_{0}$.  The corresponding $\bbR^+$-action is the standard one, so the weight is just $w_{0}=1$.  We will also be interested in the case $m_0=0$, in which case we will set $W_0=\{0\}$. 

We will consider the 
Cartesian product
\begin{equation}
     C_0:= W_0\times W_1\times\cdots\times W_q \subset \bbC^{m_0+n_0}\times \cdots\times \bbC^{m_q+n_q}=\bbC^{m+n}, \quad m:=\sum_{q=0}^N m_q, \; n:=\sum_{q=0}^N n_q,
\label{pr.1}\end{equation}
with product Calabi-Yau metric $g_{C_0}= g_0\times \cdots g_N$.  Thus, $(C_0,g_{C_0})$ is a Calabi-Yau cone with radial function
\begin{equation} 
     r:= \sqrt{\sum_{q=0}^{N} r_q^2},
\label{pr.2}\end{equation}  
 K\"ahler form 
\begin{equation}
    \omega_{C_0}=\frac{\sqrt{-1}}2\pa\db r^2
\label{pr.3}\end{equation}
and holomorphic volume form
\begin{equation}
 \Omega^m_{C_0}= (-1)^{\sum_{q=2}^Nm_q(n_{q-1}+\cdots+n_1)}\Omega^{m_0}_0\wedge \Omega^{m_1}_1\wedge\ldots\wedge \Omega^{m_N}_N.
\label{pr.3b}\end{equation}
The possible sign in the right hand side of \eqref{pr.3b} is to ensure that this definition agrees with the implicit definition
\begin{multline}
dz_{0}|_{C_0}\wedge (dz_{1,1}\wedge\ldots\wedge dz_{1,m_1+n_1})|_{C_0}\wedge\ldots\wedge(dz_{N,1}\wedge\ldots\wedge dz_{N,m_N+n_N})|_{C_0}=  \\\Omega^m_{C_0}\wedge (dP_{1,1}\wedge \ldots dP_{1,n_1})|_{C_0}\wedge\ldots\wedge (dP_{N,1}\wedge \ldots\wedge dP_{N,n_N})|_{C_0},
\label{pr.3c}\end{multline}
where we use the convention that $dz_{0}|_{C_0}=1$ if $W_0=\{0\}$.
Since $W_q$ is singular at the origin for $q>0$, notice that $(C_0,g_0)$ has a singular cross-section.   In fact, for each subset $\{0\}\subsetneq \mathfrak{q}\subset \{0,1\ldots,N\}$, 
\begin{equation}
    W_{\mathfrak{q}}:= \{ z=(z_0,z_1,\ldots,z_q)\in C_0\; | \; z_q=0 \; \mbox{for} \; q\notin \mathfrak{q}\}
\label{pr.4}\end{equation}
consists of singular points of $C_0$ with the singular locus of $C_0$ given by
\begin{equation}
      C_{0,\sing}= \bigcup_{\{0\}\subset \mathfrak{q}\subsetneq \{0,\ldots,N\}} W_{\mathfrak{q}}.
\label{pr.5}\end{equation}
Moreover, for each such $\mathfrak{q}$, 
\begin{equation}
    C_0= W_{\mathfrak{q}}\times W_{\mathfrak{q}}^{\perp}
\label{pr.6}\end{equation}
with 
\begin{equation}
    W_{\mathfrak{q}}^{\perp}= W_{\mathfrak{q}^c}= \{ z=(z_0,z_1,\ldots,z_q)\in C_0\; | \; z_q=0 \; \mbox{for} \; q\in \mathfrak{q}\},
\label{pr.7}\end{equation}
where $\mathfrak{q}^c= \{0,1,\ldots,N\}\setminus \mathfrak{q}$.  The decomposition \eqref{pr.6} is also Riemannian, namely the restriction of the Calabi-Yau metric $g_{C_0}$ to $W_{\mathfrak{q}}$ and $W_{\mathfrak{q}^c}$ induces Calabi-Yau cone metrics $g_{\mathfrak{q}}$ and $g_{\mathfrak{q}^c}$ such that $g_{C_0}$ is just the Cartesian product of $g_{\mathfrak{q}}$ and $g_{\mathfrak{q}^c}$.  Let
\begin{equation}
V_{\mathfrak{q}}:= \{ z=(z_0,z_1,\ldots,z_q)\in \bbC^{m+n}\; | \; z_q=0 \; \mbox{for} \; q\notin \mathfrak{q}\},
\label{pr.8}\end{equation}
\begin{equation}
V_{\mathfrak{q}}^{\perp}= V_{\mathfrak{q}^c}= \{ z=(z_0,z_1,\ldots,z_q)\in \bbC^{m+n}\; | \; z_q=0 \; \mbox{for} \; q\in \mathfrak{q}\},
\label{pr.9}\end{equation}
be the subspaces such that $W_{\mathfrak{q}}=C_0\cap V_{\mathfrak{q}}$ and $W_{\mathfrak{q}^c}=C_0\cap V_{\mathfrak{q}^c}$.

We will suppose that for some $\epsilon\in\bbC\setminus\{0\}$, the cone $C_0$ admits a smoothing $C_{\epsilon}$ of the form
\begin{equation}
C_{\epsilon}= \{ z=(z_0,\ldots,z_N)\in\bbC^{m+n}\; | \; P_{q,i}(z_q)=\epsilon Q_{q,i}(z_0)\quad  \mbox{for} \quad q\in\{1,\ldots,N\}\; \mbox{and} \; i\in \{1,\ldots,n_q\}\},
\label{cyc.6}\end{equation} 
where each $Q_{q,i}$ is a polynomial in $z_0\in\bbC^{m_0+n_0}=\bbC$ of (weighted) degree $\ell<d$ for some fixed $\ell\ge 0$ not depending on $q$ and $i$.  If  $\ell=0$, we will assume that $W_0=\{0\}$. It comes with a natural holomorphic volume form $\Omega^m_{C_{\epsilon}}$ defined explicitly by
\begin{multline}
dz_{0}|_{C_{\epsilon}}\wedge(dz_{1,1}\wedge\ldots\wedge dz_{1,m_1+n_1})|_{C_{\epsilon}}\wedge\ldots\wedge(dz_{N,1}\wedge\ldots\wedge dz_{N,m_N+n_N})|_{C_{\epsilon}}=  \\\Omega^m_{C_{\epsilon}}\wedge (d(P_{1,1}-\epsilon Q_{1,1})\wedge \ldots d(P_{1,n_1}-\epsilon Q_{1,n_1}))|_{C_\epsilon}\wedge\ldots\wedge (d(P_{N,1}-\epsilon Q_{N,1})\wedge \ldots\wedge d(P_{N,n_N}-\epsilon Q_{N,n_N}))|_{C_\epsilon},
\label{pr.3d}\end{multline}
where again we use the convention that $dz_0|_{C_{\epsilon}}=1$ if $W_0=\{0\}$.  We will make the following two assumptions on the polynomials $Q_{q,i}$.

\begin{assumption} The polynomials $Q_{q,i}$ will not be assumed to be homogeneous, but we will assume that the homogenous part of degree $\ell$, denoted $[Q_{q,i}]$, is non-zero for each $q$ and $i$.  Moreover, if $\ell>0$ and $N>1$, then we will assume that $m_0=1$, so that the zero locus $[Q_{q,i}](z_0)=0$ is the origin in $\bbC$ and corresponds to the hyperplane
$$
      V_{\{1,\ldots,N\}}= \{0\}\times \bbC^{m_1+n_1}\times\cdots \times \bbC^{m_N+n_N}\subset \bbC^{m+n}
$$ 
in $\bbC^{m+n}$.
\label{smooth.1}\end{assumption}

\begin{assumption} For the same $\epsilon\in \bbC\setminus\{0\}$ such that \eqref{cyc.6} is smooth, we suppose that
for each subset $\{0\}\subset \mathfrak{q}\subset \{0,1,\ldots,N\}$,
\begin{equation}
        W_{\mathfrak{q},\epsilon}:= \{ z_{\mathfrak{q}}\in V_{\mathfrak{q}}\; | \; P_{q,i}(z_q)=\epsilon Q_{q,i}(z_0) \quad \mbox{for} \; q\in \mathfrak{q}\setminus\{0\}, \; i\in \{1,\ldots,n_q\}\}
\label{pr.10}\end{equation}
is a smoothing of $W_{\mathfrak{q}}$.  If $m_0=1$ and $W_0=\bbC$, then for all $\omega_{\mathfrak{q}}\in W_{\mathfrak{q}}\cap\bbS(V_{\mathfrak{q}})$ with $\bbS(V_{\mathfrak{q}})$ the unit sphere in $V_{\mathfrak{q}}$, we will also suppose that 
\begin{multline}
 W^{\perp}_{\mathfrak{q},\omega_{\mathfrak{q}},\epsilon}= W_{\mathfrak{q}^c,\omega_{\mathfrak{q}},\epsilon}= \left\{ z_{\mathfrak{q}^c}\in V_{\mathfrak{q}^c}\; | \; \forall q\in \mathfrak{q}^c, \; i \in\{1,\ldots,n_q\}, \right. \\  \left. P_{q,i}(z_{q})=\epsilon[Q_{q,i}](\omega_{0}), \; i\le k_q,  \quad P_{q,i}(z_{q})=0, \;i> k_q \right\}
\label{pr.11}\end{multline}
is a smoothing of $W^{\perp}_{\mathfrak{q}}= W_{\mathfrak{q}^c}$ for $\omega_0\ne 0$, where $\omega_0$ is the component of $\omega_{\mathfrak{q}}$ in $V_0$.
\label{cyc.6c}\end{assumption}
\begin{remark}
The smoothing \eqref{pr.11} is a Cartesian product of the smoothings $W_{q,\omega_{\mathfrak{q}},\epsilon}$ of $W_q$ for $q\in\mathfrak{q}^c$.
\label{pr.11b}\end{remark}

We can now state the main result of this paper; see also Corollaries~\ref{mae.15} and \ref{mae.16} below for more details.
\begin{theorem}
Suppose that Assumptions~\ref{smooth.1} and \ref{cyc.6c} hold.  If $\nu:= \frac{\ell}{d}$ and $\beta:=\min\{d,2m_1,\ldots,2m_N\}$ are such that
\begin{equation}
   \beta> \frac{2}{1-\nu},
\label{int.1a}\end{equation}
then $C_{\epsilon}$ admits a complete Calabi-Yau metric of maximal volume growth with tangent cone at infinity $(C_0,g_{C_0})$.  Furthermore, if $N=1$, this result still holds if instead of \eqref{int.1a}, $\nu$ and $\beta$ are such that
\begin{equation}
   3<\beta\le \frac{2}{1-\nu}<2m_1+5.  
\label{int.1b}\end{equation}
\label{int.1}\end{theorem}
\begin{remark}
We have not explored it here, but instead of \cite{DM2014}, one could try to develop a higher depth version of the approach of \cite{Szekelyhidi}, see also \cite{Firester, Yan}, to obtain mapping properties of the Laplacian.  The approach of \cite{Szekelyhidi} allows in principle to consider weights for which the Laplacian is no longer an isomorphism, but is still surjective and Fredholm, which would possibly allow condition \eqref{int.1a} in Theorem~\ref{int.1} to be weakened.
\label{int.7}\end{remark}

When $N=1$, $m_0=1$, $n_1=1$ and $\ell=1$ with $Q_{1,1}(z_{0,1})=z_{0,1}$, this theorem corresponds to most of the examples obtained in \cite{YangLi, CR2021, Szekelyhidi}, while most of the examples of \cite{Firester} corresponds to the case $N=1$, $m_0=1$ and $\ell=1$ with $Q_{1,i}$ homogeneous for each $i\in\{1,\ldots, n_1\}$.  Allowing other values of these parameters yields many new examples of Calabi-Yau metrics.  Let us illustrate this with few examples.  
\begin{example}
Letting $N=1$, $m_0=1$ and $n_1=1$ as in \cite{YangLi, CR2021, Szekelyhidi}, we can obtain new examples of Calabi-Yau metrics by taking $\ell\ne 1$ or $Q_{1,1}$ not homogeneous.  For instance, we can take $m_1=2$ with 
$$
      P_{1,1}(z_{1,1},z_{1,2},z_{1,3})= z_{1,1}^2+ z_{1,2}^2+ z_{1,3}^2 \quad \mbox{and} \quad Q_{1,1}(z_{0,1})=-z_{0,1}^{\ell}+1,
$$
so that $C_{\epsilon}\subset \bbC^4$ is the affine hypersurface given by
$$
   z_{1,1}^2+ z_{1,2}^2+ z_{1,3}^2 + \epsilon z_{0,1}^{\ell}=\epsilon 
$$
with $\ell=2$ or $\ell=3$.  In this case, we see from \eqref{cyc.4} that $w_{1,i}=2$ for all $i$, so $d=4$, $\ell\in \{2,3\}$ as just mentioned, $\nu=\frac{\ell}{4}$ and $\beta=4$.  In particular, \eqref{int.1b} holds, so Theorem~\ref{int.1} yields a Calabi-Yau metric on $C_{\epsilon}$ with tangent cone at infinity $\bbC\times A_1$ with $A_1$ the Stenzel cone $z_{1,1}^2+ z_{1,2}^2+ z_{1,3}^2=0$.  When $\ell=2$, notice that $C_{\epsilon}$ is also a smoothing of the Stenzel cone 
\begin{equation}
z_{1,1}^2+ z_{1,2}^2+ z_{1,3}^2 + \epsilon z_{0,1}^{2}=0,
\label{int.2a}\end{equation}
so as such, it also admits an asymptotically conical Calabi-Yau metric with tangent cone at infinity \eqref{int.2a}.  When $\ell=3$, $C_{\epsilon}$ is a smoothing of the $A_2$ singularity
  $$
  z_{1,1}^2+ z_{1,2}^2+ z_{1,3}^2 + \epsilon z_{0,1}^{3}=0.
  $$
By \cite[Corollary~1.6]{LiSun}, the $3$-dimensional $A_2$ singularity admits a Calabi-Yau cone metric, but for a different choice of weights.  
\label{int.2}\end{example}  
\begin{example}
Letting $N=1$, $m_0=1$ and $n_1=1$ as in the previous example, we can also take $m_1\ge 3$ and 
\begin{equation}
   P_{1,1}(z_q)= z_{1,1}^k+\cdots + z_{1,m_1+1}^k
\label{int.2c}\end{equation}
for some $2\le k\le m_1$.  It is well-known in this case, see for instance \cite[Example~2.1]{CR2021}, that $W_q$ admits a Calabi-Yau cone metric.  From \eqref{cyc.4}, we see that
\begin{equation}
      w_{1,i}= \frac{m_1}{m_1+1-k} \quad \mbox{and} \quad d= \frac{km_1}{m_1+1-k}.
\label{int.2d}\end{equation}
Assuming $k\ge 3$ ensures that $d>3$.  Taking $\ell=k$, then 
$$
      \frac{2}{1-\nu}= \frac{2m_1}{k-1}< 2m_1+5 \quad \mbox{and} \quad  \beta>3,
$$
so either \eqref{int.1a} or \eqref{int.1b} is satisfied.  To satisfy Assumption~\ref{cyc.6c}, we can take 
\begin{equation}
   Q_{1,1}(z_{0,1})=-z_{0,1}^k+1,
\label{int.2ee}\end{equation}
so that Theorem~\ref{int.1} yields a complete Calabi-Yau metric on $C_{\epsilon}$ with tangent cone at infinity $C_0=\bbC\times W_1$.  Notice that $C_{\epsilon}$ is also a smoothing of the affine hypersurface given by
\begin{equation}
 z_{1,1}^k+\cdots + z_{1,m_1+1}^k+\epsilon z_{0,1}^k=0,
\label{int.2e}\end{equation}
 which admits a Calabi-Yau cone metric.  As such, $C_{\epsilon}$ also admits an asymptotically conical Calabi-Yau metric with tangent cone at infinity \eqref{int.2e}.
\label{int.2b}\end{example}

\begin{example}
Within the framework of Examples~\ref{int.2} and \ref{int.2b}, our method also yields in any compactly supported K\"ahler class a complete Calabi-Yau metric with tangent cone at infinity $C_0=\bbC\times W_1$ on the crepant resolution $K_{F_k}$ of \eqref{int.2e}, where 
\begin{equation}
  F_k=\{ [z_{1,1}:\ldots : z_{1,m_1+1}: z_{0,1}]\in \bbC\bbP^{m_1+1}\; |\; z_{1,1}^k+\cdots + z_{1,m_1+1}^k+\epsilon z_{0,1}^k=0\}
\label{cr.1a}\end{equation}
is the corresponding Fermat hypersurface.  Indeed, the only minor modification needed for Theorem~\ref{int.1} to apply in this setting concerns the construction of the model metric.  This is explained in Remark~\ref{cr.2} below.
 \label{cr.1}\end{example} 
\begin{example}
Taking $N>1$, we can obtain Calabi-Yau metrics with tangent cone at infinity $(C_0,g_{C_0})$ in equation \eqref{pr.1} whose cross-section has singularities of depth $N$.  For instance, we can take $m_0=1$ with $n_1=\cdots=n_N=1$ and $m_1=\cdots=m_N=m$ with 
\begin{equation}
     P_{q,1}(z_q)= z_{q,1}^k+\cdots + z_{q,m+1}^k \quad \forall q\in \{1,\ldots,N\},
\label{int.4a}\end{equation} 
for some $2\le k\le m$ so that $W_q$ admits a Calabi-Yau cone metric.  From \eqref{cyc.4}, we see that
\begin{equation}
      w_{q,i}= \frac{m}{m+1-k} 
\label{int.4b}\end{equation}
for $q>0$, so $d=\frac{km}{m+1-k}>k$.  Assuming $k>3$, then $d>3$ and to satisfy \eqref{int.1a}, it suffices to take $\ell<d-2$.  To satisfy Assumption~\ref{cyc.6c}, we can take for instance $Q_{q,i}(z_0)=z_0$ for each $q$.
\label{int.4}\end{example}

\begin{remark}
When $N=1$, it would have been possible in principle to take $m_0>1$ in Theorem~\ref{int.1}.  One interesting special case would be to take $N=1$, $\ell=1$ and $n_1>1$ as in \cite{Firester}, but with $k_1=n_1=m_0>1$ and $Q_{1,i}(z_0)=z_{0,i}$ for $i\in \{1,\ldots, n_1\}$, so that $C_{\epsilon}$ corresponds to the smoothing
\begin{equation}
   P_{1,i}(z_1)=\epsilon z_{0,i}\quad \mbox{for} \quad i\in\{1,\ldots,n_1\}.
\label{int.3a}\end{equation}
Since it is the graph of a holomorphic map from $\bbC^{m_1+n_1}$ to $\bbC^{m_0}=\bbC^{n_1}$, the smoothing $C_{\epsilon}$ is in particular biholomorphic to $\bbC^{m_1+n_1}$.  Unfortunately however, the problem is that  condition \eqref{pr.11} of Assumption~\ref{cyc.6c} is never satisfied in this case, so we cannot actually apply Theorem~\ref{int.1} to obtain new examples of Calabi-Yau metrics.  Indeed, the singular locus $\cS$ of the map $F:=(P_{1,1},\ldots, P_{1,n_1})$ is such that
$$
     \dim_{\bbC} \cS\ge m_1+n_1 -(m_1+1)=n_1-1>0
$$
as a subvariety of $\bbC^{m_1+n_1}$. Moreover,  it contains the origin and is also clearly invariant under the $\bbR^+$-action.  Hence, taking $s\in \cS\setminus\{0\}$ close to the origin, then we can take $\omega_0=(P_{1,1}(s),\ldots,P_{1,n_1}(s))$ in \eqref{pr.11} to obtain a singular variety.  Since $W_1$ is assumed to be smooth outside the origin, notice that $\omega_0\ne 0$ automatically. 
\label{int.3}\end{remark}
\begin{remark}
When $\ell=0$ and $m_0=0$, the Calabi-Yau metrics of Theorem~\ref{int.1} are QAC-metrics.  In fact, they are more precisely a Cartesian product of asymptotically conical Calabi-Yau metrics (by \cite[Theorem~6.6]{KR3}, the class of $\QAC$-metrics is closed under taking the Cartesian product), so this does not yield new examples of Calabi-Yau metrics.  
\label{int.3b}\end{remark}

As suggested by Yang Li in \cite[p.41]{YangLi2018}, we can use the metrics in Examples~\ref{int.2} and \ref{int.3} to construct examples of singular Calabi-Yau metrics making a transition between two distinct Calabi-Yau cones.  To do this, we will, instead of the smoothing induced by \eqref{int.2ee}, consider the slightly modified one given by
\begin{equation}
  z_{1,1}^k+\cdots+ z_{1,m_1+1}^k+ z_{0,1}^k=\epsilon
\label{yl.1}\end{equation}
and denote by $\cH_{\epsilon}$ the induced affine hypersurface.  In the limit $\epsilon\searrow 0$, this yields the singular affine hypersurface $\cH_0$ with defining equation
\begin{equation}
  z_{1,1}^k+\cdots+ z_{1,m_1+1}^k+ z_{0,1}^k=0.
\label{yl.2}\end{equation}
As mentioned earlier, it admits a Calabi-Yau cone metric that we will denote by $g_{\cH_0}$.  Now, let $g_{\epsilon}$ be the complete Calabi-Yau metric on $\cH_{\epsilon}$ given by Theorem~\ref{int.1}.  Taking the limit $\epsilon\searrow 0$, we obtain the following result; see Theorem~\ref{sing.37} below for more details.
\begin{theorem}
Take $k=2$, $m_1=2$ and $\ell=2$ as in Example~\ref{int.2}, or else $k$ and $m_1$ as in Example~\ref{int.2b}.  Then as $\epsilon\searrow 0$, $(\cH_{\epsilon},g_{\epsilon})$ converges  smoothly locally away from the origin to a Calabi-Yau metric $g_0$ on $\cH_0\setminus\{0\}$ which has a conical singularity modelled asymptotically on $(\cH_0,g_{\cH_0})$ in the sense of Definition~\ref{sing.23} below, and  elsewhere is smooth and complete with tangent cone at infinity $(C_0,g_{C_0})$.  Moreover, $(\cH_\epsilon, g_\epsilon)$ also converges to $(\cH_0,g_{\cH_0})$ in the pointed Gromov-Hausdorff sense.
\label{yl.3}\end{theorem}

For $\epsilon\ne 0$, notice that $\cH_{\epsilon}$ is biholomorphic to $\cH_1$.  On the other hand, by Theorem~\ref{yl.3}, the metrics $g_{\epsilon}$ are not all isometric since they  converge in the pointed Gromov-Hausdorff sense to the singular metric $g_{0}$.  Hence, we deduce the following non-uniqueness result.  
\begin{corollary}
For $\cH_1$ as in Theorem~\ref{yl.3}, there exist infinitely many complete Calabi-Yau metrics with tangent cone at infinity $(C_0,g_{C_0})$ that cannot be related by  scaling and pullback by biholomorphism.
\label{ucy.1}\end{corollary} 
\begin{remark}
Since our proof of Corollary~\ref{ucy.1} in \S~\ref{sing.0} is by contradiction, in contrast to \cite{Chiu}, we cannot rule out the possibility that some metrics in the family be related by scaling and pull-back by biholomorphism.
\end{remark}

Since the Calabi-Yau metrics of Theorem~\ref{int.1} are obtained by solving a complex Monge-Amp\`ere equation, they are not given by an explicit formula.  However, as in \cite{CR2021}, we provide a detailed description of their asymptotic behavior at infinity.  In fact, the Calabi-Yau metrics of Theorem~\ref{int.1} are a higher depth version of the warped $\QAC$-metrics of \cite{CR2021}.  Indeed, the local models at infinity are again a warped version of the local model for $\QAC$-metrics, namely a warped product of a cone metric $dr^2+ r^2g_B$ and a warped $\QAC$-metric of lower depth $K_w$, that is
\begin{equation}
dr^2+ r^2g_B+ r^{2\nu_K} K_w
\label{int.5}\end{equation}
with $\nu_K\in \{0,\nu\}$.

Compared to \cite{CR2021}, a new feature of this higher depth version is that warped $\QAC$-metrics are no longer necessarily conformal to a $\QAC$-metric.  They are in general only conformal to a slightly different class of metrics introduced in Definition~\ref{wqb.8} below that we call weighted $\QAC$-metrics.  As for $\QAC$-metrics, those come from a Lie structure at infinity in the sense of \cite{ALN04}, so admit a nice global coordinate free description.  They are in particular automatically complete of infinite volume with bounded geometry and the same holds true for warped $\QAC$-metrics.  There is also a conformally related class of metrics introduced in Definition~\ref{wqb.12} below  that we call weighted $\Qb$-metrics.  They play a similar role to $\Qb$-metrics in \cite{CDR2016} by allowing us to define the right weighted H\"older spaces in which to solve the complex Monge-Amp\`ere equation.  In terms of these H\"older spaces, we can in fact as in \cite{CR2021} adapt the arguments of \cite{DM2014} to derive nice mapping properties of the Laplacian of a warped $\QAC$-metric; see in particular Corollary~\ref{w.53} below, as well as Theorem~\ref{sing.35} for the singular setting of Theorem~\ref{yl.3}.

All these classes of metrics are defined in terms of a compactification by a manifold with corners, so to define them on $C_{\epsilon}$, we need to introduce a suitable compactification $\hC_{\epsilon}$ of $C_{\epsilon}$.  As in \cite{CR2021}, our starting point is the radial compactification $\overline{\bbC^{m+n}_w}$ of $\bbC^{m+n}$ specified by the $\bbR^+$-action induced by the weights $w_{q,i}$.  Let  $\overline{C}_{\epsilon}$ be the closure of $C_{\epsilon}$ in $\overline{\bbC^{m+n}_w}$.  This closure has singularities on the boundary $\pa \overline{\bbC^{m+n}_w}$ of $ \overline{\bbC^{m+n}_w}$ that can be resolved by blowing them up.  However, this cannot be achieved with the usual notion of blow-up \cite[\S~2.2]{HHM1995} of a $p$-submanifold in a manifold with corners.  We need instead to consider a weighted version of it, one special instance being the parabolic blow-up of Melrose \cite[(7.4)]{MelroseAPS}.  Compared to the usual blow-up, this weighted version is sensitive to the choice of a tubular neighborhood of the $p$-submanifold, but in our setting, the $p$-submanifolds that need to be blown up in a weighted manner come with a natural choice of tubular neighborhood, ensuring that each weighted blow-up is well-defined.  

Once the natural compactification $\hC_{\epsilon}$ is given, each boundary hypersurface gives a model at infinity, essentially \eqref{pr.10} and \eqref{pr.11}.  Proceeding inductively on the depth of $\hC_{\epsilon}$ and using a convexity argument of \cite[Lemma~4.3]{vanC}, we can first construct K\"ahler examples of warped $\QAC$-metrics on $C_{\epsilon}$.  To obtain Calabi-Yau examples, we need to solve a complex Monge-Amp\`ere equation.  As in \cite{CDR2016}, proceeding by induction on the depth of $\hC_{\epsilon}$, we can first solve the complex Monge-Amp\`ere equation on each model at infinity.  Using the fixed point argument of \cite[Proposition~25]{Szekelyhidi}, we can then solve the complex Monge-Amp\`ere equation outside a large compact set.  From that point, we can then completely solve the complex Monge-Amp\`ere equation using standard techniques.  

To construct the limiting metric in Theorem~\ref{yl.3}, we adapt recent results of Collins-Guo-Tong \cite{CGT} and Sun-Zhang \cite{Sun-Zhang} to first show that away from the origin, the metrics $\{g_{\epsilon}\}$ converges smoothly locally to a Calabi-Yau metric $g_0$ on $\cH_0\setminus\{0\}$ having the expected tangent cone at infinity.  To show that $g_0$ has a conical singularity at the origin asymptotically modelled on $(\cH_0,g_{\cH_0})$, we apply the continuity method of Hein-Sun \cite{Hein-Sun}.  For openness, this requires a suitable isomorphism provided by Theorem~\ref{sing.35} below, while for closedness, this requires a few adjustments essentially provided by \cite{CGT} and \cite{Sun-Zhang}.

The paper is organized as follows.  in \S~\ref{wqac.0}, we introduce the notion of warped $\QAC$-metrics, derive their main properties and introduce the weighted H\"older spaces that we will need.   In \S~\ref{mpl.0}, following \cite{DM2014}, we derive the mapping properties of the Laplacian of a warped $\QAC$-metric that we will need.  In \S~\ref{wbu.0}, we describe the notion of weighted blow-up for manifolds with corners and derive some of its features.  In \S~\ref{sCY.0}, we introduce the compactification $\hC_{\epsilon}$ and construct examples of K\"ahler warped $\QAC$-metrics on $C_{\epsilon}$.  In \S~\ref{mae.0}, we solve a complex Monge-Amp\`ere equation to obtain Theorem~\ref{int.1}.  Finally, \S~\ref{sing.0} gives a proof of Theorem~\ref{yl.3} and Corollary~\ref{ucy.1}. 

\begin{Acknowledgement}
The authors are grateful to Shih-Kai Chiu, Tristan Collins, Benjy Firester and Freid Tong for helpful discussions, as well as to an anonymous referee for useful comments. The first author is supported by NSF grant DMS-1906466 and the second author is supported by NSERC and FRQNT.
\end{Acknowledgement}

\section{Warped quasi-asymptotically conical metrics} \label{wqac.0}

Since it will be a central tool in this paper, let us first recall the notion of manifold with fibered corners.  Thus, let $M$ be a manifold with corners. Unless otherwise stated, we will usually assume that $M$ is compact. Denote by $\cM_1(M)$ the set of boundary hypersurfaces of $M$, that is, the set of corners of codimension $1$.  As in \cite{Melrose1992}, we will assume that its boundary hypersurfaces are all embedded.    We denote by $\pa M$ the union of all the boundary hypersurfaces of $M$.   Suppose that each boundary hypersurface $H$ of $M$ is endowed with a fiber bundle $\phi_H: H\to S_H$ with base $S_H$ and fibers also manifolds with corners.  Denote by $\phi$ the collection of these fiber bundles.
\begin{definition}[\cite{AM2011,ALMP2012,DLR}]
We say that $(M,\phi)$ is a \textbf{manifold with fibered corners} or that $\phi$ is an \textbf{iterated fibration structure for $M$} if there is a partial order on $\cM_1(M)$ such that
\begin{itemize}
\item Any subset $\cI\subset \cM_1(M)$ such that $\displaystyle \bigcap_{H\in\cI} H\ne \emptyset$ is totally ordered;
\item If $H<G$, then $H\cap G\ne \emptyset$, the map $\phi_H|_{H\cap G}: H\cap G \to S_H$ is a surjective submersion, $S_{GH}:= \phi_G(H\cap G)$ is one of the boundary hypersurfaces of $S_G$ and there is a surjective submersion $\phi_{GH}: S_{GH}\to S_H$ such that $\phi_{GH}\circ \phi_G=\phi_H$ on $H\cap G$;
\item The boundary hypersurfaces of $S_G$ are given by $S_{GH}$ for $H<G$.
\end{itemize}
\label{mwfc.1}\end{definition}

One can check directly from the definition that the base $S_H$ and the fibers of $\phi_H: H\to S_H$ are also naturally manifolds with fibered corners.  If $H$ is minimal with respect to the partial order, then $S_H$ is in fact a closed manifold.  Conversely, if $H$ is maximal with respect to the partial order, then the fibers of $\phi_H: H\to S_H$ are closed manifolds.  In various settings, this allows us to prove assertions by proceeding by induction on the \textbf{depth} of $(M,\phi)$, which is the largest codimension of a corner of $M$.  Manifolds with fibered corners are intimately related with stratified spaces.
\begin{definition}
A \textbf{stratified space} of dimension $n$ is a locally separable metrizable space $X$ together with a \textbf{stratification}, which is a locally finite partition $\mathcal{S}=\{s_i\}$ into locally closed subsets of $X$, called the \textbf{strata}, which are smooth manifolds of dimension $\dim s_i\le n$ such that at least one is of dimension $n$ and 
$$
    s_i\cap \overline{s}_j\ne \emptyset  \quad \Longrightarrow \quad s_i\subset \overline{s}_j.
$$ 
In this case, we write $s_i\le s_j$ and $s_i<s_j$ if $s_i\ne s_j$.  A stratification induces a filtration
$$
   \emptyset \subset X_0\subset \cdots \subset X_n= X,
$$
where $X_j$ is the union of all strata of dimension at most $j$.  The strata included in $X\setminus X_{n-1}$ are said to be \textbf{regular}, while those included in $X_{n-1}$ are said to be \textbf{singular}.
\label{strat.1}\end{definition}
Given a stratified space, notice that the closure of each of its strata is also naturally a stratified space.  The \textbf{depth} of a stratified space is the largest integer $k$ such that one can find $k+1$ different strata with
$$
      s_1<\cdots<s_{k+1}.
$$
As described in \cite{AM2011,ALMP2012,DLR}, a manifold with fibered corners $(M,\phi)$ arises as a resolution of the stratified space given by ${}^{S}M:= M/\sim$, where $\sim$ is the equivalence relation 
$$
    p\sim q \quad \Longleftrightarrow \quad p=q \quad \mbox{or \; $p,q\in H$ with $\phi_H(p)=\phi_H(q)$ for some $H\in\cM_1(M)$}.
$$
If $\beta: M\to {}^{S}M$ denotes the quotient map, which can be thought of as a blow-down map, then  $\beta(M\setminus \pa M)$ yields the regular strata.  In fact, the map $\beta$ gives a one-to-one correspondence between the boundary hypersurfaces of $M$ and the singular strata of ${}^{S}M$, namely $H\in \cM_1(M)$ corresponds to the stratum $s_H$ whose closure  is given by $\overline{s}_H:= \beta(H)$.  In this correspondence, the base $S_H$ of the fiber bundle $\phi_H$ is itself a resolution of $\overline{s}_H$ and $s_{H}= \beta(\phi_H^{-1}(S_H\setminus\pa S_H) )$.  Moreover, the depth of $(M,\phi)$ as a manifold with fibered corners is the same as the depth of ${}^{S}M$ and the partial order on $\cM_1(M)$ matches the one on the strata of ${}^{S}M$.  A stratified space admitting a resolution by a manifold with fibered corners is said to be \textbf{smoothly stratified}.  Not all stratified spaces are smoothly stratified, but as discussed in \cite{ALMP2012,DLR}, the property of being smoothly stratified can be described intrinsically on a stratified space without referring to a manifold with fibered corners. 

Recall from \cite{Melrose1992} that a \textbf{boundary defining function} for $H\in \cM_1(M)$ is a function $x_H\in \CI(M)$ such that $x_H\ge 0$, $H=x_{H}^{-1}(0)$ and $dx_H$ is nowhere zero on $H$.  Following \cite[Definition~1.9]{CDR2016}, we will say that a boundary defining function $x_H$ of $H$ is \textbf{compatible} with the iterated fibration structure if for each $G>H$, the restriction of $x_H$ to $G$ is constant along the fibers of $\phi_G: G\to S_G$.  By \cite[Lemma~1.4]{DLR}, compatible boundary defining functions always exist.  If $p\in \pa M$ is contained in the interior of a corner $H_1\cap\cdots \cap H_k$ of codimension $k$, then without loss of generality, we can assume that $H_1<\ldots<H_k$.  If $x_i$ is a choice of compatible boundary defining function for $H_i$, then by \cite[Lemma~1.10]{CDR2016}, in a neighborhood of $p$ where each fiber bundle $\phi_{H_i}: H_i\to S_{H_i}$ is trivial, we can consider tuples of functions $y_i=(y_i^1,\ldots, y_i^{k_i})$ and $z=(z_1,\ldots,z_q)$ such that
\begin{equation}
       (x_1,y_1,\ldots,x_k,y_k,z)
\label{coord.1}\end{equation}
provides coordinates near $p$ with the property that on $H_i$, $(x_1,y_1,\ldots, x_{i-1},y_{i-1},y_i)$ induces coordinates on the base $S_{H_i}$ with $\phi_{H_i}$ corresponding to the map
$$
 (x_1,y_1,\ldots,\widehat{x}_i,y_i,\ldots,x_k,y_k,z)\mapsto (x_1,y_1,\ldots, x_{i-1},y_{i-1},y_i),
$$
where the notation  `` $\widehat{\;\;}$ " above a variable denotes its omission.

To describe the type of metrics we want to consider on  a manifold with fibered corners $(M,\phi)$, let $\cA_{\phg}(M)$ denote the space of bounded continuous functions on $M$ that are smooth on $M\setminus \pa M$ and polyhomogeneous on $M$ for some index family.  We refer to \cite{Melrose1992} for more details on polyhomogeneous functions on a manifold with corners.
 \begin{definition}
A \textbf{weight function} for $(M,\phi)$ is a function
$$
\begin{array}{lccc}\mathfrak{n}:& \cM_1(M)&\to &\{0,\nu\} \\ & H & \mapsto & \nu_H
\end{array}
$$
for some $\nu\in [0,1)$ such that for all $G,H\in \cM_1(M)$, 
\begin{equation}
      H<G \quad \Longrightarrow \quad \nu_H\le \nu_G.
\label{wqb.1b}\end{equation}
For choices of compatible  boundary defining functions $x_H\in \CI(M)$ for each $H\in \cM_1(M)$, the corresponding \textbf{$\mathfrak{n}$-weighted distance} is the function
$$
     \rho:= \prod_{H\in\cM_1(M)} x_H^{-\frac{1}{1-\nu_H}}.
$$ 
Alternatively, we say that $\rho^{-1}$ is an \textbf{$\mathfrak{n}$-weighted total boundary defining function}.
\label{wqb.1}\end{definition}
A choice of $\mathfrak{n}$-weighted total boundary defining function specifies a Lie algebra of vector fields as follows.
\begin{definition}
Let $\rho^{-1}$ be an $\mathfrak{n}$-weighted total boundary defining function for $(M,\phi)$.  For such a choice, a \textbf{$\mathfrak{n}$-weighted quasi-fibered boundary vector field} ($\nQFB$-vector fields for short) is a $b$-vector field $\xi\in \CI(M;{}^bTM)$ such that
\begin{enumerate}
\item $\xi|_{H}$ is tangent to the fibers of $\phi_H:H\to S_H$ for each $H\in\cM_1(M)$; 
\item $\xi\rho^{-1}\in v\rho^{-1}\cA_{\phg}(M)$, where $v=\prod_{H\in \cM_1(M)}x_H$ is a total boundary defining function.
\end{enumerate} 
We denote by $\cV_{\nQFB}(M)$ the space of all $\nQFB$-vector fields.
\label{wqb.2}\end{definition}
\begin{remark}
When $\mathfrak{n}$ is the trivial weight function given by $\mathfrak{n}(H)=0$ for all $H\in \cM_1(M)$, $\cV_{\nQFB}(M)=\cV_{\QFB}(M)$ corresponds to the Lie algebra of $\QFB$-vector fields of \cite{CDR2016}. 
\label{wqb.2b}\end{remark}
The first condition in Definition~\ref{wqb.2} is clearly closed under taking the Lie bracket, while for the second one, it is closed thanks to the fact that for any $b$-vector field $\xi$, $\xi v\in v\CI(M)$.  Thus, $\nQFB$-vector fields indeed form a Lie subalgebra of the Lie algebra of $b$-vector fields.

The first condition means that $\xi$ is an edge vector field in the sense of \cite{MazzeoEdge, ALMP2012,AG}.  We denote by $\cV_e(M)$ the Lie algebra of edge vector fields.  In the local coordinates \eqref{coord.1}, it is locally generated over $\CI(M)$ by 
\begin{equation}
v_1\frac{\pa}{\pa x_1}, v_1\frac{\pa}{\pa y_1}, \ldots, v_k\frac{\pa}{\pa x_k}, v_k\frac{\pa}{\pa y_k}, \frac{\pa}{\pa z},
\label{wqb.3}\end{equation}
where $v_i:=\prod_{j=i}^k x_j$ and where $\frac{\pa}{\pa y_i}$ and $\frac{\pa}{\pa z}$ stand for 
$$
     \frac{\pa}{\pa y_i^1},\ldots, \frac{\pa}{\pa y_i^{\ell_i}} \quad \mbox{and} \quad \frac{\pa}{\pa z^1},\ldots, \frac{\pa}{\pa z^q} \quad \mbox{respectively}.
$$
Taking into account the second condition in Definition~\ref{wqb.2}, $\nQFB$-vector fields are locally generated over $\CI(M)$ by
\begin{multline}
v_1x_1\frac{\pa}{\pa x_1}, v_1\frac{\pa}{\pa y_1}, v_2\left( (1-\nu_2)x_2\frac{\pa}{\pa x_2}- (1-\nu_1)x_1\frac{\pa}{\pa x_1} \right), v_2\frac{\pa}{\pa y_2},\ldots,  \\ v_k\left( (1-\nu_k)x_k\frac{\pa}{\pa x_k}- (1-\nu_{k-1})x_{k-1}\frac{\pa}{\pa x_{k-1}} \right), v_k\frac{\pa}{\pa y_k}, \frac{\pa}{\pa z},
\label{wqb.4}\end{multline}
where $\nu_i:=\nu_{H_i}$.
Indeed, the natural modification $v_i x_i\frac{\pa}{\pa x_i}$ of the edge vector fields `normal' to the boundary hypersurfaces satisfy the second condition, but the difference $(1-\nu_i)x_i \frac{\pa}{\pa x_i}- (1-\nu_j)x_j\frac{\pa}{\pa x_j}$ even annihilates $\rho^{-1}$, while multiplied by $v_{\min(i,j)+1}$ makes it an edge vector field, so that \eqref{wqb.4} provides a generating set.

In a similar way that the Lie algebra of $\QFB$-vector fields depends on a choice of total boundary defining function \cite[Lemma~1.1]{KR1}, the Lie algebra of $\nQFB$-vector fields depends on a choice of $\mathfrak{n}$-weight total boundary defining function $\rho^{-1}$.  Two such functions will be said to be \textbf{$\nQFB$-equivalent} if they yield the same Lie algebra of vector fields.  The following generalization of \cite[Lemma~1.1]{KR1} gives a simple criterion for two $\mathfrak{n}$-weighted total boundary defining functions to be $\nQFB$-equivalent.  
\begin{lemma}
Two $\mathfrak{n}$-weighted total boundary defining functions $\rho^{-1}$ and $(\rho')^{-1}$ are $\nQFB$-equivalent provided the function 
$$
     f:= \log \lrp{\frac{\rho'}{\rho}} \in \cA_{\phg}(M)
$$
is such that near each boundary hypersurface $H\in \cM_1(M)$, $f=\phi_H^*f_H+ \mathcal{O}(x_H)$ for some $f_H\in\cA_{\phg}(S_H)$.
\label{wqb.5}\end{lemma}
\begin{proof}
Clearly, $\rho^{-1}$ and $(\rho')^{-1}$ are $\nQFB$-equivalent if and only if for all $\xi\in\cV_e(M)$, 
\begin{equation}
   \frac{d\rho}{\rho v}(\xi)\in \cA_{\phg}(M) \; \Longleftrightarrow \; \frac{d\rho'}{\rho'v}\in \cA_{\phg}(M),
\label{wqb.5b}\end{equation}
where $v= \prod_{H\in\cM_1(M)} x_H$ is a total boundary defining function for $M$.  Now, since $\rho'=e^f \rho$, 
\begin{equation}
     \frac{d\rho'}{\rho'v}=\frac{1}{v}\lrp{\frac{d\rho}{\rho}+df}.
\label{wqb.5c}\end{equation}
In particular, if near each $H\in \cM_1(M)$, $f= \phi_H^* f_H+ \cO(x_H)$ for some $f_H\in \cA_{\phg}(S_H)$, then \eqref{wqb.5b} holds for each $\xi\in\cV_e(M)$.  
\end{proof}

By the local description \eqref{wqb.4}, we see that $\cV_{\nQFB}(M)$ is a locally free sheaf of rank $m=\dim M$ over $\CI(M)$.  Thus, by the Serre-Swan theorem, there is a natural smooth vector bundle, the \textbf{$\mathfrak{n}$-weighted $\QFB$-tangent bundle} ($\nQFB$-tangent bundle for short), denoted ${}^{\mathfrak{n}}TM$, and a natural bundle map $\iota_{\mathfrak{n}}: {}^{\mathfrak{n}}TM\to TM$, restricting to an isomorphism on $M\setminus \pa M$, such that
\begin{equation}
     \cV_{\nQFB}(M)= (\iota_{\mathfrak{n}})_* \CI(M;{}^{\mathfrak{n}}TM).
\label{wqb.6}\end{equation}
In fact, at $p\in M$, the fiber of ${}^{\mathfrak{n}}TM$ above $p$ is given by ${}^{\mathfrak{n}}T_pM= \cV_{\nQFB}(M)/\cI_p\cdot \cV_{\nQFB}(M)$, where $\cI_p$ is the ideal of smooth functions vanishing at $p$.  The vector bundle dual to ${}^{\mathfrak{n}}TM$, denoted ${}^{\mathfrak{n}}T^*M$, will be called the $\mathfrak{n}$-weighted $\QFB$-cotangent bundle ($\nQFB$-cotangent bundle for short).  We see from \eqref{wqb.4} that in the coordinates \eqref{coord.1}, the $\nQFB$-cotangent bundle is locally spanned over $\CI(M)$ by
\begin{equation}
  \rho_1^{-\mathfrak{n}}d\rho_1, \frac{dy^1_1}{v_1},\ldots,\frac{dy^{\ell_1}_1}{v_1} , \rho_2^{-\mathfrak{n}}d\rho_2, \frac{dy^1_2}{v_2}, \ldots,  \frac{dy^{\ell_2}_2}{v_2}, \ldots, \rho_k^{-\mathfrak{n}}d\rho_k, \frac{dy^1_k}{v_k},\ldots, \frac{dy^{\ell_k}_k}{v_k}, dz^1,\ldots,dz^q,
\label{wqb.7}\end{equation}
where  $v_i= \prod_{j=i}^k x_j$, $\rho_i= \prod_{j=i}^k x_j^{-\frac{1}{1-\nu_j}}$ and $\rho_i^{-\mathfrak{n}}= \prod_{j=i}^k x_j^{\frac{\nu_j}{1-\nu_j}}$ with $\nu_j:=\mathfrak{n}(H_j)$.  

The natural map $\iota_{\mathfrak{n}}: {}^{\mathfrak{n}}TM\to TM$ gives ${}^{\mathfrak{n}}TM$ the structure of a Lie algebroid and indicates that $(M,\cV_{\nQFB}(M))$ is a Lie structure at infinity for $M\setminus \pa M$ in the sense of \cite[Definition~3.1]{ALN04}.  As such, it comes with the following natural class of metrics.
\begin{definition}
An \textbf{$\mathfrak{n}$-weighted quasi-fibered boundary metric} ($\nQFB$-metric for short) is a choice of Euclidean metric $g_{\nQFB}$ for the vector bundle ${}^{\mathfrak{n}}TM$.  A \textbf{smooth $\nQFB$-metric} is a Riemannian metric on $M\setminus \pa M$ induced by some $\nQFB$-metric $g_{\nQFB}$ via the map $\iota_{\mathfrak{n}}: {}^{\mathfrak{n}}TM\to TM$.  Trusting this will lead to no confusion, we will also denote by $g_{\nQFB}$ the smooth $\nQFB$-metric induced by $g_{\nQFB}\in \CI(M;{}^{\mathfrak{n}}TM\otimes {}^{\mathfrak{n}}TM)$. 
\label{wqb.8}\end{definition}
In terms of a choice of $\nQFB$-metric, notice that the Lie algebra of $\nQFB$-vector fields can alternatively be defined by
\begin{equation}
  \cV_{\nQFB}(M)= \{ \xi\in \CI(M;TM) \; | \; \sup_{M\setminus \pa M} g_{\nQFB}(\xi,\xi)<\infty\}.
\label{wqb.8b}\end{equation}
Since $\nQFB$-metrics are induced by a Lie structure at infinity, they come with the following geometric properties.
\begin{lemma}
Any smooth $\nQFB$-metric is complete of infinite volume with bounded geometry.
\label{wqb.9}\end{lemma}
\begin{proof}
Since these metrics come from the Lie structure at infinity $(M,\cV_{\nQFB}(M))$ in the sense of \cite[Definition~3.3]{ALN04}, the result follows from \cite{ALN04} and \cite{Bui}.
\end{proof}
As in \cite{CDR2016}, we will be mostly interested in the case where the  manifold with fibered corners $(M,\phi)$ is such that for each maximal hypersurface $H$, $S_H=H$ and $\phi_H:H\to S_H$ is the identity map.  We say in this case that $(M,\phi)$ is a \textbf{$\QAC$-manifold with fibered corners} and  that an $\mathfrak{n}$-weighted quasi-fibered boundary metric is an  \textbf{$\mathfrak{n}$-weighted quasi-asymptotically conical metric} ($\nQAC$-metric for short).   More generally, we will replace $\QFB$ by $\QAC$ and quasi-fibered boundary by quasi-asymptotically conical whenever $(M,\phi)$ is a $\QAC$-manifold with fibered corners.  In fact, from now on, unless otherwise specified, we will assume that $(M,\phi)$ is a $\QAC$-manifold with fibered corners.  Now, on such a manifold, the class of metrics we are interested in is not quite $\nQAC$-metrics, but a conformally related one.
\begin{definition}
Let $(M,\phi)$ be a $\QAC$-manifold with fibered corners together with some Lie algebra of $\nQAC$-vector fields associated to some weight function $\mathfrak{n}$ and some $\mathfrak{n}$-weighted total boundary defining function 
$$
\rho^{-1}= \prod_{H\in \cM_1(M)} x_H^{\frac{1}{1-\nu_H}}.
$$  In such a setting, a \textbf{smooth $\mathfrak{n}$-warped $\QAC$-metric}  is a Riemannian metric $g_w$ on $(M\setminus \pa M)$ of the form 
\begin{equation}
  g_w= \rho^{2\mathfrak{n}}g_{\nQAC}
\label{wqb.10a}\end{equation} 
for some smooth $\nQAC$-metric $g_{\nQAC}$, where 
$$
    \rho^{\mathfrak{n}}:= \prod_{H\in\cM_1(M)} x^{-\frac{\nu_H}{1-\nu_H}}.
$$
\label{wqb.10}\end{definition}
\begin{example}
In the local basis \eqref{wqb.7}, an example of an $\mathfrak{n}$-warped $\QAC$-metric is given by 
$$
   g_w= \sum_{i=1}^k \rho^{2\mathfrak{n}} \lrp{ \rho_i^{-2\mathfrak{n}} d\rho_i^2 + \sum_{j=i}^{\ell_i} \frac{d(y_i^j)^2}{v_i^2}   } + \rho^{2\mathfrak{n}} \sum_{j=1}^q d(z^j)^2,
$$
where $\rho_i^{2\mathfrak{n}}= (\rho_i^{\mathfrak{n}})^2= \prod_{j=i}^k x_j^{\frac{-2\nu_j}{1-\nu_j}}$.
\label{wqb.11}\end{example}
In fact, $\mathfrak{n}$-warped $\QAC$-metrics have a nice iterative structure as illustrated by the next example.
\begin{example}
Near $H\in \cM_1(M)$ and a local trivialization of $\phi_H: H\to S_H$ over $\cU\subset S_H$, but away from $G$ for $G<H$, a local model of an $\mathfrak{n}$-warped $\QAC$-metric is given by 
\begin{equation}
   d\rho_H^2+ \rho_H^2g_{\cU}+ \rho_H^{2\nu_H}\kappa_H,
\label{w.5}\end{equation}
where $\rho_H:= \prod_{G\ge H}x_G^{-\frac{1}{1-\nu_G}}$,  $g_{\cU}$ is a smooth metric in $\cU$ and $\kappa_H$ is an $\mathfrak{n}_{Z_H}$-warped $\QAC$-metric on the fiber $Z_H= \phi_H^{-1}(s)$ for some $s\in \cU$ with weight function $\mathfrak{n}_{Z_H}$ given by
$$
      \mathfrak{n}_{Z_H}(Z_H\cap G):= \frac{\nu_G-\nu_H}{1-\nu_H} \quad \mbox{for} \quad G>H.
$$  
With respect to a fixed point on $Z_H$, the $\mathfrak{n}_{Z_H}$-weighted distance function of $\kappa_H$ in \eqref{w.5} is given by
$$
       \rho_{Z_H}:=  \prod_{G>H} x_G^{-\frac{1}{1-\mathfrak{n}_{z_H}(G)}}.
$$
In particular, in the model \eqref{w.5}, to be close to $H$ means that $x_H<c$ for some small constant $c>0$, which in terms of the functions $\rho_H$ and $\rho_{Z_H}$ corresponds to the inequality
\begin{equation}
    \rho_{Z_H}<c\rho_H^{1-\nu_H}.
\label{wg.4}\end{equation}
\label{w.1}\end{example}
\begin{remark}
The definition of $\nQAC$-metrics and $\mathfrak{n}$-warped $\QAC$-metrics would still make sense for a weight function $\mathfrak{n}: \cM_1(M)\to [0,1)$ satisfying  \eqref{wqb.1b}, but we need to restrict the values of $\mathfrak{n}$ to $\{0,\nu\}$ to have a nice iterative structure as in Example~\ref{w.1}.
\label{wqb.11b}\end{remark}
We can formally define the $\mathfrak{n}$-warped $\QAC$-tangent bundle by
\begin{equation}
        {}^{w}TM:= (\rho^{\mathfrak{-n}})({}^{\mathfrak{n}}TM) \quad \mbox{with space of sections} \quad \CI(M;{}^{w}TM):= \rho^{-\mathfrak{n}}\CI(M;{}^{\mathfrak{n}}TM)
\label{wtb.1}\end{equation}
and denote by ${}^{w}T^*M$ its dual, so that smooth $\mathfrak{n}$-warped $\QAC$-metrics correspond to elements of \linebreak $\CI(M;S^2({}^{w}T^*M))$.

As in \cite{CDR2016} and \cite{CR2021}, it will be useful to introduce yet another class of metrics to describe the weighted H\"older spaces that we will use.  
\begin{definition}
Let $(M,\phi)$ be a $\QAC$-manifold with fibered corners and let $x_{\max}$ be a product of boundary defining functions associated to all the maximal boundary hypersurfaces of $M$.   Let $\mathfrak{n}$ be a weight function and let $\rho$ be an $\mathfrak{n}$-weighted distance.  In such a setting, a \textbf{smooth $\mathfrak{n}$-weighted quasi $b$-metric} ($\nQb$-metric for short) on $M$ is a Riemannian metric $g_{\nQb}$ of the form
$$
      g_{\nQb}= x_{\max}^2 g_{\nQAC}
$$
for some smooth $\mathfrak{n}$-weighted $\QAC$-metric $g_{\nQAC}$.  
\label{wqb.12}\end{definition} 
\begin{remark}
When $\mathfrak{n}$ is the trivial weight function given by $\mathfrak{n}(H)=0$ for all $H\in\cM_1(M)$, an $\nQb$-metric corresponds to an $\Qb$-metric in the sense of \cite{CDR2016}.
\label{wqb.13}\end{remark}
As for $\nQAC$-metrics, $\nQb$-metrics can be defined in terms of a Lie structure at infinity.  To see this, consider the space
\begin{equation}
  \cV_{\nQb}(M):= \{ \xi\in\CI(M;TM)\; | \; \sup_{M\setminus \pa M} g_{\nQb}(\xi,\xi)<\infty\}
\label{wqb.14}\end{equation}
of smooth vector fields on $M$ uniformly bounded with respect to some choice of $\nQb$-metric $g_{\nQb}$.  From Definition~\ref{wqb.2}, Definition~\ref{wqb.12} and \eqref{wqb.8b}, the space \eqref{wqb.14} can alternatively be defined as the space of $b$-vector fields $\xi$ on $M$ such that
\begin{enumerate}
\item $\xi|_{H}$ is tangent to the fibers of $\phi_H:H\to S_H$ for all $H\in\cM_1(M)$ not maximal with respect to the partial order;
\item $\xi\rho^{-1}\in \frac{v\rho^{-1}}{x_{\max}}\CI(M)$.
\end{enumerate}
As for $\nQAC$-vector fields, one can check that these two conditions are closed under taking the Lie bracket, so that $\cV_{\nQb}(M)$ is a Lie subalgebra of the Lie algebra of $b$-vector fields.
By \eqref{wqb.14} and \eqref{wqb.4}, $\nQb$-vector fields are locally generated over $\CI(M)$ by \eqref{wqb.4} when $H_k$ is not maximal and by 
\begin{multline}
\frac{v_1x_1}{x_k}\frac{\pa}{\pa x_1}, \frac{v_1}{x_k}\frac{\pa}{\pa y_1}, \frac{v_2}{x_k}\left( (1-\nu_2)x_2\frac{\pa}{\pa x_2}- (1-\nu_1)x_1\frac{\pa}{\pa x_1} \right), \frac{v_2}{x_k}\frac{\pa}{\pa y_2},\ldots,  \\ \left( (1-\nu_k)x_k\frac{\pa}{\pa x_k}- (1-\nu_{k-1})x_{k-1}\frac{\pa}{\pa x_{k-1}} \right) \quad \mbox{and} \quad \frac{\pa}{\pa y_k}
\label{wqb.15}\end{multline}
otherwise.  Thus, $\cV_{\nQb}(M)$ is a locally free sheaf of rank $m$ over $\CI(M)$.  By the Serre-Swan theorem, there is a corresponding vector bundle ${}^{\nQb}TM= x_{\max}^{-2}{}^{\mathfrak{n}}TM$ and a map $\iota_{\nQb}: {}^{\nQb}TM\to TM$ such that
$$
     \cV_{\nQb}(M)= (\iota_{\nQb})_* \CI(M;{}^{\nQb}TM).
$$
In other words, the map $\iota_{\nQb}$ gives ${}^{\nQb}TM$ the structure of a Lie algebroid over $M$ and $(M,\cV_{\nQb}(M))$ is a Lie structure at infinity for $M\setminus \pa M$.  In particular, by \cite{ALN04} and \cite{Bui}, smooth $\nQb$-metrics have the following geometric properties.
\begin{lemma}
Any smooth $\nQb$-metric is automatically complete of infinite volume with bounded geometry.
\label{wqb.16}\end{lemma}

We will need various function spaces associated to these metrics.  First, recall that if $(X,g)$ is a Riemannian manifold and $E\to X$ is a Euclidean vector bundle over $X$ together with a connection $\nabla$ compatible with the Euclidean structure, we can for each $\ell\in \bbN_0$ associate the space $\cC^{\ell}(X;E)$ of continuous sections $\sigma: X\to E$ such that
$$
       \nabla^j\sigma\in \cC^0(X;T^0_jX\otimes E) \quad \mbox{and} \quad \sup_{p\in X} |\nabla^j\sigma|_g<\infty \quad \forall j\in\{0,\ldots,\ell\},
$$ 
where $\nabla$ denotes as well the connection induced by the Levi-Civita connection of $g$ and the connection on $E$, while $|\cdot|_g$ is the norm induced by $g$ and the Euclidean structure on $E$.  This is a Banach space with norm
$$
     \| \sigma\|_{g,\ell}:= \sum_{j=0}^\ell \sup_{p\in X} |\nabla^j\sigma(p)|_g.
$$
The intersection of these spaces yields the Fr\'echet space 
$$
    \CI_g(X;E):= \bigcap_{\ell\in\bbN_0} \cC^{\ell}_g(X;E).  
$$
For $\ell\in\bbN_0$ and $\alpha\in (0,1]$, there is also the H\"older space $\cC^{\ell,\alpha}_g(X;E)$ consisting of sections $\sigma\in \cC^{\ell}(X;E)$ such that
$$
   [\nabla^{\ell}\sigma]_{g,\alpha}:= \sup \left\{  \left.\frac{|P_{\gamma}(\nabla^{\ell}\sigma(\gamma(0)))- \nabla^{\ell}\sigma(\gamma(1))|}{\ell(\gamma)^{\alpha}} \; \right| \; 
   \gamma\in \CI([0,1];X), \; \gamma(0)\ne \gamma(1) \right\}<\infty,
$$
where $\ell(\gamma)$ is the length of $\gamma$ with respect to $g$ and $P_{\gamma}: T^0_{\ell}X\otimes E|_{\gamma(0)}\to T^0_{\ell}X\otimes E|_{\gamma(1)}$ is parallel transport along $\gamma$.  Again, this is a Banach space with norm given by 
$$
       \|\sigma\|_{g,\ell,\alpha}:= \|\sigma\|_{g,\ell}+ [\nabla^{\ell}\sigma]_{g,\alpha}.
$$
For $\mu\in\CI(X)$ a positive function, there are corresponding weighted versions
$$
      \mu\cC^{\ell,\alpha}_g(X;E):= \left\{ \sigma \; \left| \; \frac{\sigma}{\mu}\in \cC^{\ell,\alpha}_g(X;E) \right.\right\} \quad \mbox{with norm} \quad \|\sigma\|_{\mu\cC^{\ell,\alpha}_g}:= \left\| \frac{\sigma}{\mu}\right\|_{g,\ell,\alpha}.
$$

When $X=M\setminus \pa M$ and $g=g_{\nQb}$ is an $\nQb$-metric, we obtain the $\nQb$-H\"older space $C^{\ell,\alpha}_{\nQb}(M\setminus \pa M;E)$, as well as the space $\cC^{\ell}_{\nQb}(M\setminus \pa M;E)$.  Similarly, if $g=g_w$ is an $\mathfrak{n}$-warped $\QAC$-metric, we can define the $\mathfrak{n}$-warped $\QAC$-H\"older space $\cC^{\ell,\alpha}_w(M\setminus \pa M;E)$.  Since an $\nQb$-metric is conformally related to an $\mathfrak{n}$-warped $\QAC$-metric via
\begin{equation}
   g_w= \frac{g_{\nQb}}{\chi^2} \quad \mbox{with}  \quad \chi:= \frac{x_{\max}}{\rho^{\mathfrak{n}}}\ge 0 \; \mbox{bounded},
\label{hol.2b}\end{equation}
there is an obvious continuous inclusion 
\begin{equation}
\cC^{\ell,\alpha}_{\nQb}(M\setminus \pa M;E)\subset \cC^{\ell,\alpha}_w(M\setminus \pa M;E).
\label{hol.2}\end{equation}  
Conversely, there is the following partial counterpart.  
\begin{lemma}
For $0<\delta<1$, there is a continuous inclusion $\chi^{\delta}\cC^{0,1}_w(M\setminus \pa M;E)\subset \cC^{0,\alpha}_{\nQb}(M\setminus\pa M;E)$ for $\alpha\le \delta$.
\label{hol.1}\end{lemma}
\begin{proof}
Since $\chi$ is a product of powers of boundary defining functions, its logarithmic differential is automatically a $b$-differential in the sense of \cite{MelroseAPS}, so
$$
     \frac{d\chi}{\chi}\in \cA_{\phg}(M; {}^{b} T^*M)\subset \cA_{\phg}(M;{}^{\nQb}T^*M)\subset \CI_{\nQb}(M\setminus \pa M;{}^{\nQb}T^*M).
$$
Using this observation, we can run the same argument as in the proof of \cite[Lemma~3.9]{CR2021} to obtain the result, the starting point of this proof being a similar observation, namely \cite[(3.38)]{CR2021}.

\end{proof}

We can also consider the Sobolev space $H^{\ell}_{\nQb}(M\setminus\pa M)$ associated to an $\nQb$-metric.  For an $\mathfrak{n}$-warped $\QAC$-metric, instead of the natural Sobolev space associated to such a metric, we will consider the weighted version of the $\nQb$-Sobolev space
\begin{equation}
     H^{\ell}_{w}(M\setminus \pa M):= \chi^{\frac{\dim M}2}H^{\ell}_{\nQb}(M\setminus\pa M)),
\label{Sob.1}\end{equation}
where the factor $\chi^{\frac{\dim M}2}$ ensures that we integrate with respect to the volume density of an $\mathfrak{n}$-warped $\QAC$-metric, but with pointwise norms of the derivatives measured with respect to an $\nQb$-metric instead of an $\mathfrak{n}$-warped $\QAC$-metric.

So far, we have only considered \textit{smooth} $\nQb$-metrics and $\mathfrak{n}$-warped $\QAC$-metrics, that is, metrics corresponding to elements of $\CI(M;S^2({}^{\nQb}T^*M))$ and $\CI(M;S^2({}^{w}T^*M))$,  but to look for Calabi-Yau examples, it will be important to be less restrictive on the regularity of these metrics at the boundary.  More precisely, we will look at $\nQb$-metrics corresponding to sections of $\CI_{\nQb}(M\setminus\pa M;S^2({{}^{\nQb}T^*M}))$ (quasi-isometric to some fixed smooth $\nQb$-metric).    For $\mathfrak{n}$-warped $\QAC$-metrics, we could look at those corresponding to sections of $\CI_{w}(M\setminus\pa M;S^2({}^{w}T^*M))$ (quasi-isometric to some smooth $\mathfrak{n}$-warped $\QAC$ metric), but keeping in mind the continuous inclusion \eqref{hol.2}, we will in fact be stricter and consider $\mathfrak{n}$-warped $\QAC$-metrics corresponding to sections of $\CI_{\nQb}(M\setminus\pa M;S^2({}^{w}T^*M))$, that is, $g_w\in \CI_{w}(M\setminus\pa M;S^2({}^{w}T^*M))$ of the form
$$
g_w= \frac{g_{\nQb}}{\chi^2} 
$$
for some $\nQb$-metric $g_{\nQb}\in \CI_{\Qb}(M\setminus\pa M;S^2({{}^{\nQb}T^*M}))$.  We will say that such an $\mathfrak{n}$-warped $\QAC$-metric is $\nQb$-smooth.    Clearly, Lemma~\ref{wqb.16} still holds for $\nQb$-metrics in $\CI_{\nQb}(M\setminus \pa M;S^2({{}^{\nQb}T^*M}) )$, since by assumption  we control the curvature and its derivatives.  For $\mathfrak{n}$-warped $\QAC$-metrics, we also have such a result. 
\begin{proposition}
Any $\mathfrak{n}$-warped $\QAC$-metric $g_w\in\CI_{\nQb}(M\setminus \pa M;S^2({}^{w}T^*M))$ is complete of infinite volume with bounded geometry.
\label{w.2}\end{proposition}
\begin{proof}
By assumption,
$$
g_w= \frac{g_{\nQb}}{\chi^2} 
$$
for some $\nQb$-metric $g_{\nQb}\in \CI_{\Qb}(M\setminus\pa M;S^2({{}^{\nQb}T^*M}))$. By Lemma~\ref{wqb.16} and the comment above, we know that $g_{\nQb}$ is complete of infinite volume with bounded geometry.  On the other hand, $\chi$ is a bounded positive function, so $g_w$ is automatically complete of infinite volume.  By the continuous inclusion \eqref{hol.2} and \cite[Theorem~1.159]{Besse}, we see that the curvature of $g_w$ is bounded, as well as its covariant derivatives.  To see that it is of bounded geometry, it suffices then to show that the injectivity radius is positive, which follows from the facts that $g_{\nQb}$ has positive injectivity radius and that $\chi$ is a bounded positive function on $M\setminus \pa M$.  
\end{proof}

\section{Mapping properties of the Laplacian} \label{mpl.0}

The mapping properties of the Laplacian of a $\QAC$-metric obtained in \cite{DM2014} have a natural analogue for warped $\QAC$-metrics.  Indeed, such results were already obtained for some specific examples of warped $\QAC$-metrics of depth one in \cite{CR2021}; see also \cite{YangLi, Szekelyhidi,Firester} for different approaches.  In this section, we will combine the arguments \cite{DM2014} and \cite{CR2021} to obtain mapping properties of the Laplacian for warped $\QAC$-metrics of arbitrary depth.  It would presumably be possible to obtain such mapping properties using barrier functions as in \cite{Joyce}, though probably only ensuring invertibility of the Laplacian for a smaller range of weights.    Thus, let 
$$
      \mathfrak{n}: \cM_1(M)\to \{0,\nu\}
$$
be a weight function for some $\nu\in [0,1)$.  The set of boundary hypersurfaces of $M$ therefore decomposes as
$$
    \cM_{1}(M)= \cM_{1,0}(M)\cup \cM_{1,\nu}
$$
with $\cM_{1,a}(M)= \mathfrak{n}^{-1}(a)$.  This is a disjoint union if $\nu>0$, and otherwise $\cM_{1}(M)=\cM_{1,0}(M)=\cM_{1,\nu}(M)$.  

For the convenience of the reader, let us first recall the general strategy of \cite{DM2014}.  If $g$ is a complete Riemannian metric on a manifold $Z$ and $h$ is a positive smooth function on it, then we can introduce a measure $d\mu=h^2dg$ using the volume density $dg$ of $g$.  The triple $(Z,g,\mu)$ is then a complete weighted Riemannian manifold in the sense of \cite{GSC2005}.  On such manifolds, the Riemannian metric $g$ induces a distance function $d(p,q)$ between two points $p,q\in Z$.  We will denote by 
$$
    B(p,r):= \{ q\in Z \; | \; d(p,q)<r\}
$$  
the geodesic ball of radius $r$ centered at $p\in Z$.  To measure the volume of such balls, we use however the measure $\mu$, not the volume density of $g$.  Similarly, the natural $L^2$-inner product of two functions is the one induced by the measure $\mu$, namely
\begin{equation}
  \langle u,v\rangle_{\mu}:= \int_Z uv d\mu.
\label{w.6}\end{equation}
If $\nabla$ is the Levi-Civita connection of $g$, we will be interested in studying the mapping properties of the corresponding Laplacian $\Delta=\operatorname{div}\circ\nabla$ of $g$.  More generally, for $\cR$ a function, we will consider the operator $\cL:=-\Delta+\cR$, as well as its Doob transform with respect to $h$,
\begin{equation}
\widetilde{\cL}:= h^{-1}\circ \cL \circ h= -\Delta_{\mu}+V+\cR,
\label{w.7}\end{equation}
where $V:=-\frac{\Delta h}h$ and $-\Delta_{\mu}=\nabla^{*,\mu}\nabla$ with $\nabla^{*,\mu}$ the adjoint of $\nabla$ with respect to the $L^2$-inner product \eqref{w.6} and the $L^2$-inner product on forms given by
\begin{equation}
\langle \eta_1,\eta_2\rangle_{\mu}:= \int_Z (\eta_1,\eta_2)_g(z) d\mu(z).
\label{w.8}\end{equation}
Denote by $H_{\cL}(t,z,z')$ and $H_{-\Delta+V}(t,z,z')$ the heat kernels of $\cL$ and $-\Delta+V$ with respect to the volume density $g$ and let 
\begin{equation}
G_{\cL}(z,z')=\int_0^{\infty} H_{\cL}(t,z,z')dt \quad \mbox{and} \quad G_{-\Delta+V}(z,z')=\int_0^{\infty} H_{-\Delta+V}(t,z,z')dt
\label{w.9}\end{equation}
be the corresponding Green's functions.  Similarly, let $H_{-\Delta_{\mu}}(t,z,z')$ be the heat kernel of $-\Delta_{\mu}$ with respect to the measure $\mu$ with corresponding Green's function
$$
  G_{-\Delta_{\mu}}(z,z')=\int_0^{\infty} H_{-\Delta_{\mu}}(t,z,z')dt.
$$
By a result of \cite{DM2014}, those heat kernels and Green's functions are related as follows.

\begin{lemma}[Theorem~3.12 in \cite{DM2014}]  If $\cR\ge V$, then 
$$
      |H_{\cL}(t,z,z')|\le H_{-\Delta+V}(t,z,z')\le h(z)h(z')H_{-\Delta_{\mu}}(t,z,z')
$$
and 
$$
   |G_{\cL}(z,z')|\le G_{-\Delta+V}(z,z')\le h(z)h(z')G_{-\Delta_{\mu}}.
$$

\label{w.10}\end{lemma}

Thus, to obtain control on $H_{\cL}(t,z,z')$ and $G_{\cL}(z,z')$, it suffices to obtain control on $H_{-\Delta_{\mu}}$.  This can be achieved by the method of Grigor'yan and Saloff-Coste invoking the following notions.  
\begin{definition}
The complete weighted Riemannian manifold $(Z,g,\mu)$ satisfies
\begin{itemize}

\item[$(VD)_{\mu}$] \textbf{the weighted volume doubling property} if there exists $C_D>0$ such that 
 $$
        \mu(B(p,2r))\le C_D\mu(B(p,r)) \quad \forall p\in Z, \; \forall r>0;
 $$
\item[$(PI)_{\mu,\delta}$] \textbf{the uniform weighted Poincar\'e inequality with parameter} $\delta\in (0,1]$ if there exists a constant $C_P>0$ such that 
$$
     \int_{B(p,r)} (f-\overline{f})^2d\mu \le C_P r^2 \int_{B(p,\delta^{-1}r)} |df|^2_g d\mu \quad \forall f\in W^{1,2}_{loc}(Z), \; \forall p\in Z, \; \forall r>0;
$$
\item[$(PI)_{\mu}$] \textbf{the uniform weighted Poincar\'e inequality} if we can take $\delta=1$ in the previous statement.
\end{itemize}
\label{w.11}\end{definition}
In terms of these conditions, one of the main results of \cite{GSC2005} is the following.
\begin{theorem}[Theorem~2.7 in \cite{GSC2005}]  Let $(Z,g,\mu)$ be a complete weighted Riemannian manifold satisfying $(VD)_{\mu}$ and $(PI)_{\mu}$.  Then there are positive constants $C$ and $c$ such that
$$
  c(\mu(B(z,\sqrt{t})),\mu(B(z',\sqrt{t})))^{-\frac12}e^{-C\frac{d(z,z')^2}t}\le H_{-\Delta_{\mu}}(t,z,z')\le C(\mu(B(z,\sqrt{t})),\mu(B(z',\sqrt{t})))^{-\frac12}e^{-c\frac{d(z,z')^2}t}
$$
for all $(t,z,z')\in (0,\infty)\times Z\times Z$.
\label{w.12}\end{theorem}
By Lemma~\ref{w.10}, this yields the following.
\begin{corollary}
If the complete weighted Riemannian manifold $(Z,g,\mu)$ satisfies $(VD)_{\mu}$ and $(PI)_{\mu}$, then there are positive constants $c$ and $C$ such that 
$$
  c(\mu(B(z,\sqrt{t})),\mu(B(z',\sqrt{t})))^{-\frac12}e^{-C\frac{d(z,z')^2}t}\le \frac{H_{-\Delta+V}(t,z,z')}{h(z)h(z')}\le C(\mu(B(z,\sqrt{t})),\mu(B(z',\sqrt{t})))^{-\frac12}e^{-c\frac{d(z,z')^2}t}
$$
for all $(t,z,z')\in (0,\infty)\times Z\times Z$.
\label{w.13}\end{corollary}

We want to apply this result when $Z=M\setminus \pa M$ is the interior of a manifold with fibered corners $M$ of dimension $m$ and $g$ is a warped $\QAC$-metric with weight $\mathfrak{n}$ for some choice of $\nu\in [0,1)$.  The type of measure $\mu$ we will consider will be one of the form
\begin{equation}
   d\mu_a= x^adg \quad \mbox{with} \quad x^a=\prod_{H\in\cM_1(M)} x_H^{a_H}  
\label{w.14}\end{equation}
for some 
$$
\begin{array}{lccl} a: & \cM_1(M) & \to & \bbR \\ & H & \mapsto & a_H.    \end{array}
$$
It will be convenient to distinguish between maximal and non-maximal boundary hypersurfaces.  To this end, we will use the notation
$$
    \cM_{\max}(M):= \{H\in\cM_1(M)\; | \; H \; \mbox{is maximal}\} \quad \mbox{and}  \quad \cM_{\nm}(M):=\cM_1(M)\setminus \cM_{\max}(M).
$$
We will denote by 
$$
   v:= \prod_{H\in\cM_1(M)} x_H
$$
a total boundary defining function for $M$.
The distance function of $g$ with respect to a fixed point $o\in M\setminus \pa M$ is comparable to
\begin{equation}
     \rho= \prod_{H\in\cM_1(M)} x_H^{-\frac{1}{1-\nu_H}}.
\label{w.15}\end{equation} 
Near $H$, but away from $G$ for $G<H$, the distance function with respect to a fixed point is also comparable to $\rho_H= \prod_{G\ge H}x_G^{-\frac{1}{1-\nu_G}}$.  It will be convenient also to have the following weighted version of the function $v_H$,  namely
\begin{equation}
 \widetilde{v}_H:= \rho_H^{-(1-\nu_H)}= \left\{  \begin{array}{ll} \prod_{G\ge H} x_G, & \nu_H=\nu, \\ \prod_{G\ge H} x_G^{\frac{1}{1-\nu_G}}, & \nu_H=0,  \end{array}   \right. 
\label{wg.1a}\end{equation}  
as well as the function
\begin{equation}
    \sigma=\left\{  \begin{array}{ll} \prod_{H\in\cM_{1,\nu}(M)} x_H^{-\frac{1}{1-\nu}}, & \nu>0, \\ 1, & \nu=0, \end{array}  \right.
\label{wg.1z}\end{equation}
such that $\sigma^{\nu}=\rho^{\mathfrak{n}}$.

We will also denote by
$$
      m_H= \dim M-\dim S_H-1
$$
the dimension of the fibers of $\phi_H: H\to S_H$ and by $b_H=\dim S_H$ the dimension of the base.  

To apply the previous result to this setting, we will need to check that $(VD)_{\mu}$ and $(PI)_{\mu}$ hold for suitable choices of weight $a: \cM_1(M)\to \bbR$.  To that end, recall the following notation used in \cite{DM2014}.
\begin{definition}
Fix once and for all a basepoint $o\in M\setminus \pa M$.  A ball of radius $R$ at $o$ is called \textbf{anchored} and we denote its volume by 
$$
     \cA(R;a)= \mu_a(B(o,R)).
$$
Fix $c\in (0,1)$.  With respect to this choice, a ball $B(p,r)$ is said to be \textbf{remote} if $r<cd(o,p)$, in which case we use the notation
$$
    \cR(p,r;a):= \mu_a(B(p,r)).
$$
If $B(p,r)$ is any ball, possibly neither anchored nor remote, we use the notation 
$$
      \cV(p,r;a)= \mu_a(B(p,r)).
$$
\label{w.16}\end{definition}

Following the strategy of \cite{DM2014}, we will derive an estimate of the volume of anchored balls in terms of the volume of remote balls.  We start with the following estimate.
\begin{proposition}
Provided $(1-\nu_H)(a_H+m_H)\ne m$ for all $H\in \cM_1(M)$ and $(1-\nu_H)(a_H+m_H)\ne (1-\nu_G)(a_G+m_G)$ for all $G,H \in \cM_1(M)$ with $G<H$, we have that for $R\ge 1$,
\begin{equation}
     \cA(R;a) \asymp 1+ R^m \sum_{H\in\cM_1(M)} R^{(\nu_H-1)(a_H+m_H)},
\label{w.17a}\end{equation}
where the notation $f_1\asymp f_2$ means that there exists  positive constants $c$ and $C$ such that $cf_2\le f_1\le Cf_2$.
In particular, in the $\QAC$ setting, that is, when $\nu=0$, this gives
\begin{equation}
    \cA(R;a)\asymp 1+ R^m \sum_{H\in \cM_1(M)} R^{-a_H-m_H}.
\label{w.17b}\end{equation}
\label{w.17}\end{proposition}
\begin{remark}
When we drop the condition that $(1-\nu_H)(a_H+m_H)\ne m$ for all $H\in \cM_1(M)$ and $(1-\nu_H)(a_H+m_H)\ne (1-\nu_G)(a_G+m_G)$ for all $G,H \in \cM_1(M)$ with $G<H$, a similar result holds, but with some powers of $R$ multiplied by some positive integer power of $\log R$.
\label{w.19}\end{remark}

\begin{proof}
When $M$ is of depth $1$, we are in the $\AC$-setting with $v\asymp \rho^{\nu-1}$, so we have
$$
     \cA(R;a)\asymp 1+ \sum_{H\in \cM_1(M)} R^{m+ (\nu-1)(a_H+m_H)}
$$
by taking into account the contribution of each asymptotically conical end and taking into account that $m_H=0$ in this case.  We can therefore proceed by induction on the depth of $M$ to obtain the result.  More precisely,
in an open set $\cV$ where the local model \eqref{w.5} is valid, we need to show that 
$$
   \mu_a(B(o,R)\cap \cV) \asymp 1+ R^m \sum_{G\ge H} R^{(\nu_G-1)(a_G+m_G)}.
$$
Now, using that $\rho\asymp \rho_H$ in this region, we see that
$$
\mu_a(B(o,R)\cap \cV) \asymp 1+ \int_1^R\rho^{(\nu_H-1)a_H}\rho^{b_H}\rho^{\nu_H m_H}\cA(\rho^{1-\nu_H}; \widetilde{a}) d\rho,
$$
where $\widetilde{a}: \{ G\in \cM_1(M) \; | \; G>H\} \to \bbR$ is given by
\begin{equation}
  \widetilde{a}(G)= \widetilde{a}_G:= \left\{ \begin{array}{ll}a_G-a_H, & \nu_H=\nu, \\ a_G- \frac{a_H}{1-\nu_G}, & \nu_H=0. \end{array} \right.
\label{w.18}\end{equation}

On the fibers of $\phi_H: H\to S_H$, recall that the model metrics are actually $\QAC$-metrics (even if $g$ is a warped $\QAC$-metric with factor $\nu>0$) when $\nu_H=\nu$.  Thus assuming by induction on the depth that \eqref{w.17b} holds for these metrics, we compute that when $\nu_H=\nu$, 
$$
\begin{aligned}
 \mu_a(B(o,R)\cap \cV) &\asymp 1+ \int_1^R\rho^{(\nu-1)a_H}\rho^{b_H}\rho^{\nu m_H}\cA(\rho^{1-\nu}; \widetilde{a}) d\rho \\
  & \asymp 1+ \int_1^R \rho^{(\nu-1)a_H+b_H+\nu m_H}\lrp{ 1+ \rho^{(1-\nu)m_H} \sum_{G>H} \rho^{(1-\nu)(-(a_G-a_H)-m_G)}   }d\rho \\
  & \asymp 1+ R^{(\nu-1)a_H+b_H+1+\nu m_H}\lrp{ 1+ R^{(1-\nu)m_H}  \sum_{G>H} R^{(1-\nu)(-(a_G-a_H)-m_G)}  } \\
  & \asymp 1+ R^{m} \sum_{G\ge H} R^{(\nu-1)(a_G+m_G)} \\
  &\asymp 1+ R^{m} \sum_{G\ge H} R^{(\nu_G-1)(a_G+m_G)}.
\end{aligned}
$$
If instead $\nu_H=0$, then
$$
\begin{aligned}
\mu_a(B(o,R)\cap \cV) &\asymp 1+\int_1^R \rho_H^{-a_H+b_H} \cA(\rho_H;\widetilde{a})d\rho  \\
 &\asymp 1+ \int_1^R \rho_H^{-a_H+b_H} \lrp{ 1+\rho_H^{m_H} \sum_{G>H} \rho_H^{(\nu_G-1)(a_G- \frac{a_H}{1-\nu_G}+m_G)}  }d\rho_H \\
 & \asymp 1+ R^{-a_H+b_H+1} \lrp{1+ R^{m_H}\sum_{G>H} R^{(\nu_G-1)(a_G+m_G)+a_H}} \\
 & \asymp 1+ R^m \lrp{ \sum_{G\ge H}R^{(\nu_G-1)(a_G+m_G)}}.
\end{aligned}
$$
In both computations, the conditions on the weights $a_H$ and $a_G$ have been used to ensure there are no logarithmic terms when we integrate in $\rho$.

\end{proof}

\begin{corollary}
Suppose that $a$ is such that $a_H= a_{\max}$ for all $H\in \cM_{\max}(M)$, that $a_{\max}<\frac{m}{1-\nu}$ and that $(1-\nu)a_{\max}<(1-\nu_H)(a_H+m_H)$ for all $H\in \cM_{\nm}(M)$.  In this case,  
$$
          \cA(R;a)\asymp R^{m+(\nu-1)a_{\max}} \quad \mbox{for} \; R>1.
$$
\label{w.20}\end{corollary}

For remote balls, we have the following preliminary estimate.
\begin{proposition}
Suppose that $a_H=a_{\max}$ for all $H\in\cM_{\max}(M)$.  Fix $p\in M\setminus \pa M$ and suppose that we have chosen a remote parameter $c\in(0,1)$ so that in fact $c\in (0,\frac15)$.  If $x_G(p)\ge 1-3c$ for all $G\in \cM_{\nm}(M)$, then for $r\in (0,c\rho(p))$, 
\begin{equation}
 \cR(p,r;a)\asymp \rho(p)^{(\nu-1)a_{\max}}r^m.
\label{w.24a}\end{equation} 
If instead $x_H(p)<1-3c$ for some $H\in \cM_{\nm}(M)$ and $x_G(p)\ge 1-3c$ for $G<H$, then 
\begin{equation}
 \cR(p,r;a)\asymp \rho(p)^{(\nu_H-1)a_H+\nu_H m_H}r^{b_H+1}\cV\lrp{p_{Z_H}, \frac{r}{\rho(p)^{\nu_H}}; \widetilde{a}},
\label{w.24b}\end{equation}
where $p_{Z_H}$ is the projection of $p$ onto the factor $Z_H$ in the decomposition \eqref{w.5}  and $\widetilde{a}$ is given by \eqref{w.18}.
\label{w.24}\end{proposition}
\begin{proof}
First, notice that $\rho(z)\asymp \rho(p)$ for $z\in B(p,c\rho(p))$.  If $x_G(p)\ge 1-3c$ for all $G\in \cM_{\nm}(M)$, we are in a region where the metric behaves like an $\AC$-metric with $x_H\asymp \rho^{\nu-1}$ for $H\in\cM_{\max}(M)$, so the result follows by a simple rescaling argument as in \cite[Proposition~4.3]{DM2014}.  If instead $x_H(p)<1-3c$ for some $H\in \cM_{\nm}(M)$, but $x_G(p)\ge 1-3c$ for $G<H$, then $B(p,r)\subset V_H$ with 
$$
     V_H= \{ p\in M\setminus \pa M \; | \; x_H(p)<1-c\}.
$$
Indeed, if we set $r_2:= \rho^{\nu_H} \rho_{Z_H}=\rho x_H$, then using the triangle inequality in terms of the model metric \eqref{w.5}, we see that if $q\in B(p,r)$, then
\begin{equation}
\rho(q)\ge \rho(p)-d(p,q)\ge \rho(p)-r\ge \rho(p)-c\rho(p)= (1-c)\rho(p)
\label{w.25}\end{equation}
and
\begin{equation}
r_2(q)\le r_2(p)+d(p,q)\le r_2(p)+ c\rho(p),
\label{w.26}\end{equation}
so that
$$
   x_H(q)= \frac{r_2(q)}{\rho(q)}\le \frac{r_2(p)+c\rho(p)}{(1-c)\rho(p)}\le \frac{(1-3c)\rho(p)+ c\rho(p)}{(1-c)\rho(p)}= \frac{1-2c}{1-c}=1-\frac{c}{1-c}< 1-c.  
$$
Similarly, for $G<H$, setting instead $r_2=\rho^{\nu_G}\rho_{Z_G}=\rho x_G$, we have the inequalities
$$
   \rho(q)\le \rho(p)+ d(p,q)\le (1+c)\rho(p)
$$
and 
$$
     r_2(q)\ge r_2(p)-d(p,q)\ge r_2(p)-c\rho(p)
$$
that we can use to show that $x_G(q)\ge 1-5c>0$ for $q\in B(p,r)$.  So we can use the model metric \eqref{w.5} with $\kappa_H$ a warped $\QAC$-metric on the fiber $Z_H$ of $\phi_H: H\to S_H$ to estimate $\mu_a(B(p,r))$ with $r\in (0,c\rho(p))$.  If $p$ corresponds to the point $(p_1,p_2)\in (\bbR^+\times S_H)\times Z_H$, let $B_1(p_1,r)$ be the geodesic ball in $C_H= \bbR^+\times S_H$ and $B_2(p_2,r)$ denote the geodesic ball in $(Z_H,g_{Z_H})$.  Since $\rho\asymp \rho(p)$ on $B(p,r)$, notice that the weighted volume of $B(p,r)$ is comparable to that of the product of balls
$$
        B_1(p_1,r)\times B_2\lrp{p_2, \frac{r}{\rho(p)^{\nu_H}}}.
$$ 
Now, we compute that
$$
\begin{aligned}
\mu_a\lrp{B_1(p_1,r)\times B_2\lrp{p_2,\frac{r}{\rho(p)^{\nu_H}}}} &=  \int_{B_1(p,r)} \lrp{  \int_{B_2(p_2,\frac{r}{\rho(p)^{\nu_H}})} x^{\widetilde{a}}d\kappa_H    } \rho^{(\nu_H-1)a_H+\nu_H m_H} dg_{C_H} \\
   & \asymp \rho(p)^{(\nu_H-1)a_H+\nu_H m_H}r^{b_H+1}\cV\lrp{p_2, \frac{r}{\rho(p)^{\nu_H}};\widetilde{a}}
\end{aligned}
$$
as claimed, where $g_{C_H}= d\rho^2+ \rho^2g_{S_{H}}$ is the natural cone metric on the cone $C_H= \bbR^+\times S_H$.

\end{proof}

On the other hand, for the volume of non-remote balls, we have the following estimate.
\begin{proposition}
Suppose that $a_H=a_{\max}$ for all $H\in\cM_{\max}(M)$, that $a_{\max}<\frac{m}{1-\nu}$ and that $(1-\nu)a_{\max}< (1-\nu_H)(a_H+m_H)$ for all $H\in \cM_{nm}(M)$.  If $c\in (0,\frac15)$ is a remote parameter, then for $p\in M\setminus \pa M$,
\begin{equation}
      \cV(p,r,a)\asymp r^{m+ (\nu-1)a_{\max}}
\label{w.27a}\end{equation}
for $r\ge c\rho(p)$.
\label{w.27}
\end{proposition}
\begin{proof}
We can proceed by induction on the depth of $M$ using the previous proposition.  If $C>2$ is a fixed constant, then for $r>C\rho(p)$, 
$$
   \{ q\; | \; \rho(q)\le (1-C^{-1})r\} \subset B(p,r)\subset \{ q\; | \; \rho(q)\le (C^{-1}+1)r\},
$$
so for such $r$, 
$$
   \cV(p,r;a)\asymp \cA(r;a)\asymp r^{m+ (\nu-1)a_{\max}}
$$
by Corollary~\ref{w.20}.  If instead $c\rho(p)\le r \le C\rho(p)$, then 
$$
        \cR(p,c\rho(p);a)\le \cV(p,r;a)\le \cA((1+C)\rho(p);a).
$$
By Corollary~\ref{w.20}, the right hand side behaves like $r^{m+(\nu-1)a_{\max}}$.  For the left hand side, notice that since $r\asymp \rho(p)$, it behaves like $r^{m+(\nu-1)a_{\max}}$ by Proposition~\ref{w.24}.  Indeed, if $x_G(p)\ge 1-3c$ for all $G\in \cM_{\nm}(M)$, this follows from \eqref{w.24a}.  If instead $x_H(p)<1-3c$ for some $H\in \cM_{\nm}(M)$ and $x_{G}(p)\ge 1-3c$ for all $G<H$, then in the notation of the proof of Proposition~\ref{w.24},
$$
 \rho(p)\ge \rho_{Z_H}(p) \quad \Longrightarrow \quad \frac{r}{\rho(p)^{\nu_H}}\ge c\rho(p)^{1-\nu_H}= \frac{c\rho_{Z_H}(p_{Z_H})}{x_H}\ge c\rho_{Z_H}(p_{Z_H}),
$$
so we can apply \eqref{w.27a} to $Z_H$.  By \eqref{w.24b}, we know that
$$
\cR(p,c\rho(p);a) \asymp \rho(p)^{(\nu_H-1)a_H+ \nu_H m_H}r^{b_H+1}\cV\lrp{p_{Z_H},\frac{r}{\rho(p)^{\nu_H}};\widetilde{a}},
$$
so if $\nu_H=\nu$, this gives
$$
\begin{aligned}
\cR(p,c\rho(p);a)  & \asymp \rho(p)^{(\nu-1)a_H+ \nu m_H}r^{b_H+1} \lrp{\frac{r}{\rho(p)^{\nu}} }^{m_H-(a_{\max}-a_H)} \quad \mbox{by \eqref{w.27a} on $Z_H$ with $\nu=0$,} \\
 & \asymp r^{m+ (\nu-1)a_{\max}}, \quad \mbox{since $r\asymp \rho(p)$},
\end{aligned}
$$
while if $\nu_H=0$, we obtain that
$$
\begin{aligned}
\cR(p,c\rho(p);a) 
 & \asymp \rho(p)^{-a_H}r^{b_H+1} \lrp{r }^{m_H-(1-\nu)(a_{\max}-\frac{a_H}{1-\nu})} \quad \mbox{by \eqref{w.27a} on $Z_H$,} \\
 & \asymp r^{m+ (\nu-1)a_{\max}}, \quad \mbox{since $r\asymp \rho(p)$}.
\end{aligned}
$$

\end{proof}

The last two propositions can be used to obtain a sharper estimate of the volume of remote balls, a technical result needed later.  We need however to take into account how close the point $p$ is to a corner of $M$.  In \cite{DM2014}, this is achieved with the notion of remote chains, but we will proceed differently, taking advantage of the fact that we have a compactification by a manifold with fibered corners.
\begin{definition}
Fix a remote parameter $c\in(0,\frac13)$ and boundary hypersurfaces  $H_1,\ldots, H_k\in \cM_{\nm}(M)$ with $H_1<\cdots < H_k$.  For such a choice, we say that a point $p\in M\setminus \pa M$ is \textbf{close to the non-maximal corner} $H_1\cap\cdots \cap H_k$ if $x_{H_i}(p)\le 1-3c$ for all $i\in\{1,\ldots,k\}$, but $x_{H}(p)>1-3c$ for $H\in \cM_{\nm}(M)\setminus\{H_1,\ldots, H_k\}$.  On the other hand, if $x_{H}(p)>1-3c$ for all $H\in \cM_{\nm}(M)$, we say that $p$ is \textbf{far from all non-maximal corners}.  Close to a non-maximal corner $H_1\cap\cdots H_k$, we will consider the function 
\begin{equation}
\sigma:= x_{\max}^{-\frac{1}{1-\nu}}\prod_{\nu_{H_i}=\nu} x_i^{-\frac{1}{1-\nu}}
\label{cc.1a}\end{equation}
corresponding to \eqref{wg.1z}.
\label{cc.1}\end{definition}

In terms of this notion, we have the following finer estimate on the volume of remote balls.

\begin{proposition}
Suppose that $a_H=a_{\max}$ for all $H\in\cM_{\max}(M)$ and that $(1-\nu)a_{\max}< (1-\nu_H)(a_H+m_H)$ for all $H\in \cM_{nm}(M)$.  If $c\in(0,\frac15)$ is a remote parameter and $p\in M\setminus \pa M$ is close to the non-maximal corner $H_1\cap\cdots H_k$ and
$$
     \frac{c\sigma^{\nu_{i+1}}(p)}{ \widetilde{v}_{i+1}(p)}\le r\le \frac{c\sigma^{\nu_i}(p)}{ \widetilde{v}_{i}(p)}
$$     
for some $i\in \{1,\ldots, k+1\}$,   where $ \widetilde{v}_i= \widetilde{v}_{H_i}$ in terms of \eqref{wg.1a} and $\nu_i=\nu_{H_i}$ with the conventions that $ \widetilde{v}_{k+1}=x_{\max}$, $\nu_{k+1}=\nu$ and $\frac{c\sigma^{\nu_{k+2}}(p)}{v_{k+2}(p)}:=0$.   Then, 
\begin{equation}
  \cR(p,r;a)\asymp \lrp{ \prod_{j=1}^{i-1}x_j^{a_j}(p) }  \widetilde{v}_i^{a_i}(p) \sigma^{\nu_i(a_{\max}-a_i)}(p) r^{m-(a_{\max}-a_i)+ (\nu-\nu_i)a_{\max}}
\label{w.30b}\end{equation}
with the conventions that $a_{k+1}=a_{\max}$ and $a_0=0$.  With these conventions, \eqref{w.30b} also holds with $k=0$ and $i=1$ when $p$ is far from all non-maximal corners and $r\le \frac{c\sigma(p)^{\nu_1}}{ \widetilde{v}_1(p)}=cx_{\max}^{-\frac{1}{1-\nu}}$.
\label{w.30}\end{proposition}
\begin{proof}
We can proceed by induction on the depth of $M$.  If $M$ is of depth 1, we are in the $\AC$ setting, so $k=0$ and $i=1$.  For a remote ball, 
$
      r\le c\rho(p) = c\rho_1(p)=c\rho_{\max}(p),
$
so that \eqref{w.30b} follows from \eqref{w.24a}.  If $M$ is of higher depth and $k\ge 1$, we can then assume that \eqref{w.30b} holds for $\QAC$-manifolds of lower depth.  We need to distinguish two cases.  If 
$$
         \frac{c\sigma^{\nu_2}(p)}{ \widetilde{v}_2(p)}\le r \le \frac{c\sigma^{\nu_1}(p)}{ \widetilde{v}_1(p)},
$$
then we can apply Proposition~\ref{w.24} together with Proposition~\ref{w.27} to obtain the claimed estimate, namely, with the notation that $m_i= m_{H_i}$, if $\nu_1=\nu$, then $\sigma\asymp \rho$ and
$$
\begin{aligned}
\cR(p,r;a) & \asymp \rho(p)^{(\nu-1)a_1+ \nu m_1}r^{m-m_1}\cV(p_{Z_{H_1}},\frac{r}{\rho(p)^\nu};\widetilde{a}) \quad \mbox{by Proposition~\ref{w.24}}, \\
   & \asymp \rho(p)^{a_1(\nu-1)+\nu m_1} r^{m-m_1} \lrp{ \frac{r}{\rho^{\nu}}  }^{m_1-(a_{\max}-a_1)} \quad \mbox{by Proposition~\ref{w.27} with $\nu=0$,} \\
   & \asymp \rho(p)^{\nu a_{\max}-a_1}r^{m-(a_{\max}-a_1)}= \rho(p)^{(\nu-1)a_1}\rho(p)^{\nu(a_{\max}-a_1)}r^{m-(a_{\max}-a_1)} \\
   & \asymp  \widetilde{v}_1^{a_1} \rho(p)^{\nu(a_{\max}-a_1)}r^{m- (a_{\max}-a_1)}  \asymp  \widetilde{v}_1^{a_1} \rho(p)^{\nu_1(a_{\max}-a_1)}r^{m- (a_{\max}-a_1)+ (\nu-\nu_1)a_{\max}},
\end{aligned}
$$
while if $\nu_1=0$, then
$$
\begin{aligned}
\cR(p,r;a) & \asymp \rho(p)^{-a_1}r^{m-m_1}\cV(p_{Z_{H_1}},r;\widetilde{a}) \quad \mbox{by Proposition~\ref{w.24}}, \\
   & \asymp \rho(p)^{-a_1} r^{m-m_1} r^{m_1+(\nu-1)(a_{\max}-\frac{a_1}{1-\nu})} \quad \mbox{by Proposition~\ref{w.27},} \\
   & \asymp \rho(p)^{-a_1}r^{m+(\nu-1)a_{\max}+a_1} \\
   & \asymp  \widetilde{v}(p)^{a_1}r^{m-(a_{\max}-a_1)+ (\nu-\nu_1)a_{\max}},
\end{aligned}
$$
yielding the claimed result.

If instead $\frac{c\sigma^{\nu_{i+1}}(p)}{ \widetilde{v}_{i+1}(p)}\le r\le \frac{c\sigma^{\nu_i}(p)}{ \widetilde{v}_i(p)}$ with $i\ge 2$, we can combine Proposition~\ref{w.24} with \eqref{w.30b} on $Z_{H_1}$ to obtain the estimate. First, for $\nu_1=\nu$, $\sigma\asymp \rho$ and
$$
\begin{aligned}
\cR(p,r;a) & \asymp \rho(p)^{(\nu-1)a_1+ \nu m_1}r^{m-m_1}\cV\lrp{p_{Z_{H_1}},\frac{r}{\rho(p)^{\nu}};\widetilde{a}} \quad \mbox{by Proposition~\ref{w.24}}, \\
 &\asymp \rho(p)^{(\nu-1)a_1+ \nu m_1}r^{m-m_1} \lrp{ \prod_{j=2}^{i-1}x_j^{a_j-a_1} } \widetilde{v}_i^{a_i-a_1} \lrp{\frac{r}{\rho(p)^\nu}}^{m_1-(a_{\max}-a_i)} \quad \mbox{by \eqref{w.30b} with $\nu=0$,} \\
 & \asymp \rho(p)^{(\nu-1)a_1+ \nu(a_{\max}-a_i)}r^{m-(a_{\max}-a_i)} \lrp{ \prod_{j=2}^{i-1} x_j^{a_j-a_1} }\widetilde{v}_i^{a_i-a_1} \\
 & \asymp  \widetilde{v}_1^{a_1} \rho(p)^{\nu(a_{\max}-a_i)}r^{m-(a_{\max}-a_i)}  \lrp{ \prod_{j=2}^{i-1} x_j^{a_j-a_1} } \widetilde{v}_i^{a_i-a_1} \\
& \asymp  \lrp{ \lrp{\prod_{j=1}^{i-1}x_j}  \widetilde{v}_i}^{a_1}    \rho(p)^{\nu(a_{\max}-a_i)}r^{m-(a_{\max}-a_i)}  \lrp{ \prod_{j=2}^{i-1} x_j^{a_j-a_1} } \widetilde{v}_i^{a_i-a_1} \\
& \asymp \lrp{ \prod_{j=1}^{i-1} x_j^{a_j} } \widetilde{v}_i^{a_i} \rho(p)^{\nu_i(a_{\max}-a_i)} r^{m-(a_{\max}-a_i)+ (\nu-\nu_1)a_{\max}}.
\end{aligned}
$$
If instead $\nu_1=0$, then
$$
\begin{aligned}
\cR(p,r;a) & \asymp \rho(p)^{-a_1} r^{m-m_1} \cV(p,r;\widetilde{a}) \quad \mbox{by Proposition~\ref{w.24}}, \\
  & \asymp \rho(p)^{-a_1} r^{m-m_1}  \lrp{ \prod_{j=2}^{i-1} x_j^{\widetilde{a}_j} } \widetilde{v}_i^{\widetilde{a}_i} \sigma^{\nu_i(\widetilde{a}_{\max}-\widetilde{a}_i)} r^{m_1-(\widetilde{a}_{\max}-\widetilde{a}_{i})+ (\nu-\nu_i)\widetilde{a}_{\max}} \quad \mbox{by \eqref{w.30b}},
\end{aligned}
$$
and we need to distinguish two cases.  If $\nu_i=\nu$, then in fact
$$
\begin{aligned}
\cR(p,r;a) &\asymp \rho(p)^{-a_1} \lrp{ \prod_{j=2}^{i-1} x_j^{a_j- \frac{a_1}{1-\nu_j}}  }  \widetilde{v}_i^{a_i-\frac{a_1}{1-\nu}} \sigma^{\nu(a_{\max}-a_i)}r^{m-(a_{\max}-a_i)}  \\
 &\asymp \lrp{ \prod_{j=2}^{i-1} x_j^{a_j- \frac{a_1}{1-\nu_j}}  } \lrp{\prod_{j=1}^{i-1} x_j^{\frac{a_1}{1-\nu_j}}}  \widetilde{v}_i^{a_i}\sigma(p)^{\nu(a_{\max}-a_i)}r^{m-(a_{\max}-a_i)} \\
 & \asymp \lrp{\prod_{j=1}^{i-1} x_j^{a_j} }  \widetilde{v}_i^{a_i} \sigma^{\nu_i(a_{\max}-a_i)}r^{m-(a_{\max}-a_i)},
\end{aligned}
$$
yielding the claimed result.  If instead $\nu_i=0$, then
$$
\begin{aligned}
\cR(p,r;a) & \asymp \rho(p)^{-a_1} \lrp{ \prod_{j=2}^{i-1} x_j^{a_j-a_1}  }  \widetilde{v}_i^{a_i-a_1} r^{m-(a_{\max}-\frac{a_1}{1-\nu}-a_i+a_1)+\nu(a_{\max}-\frac{a_1}{1-\nu})} \\
& \asymp \lrp{ \prod_{j=1}^{i-1} x_j^{a_1} } \lrp{ \prod_{j=2}^{i-1} x_j^{a_j-a_1}  }  \widetilde{v}_i^{a_i} r^{m-(a_{\max}-a_i)+\nu a_{\max}}\\
& \asymp \lrp{ \prod_{j=1}^{i-1} x_j^{a_j}  }  \widetilde{v}_i^{a_i} r^{m-(a_{\max}-a_i)+\nu a_{\max}},
\end{aligned}
$$
again yielding the claimed result.
\end{proof}

For the moment however, just using Propositions~\ref{w.24} and \ref{w.27}, we can deduce the volume doubling property.
\begin{corollary}
Suppose that $a_H=a_{\max}$ for all $H\in\cM_{\max}(M)$, that $a_{\max}<\frac{m}{1-\nu}$ and that $(1-\nu)a_{\max}< (1-\nu_H)(a_H+m_H)$ for all $H\in \cM_{nm}(M)$.  Then $(VD)_{\mu}$ holds.  
\label{w.28}\end{corollary}
\begin{proof}
For non-remote balls, this follows from Proposition~\ref{w.27}.  For remote balls, we can assume by induction that the result holds for $\QAC$-manifolds of lower depth.  By Proposition~\ref{w.24}, $(VD)_{\mu}$ holds when we have \eqref{w.24a}, while when \eqref{w.24b} holds,
$$
 \cR(p,r;a) \asymp \rho(p)^{(\nu_H-1)a_H+\nu_H m_H} r^{b_H+1}\cV\lrp{p_{Z_H},\frac{r}{\rho(p)^{\nu_H}};\widetilde{a}}
$$
and the result follows by $(VD)_{\mu}$ on $(Z_H, \kappa_H, \mu_{\widetilde{a}})$.
\end{proof}

We can also estimate the volume of anchored balls in terms of the volume of remote balls as follows.

\begin{corollary}
Suppose that $a_H=a_{\max}$ for all $H\in\cM_{\max}(M)$, that $a_{\max}<\frac{m}{1-\nu}$ and that $(1-\nu)a_{\max}< (1-\nu_H)(a_H+m_H)$ for all $H\in \cM_{nm}(M)$.  For $c\in (0,\frac15)$ a choice of remote parameter, there exists a constant $C_V$ such that
$$
   \cA(\rho(p);a)\le C_V \cR(p,c\rho(p);a) \quad \forall \; p\in M\setminus \pa M.
$$
\label{w.29}\end{corollary}
\begin{proof}
  This follows from Corollary~\ref{w.20} with $R=\rho(p)$ and Proposition~\ref{w.27} with $r=c\rho(p)$.
\end{proof}
These estimates will help obtain the weighted Poincar\'e inequality via the following result.
\begin{corollary}
Suppose that $a_H=a_{\max}$ for all $H\in\cM_{\max}(M)$, that $a_{\max}<\frac{m}{1-\nu}$ and that $(1-\nu)a_{\max}<(1-\nu_H)(a_H+m_H)$ for all $H\in \cM_{nm}(M)$.  Suppose moreover that $\pa M$ is connected.  If the complete weighted Riemannian manifold $(M\setminus\pa M, g, \mu_a)$ satisfies $(VD)_{\mu}$ and $(PI)_{\mu,\delta}$ with parameter $\delta\in (0,1]$ for all remote balls, then $(VD)_{\mu}$ and $(PI)_{\mu,\delta}$ hold for all balls.
\label{w.31}\end{corollary}
\begin{proof}
By the previous corollaries and \cite[Theorem~5.2]{GSC2005}, it suffices to check that $(M\setminus \pa M,g)$ satisfies the property of relatively connected annuli (RCA) with respect to the base point $o$, that is, there exists $C_A>1$ such that for all $r>C^2_A$ and for all $p,q\in M\setminus \pa M$ with $d(o,p)=d(o,q)=r$, there exists a continuous path $\gamma: [0,1]\to M\setminus \pa M$ starting at $p$ and ending at $q$ with image contained in the shell $B(o,C_Ar)\setminus B(o, C_A^{-1}r).$  But since $\pa M$ is assumed to be connected, the RCA property clearly holds.   
\end{proof}

Thus, assuming $\pa M$ is connected, it remains to check that $(PI)_{\mu,\delta}$ holds on remote balls to conclude it holds on every ball.  To be able to run an argument by induction on the depth of $M$, we need to assume as well that the boundary $\pa Z_H$ of the fibers of $\phi_H: H\to S_H$ is connected for all $H\in \cM_{\nm}(M)$.  
\begin{theorem}
Suppose that $\pa M$ is connected as well as $\pa Z_H$ for each $H\in \cM_{\nm}(M)$.  Suppose also that $a_H=a_{\max}$ for all $H\in\cM_{\max}(M)$, that $a_{\max}<\frac{m}{1-\nu}$ and that $(1-\nu)a_{\max}<(1-\nu_H)( a_H+m_H)$ for all $H\in \cM_{nm}(M)$.  Then the properties $(VD)_{\mu}$ and $(PI)_{\mu}$ hold on $(M\setminus \pa M,g,\mu_a)$.
\label{w.32}\end{theorem}
\begin{proof}
We will closely follow the proof of \cite[Theorem~5.15]{CR2021}.
By the argument of Jerison \cite{Jerison1986},  if $(VD)_{\mu}$ and $(PI)_{\mu,\delta}$ hold for all balls for some $\delta\in (0,1]$, then $(PI)_{\mu}$ also holds for all balls.  Hence, by the discussion above, it suffices to check that $(PI)_{\mu,\delta}$ holds for all remote balls.  Let $c\in (0,\frac16)$ be our remote parameter and let $B(p,r)$ be a remote ball.  If $x_G(p)\ge 1-4c$ for all $G\in \cM_{\nm}(M)$, we are in a region where the metric behaves like an $\AC$-metric, so we can apply the rescaling argument of \cite[Proposition~4.20]{DM2014} to conclude that $(PI)_{\mu}$ holds on $B(p,r)$.   In particular, proceeding by induction on the depth of $M$, we can now assume that the statement of the theorem holds for warped $\QAC$-metrics on manifolds with corners of lower depth.  

If instead $x_H(p)< 1-4c$ for some $H\in \cM_{\nm}(M)$, but $x_G(p)\ge 1-4c$ for $G<H$, then as discussed in the proof of Proposition~\ref{w.24}, we can assume that $B(p,r)$ is included in a region where $g$ is of the form \eqref{w.5}.  Regarding $B(p,r)$ as a subset of $C_H\times Z_H$ with $p$ corresponding to the point $(p_1,p_2)\in C_H\times Z_H$ where $C_H=\bbR^+\times S_H$, notice that it is contained in the product of balls
\begin{equation}
  Q(r)= B_1\times B_2:= B_1(p_1,r)\times B_2\lrp{p_2, \frac{r}{((1-c)\rho(p))^{\nu_H}}},
\label{w.33}\end{equation}
where $B_2(p_2,r)$ is a geodesic ball in $(Z_H,\kappa_H)$.  Let us first prove the uniform weighted Poincar\'e inequality on $Q(r)$ by writing  $d\mu_a= d\mu_1 d\mu_2$  with
$$
     d\mu_1= \rho^{(\nu_H-1)a_H+\nu m_H}dg_{C_H} \quad \mbox{and} \quad d\mu_2= x^{\widetilde{a}}d\kappa_H.
$$
Given a function $f$ on $Q(r)$, we define the partial averages by
$$
      \overline{f}_i:= \frac{1}{\mu_i(B_i)}\int_{B_i} fd\mu_i,  \quad i\in\{1,2\}, \quad \overline{f}_Q:= \frac{1}{\mu_a(Q(r))}\int_{Q(r)}f d\mu_{a},
$$
so that in particular $\overline{f}_Q= \overline{(\overline{f}_1)}_2= \overline{(\overline{f}_2)}_1$.  Since $Z_H$ is of lower depth, we see by induction that $(PI)_{\mu_2}$ already holds on $(Z_H\setminus \pa Z_H,\kappa_H,\mu_2)$, so there is a constant $C_2>0$ such that
\begin{equation}
\begin{aligned}
 \int_{Q(r)}|f- \overline{f}_Q|^2 d\mu_1d\mu_2 &\le 2 \int_{Q(r)} (|f-\overline{f}_2|^2+ |\overline{f}_2-\overline{f}_Q|^2)d\mu_1 d\mu_2  \\
   &\le 2 \int_{B_1} \lrp{  C_2 \lrp{\frac{r}{((1-c)\rho(p))^{\nu_H}}}^2 \int_{B_2} |d_2f|^2_{\kappa_H} d\mu_2 + \int_{B_2} |\overline{f}_2-\overline{f}_Q|^2 d\mu_2           }d\mu_1  \\
   &= 2 \int_{B_1} \lrp{  C_2r^2 \int_{B_2} |d_2 f|^2_{((1-c)\rho(p))^{\nu_H} \kappa_H}d\mu_2 + \int_{B_2} |\overline{f}_2-\overline{f}_Q|^2 d\mu_2       }d\mu_1,
 \end{aligned}
\label{w.34}\end{equation}
where $d_i$ is the exterior differential taken on the factor $B_i$.  Since $(1-c)\rho(p)\le \rho(q) \le (1+c)\rho(p)$ for $q\in Q(r)$,  the first term on the right-hand side of \eqref{w.34} is bounded by
$$
       2C_2 \lrp{\frac{1+c}{1-c}}^{2\nu_H} r^2 \int_{Q(r)} |df|^2_g d\mu_a.
$$
For the second term, since $(C_H, g_{C_H})$ is a cone, we can apply $(PI)_{\mu_1}$ on $B_1$ so that there is a constant $C_1>0$ such that
\begin{equation}
\int_{Q(r)} |\overline{f}_2-\overline{f}_Q|^2 d\mu_1 d\mu_2\le \int_{B_2} \lrp{ C_1r^2 \int_{B_1} |d_1 \overline{f}_2|^2_{g_1} d\mu_1       }d\mu_2,
\label{w.35}\end{equation}
where $g_1= g_{C_H}$.  On the other hand, using the fact that
$$
       d_1 \overline{f}_2= \frac{1}{\mu_2(B_2)}\int_{B_2} (d_1 f) d\mu_2,
$$
we deduce from the Cauchy-Schwarz inequality that 
\begin{equation}
  |d_1 \overline{f}_2|^2_{g_1}\le \frac{1}{\mu_2(B_2)}\int_{B_2} |d_1 f|^2_{g_1} d\mu_2.
\label{w.36}\end{equation}
Inserting \eqref{w.36} into \eqref{w.35}, we thus obtain that
\begin{equation}
\begin{aligned}
  \int_{Q(r)} |\overline{f}_2-\overline{f}_Q|^2 d\mu_1 d\mu_2 & \le \int_{B_2}\lrp{ C_1r^2 \int_{B_1}  \lrp{ \frac{1}{\mu_2(B_2)}\int_{B_2} |d_1 f|^2_{g_1}d\mu_2  }d\mu_1  }   d\mu_2 \\
   &\le C_1 r^2 \int_{Q(r)} |d_1 f|^2_{g_1}d\mu_a  \le C_1 r^2 \int_{Q(r)} |df|^2_g d\mu_a,
\end{aligned}
\label{w.37}\end{equation}
showing that $(PI)_{\mu_a}$ holds on $Q(r)$.  For $B(p,r)$ now, notice that we have the sequence of inclusions
\begin{equation}
  B(p,r) \subset Q(r) \subset B\lrp{p, r + \lrp{\frac{1+c}{1-c}}^{\nu_H}r }\subset B(p,3r)
\label{w.38}\end{equation}
for $c\in (0,\frac16)$ sufficiently small, in which case
\begin{equation}
\begin{aligned}
\int_{B(p,r)} |f-\overline{f}_{B(p,r)}|^2d\mu_a &= \inf_c \int_{B(p,r)} |f-c|^{2}d\mu_a  \le \int_{B(p,r)} |f-\overline{f}_Q|^2 d\mu_{a} \le \int_{Q(r)} |f-\overline{f}_Q|^2 d\mu_a \\
  &\le Cr^2 \int_{Q(r)} |df|^2_g d\mu, \quad \mbox{by $(PI)_{\mu_a}$ on $Q(r)$,} \\
  &\le Cr^2 \int_{B(p,\delta^{-1}r)} |df|^2_g d\mu_a, \quad \mbox{with} \; \delta=\frac13.
\end{aligned}
\label{w.39}\end{equation}

\end{proof}
By the results of \cite{GSC2005}, this gives the following bound on the heat kernel.  
\begin{corollary}
Let $g$ be a warped $\QAC$-metric on the interior of a manifold with fibered corners $M$.  Suppose that $\pa M$ is connected as well as $\pa Z_H$ for each $H\in \cM_{\nm}(M)$.  Suppose also that $a_H=a_{\max}$ for all $H\in\cM_{\max}(M)$, that $a_{\max}<\frac{m}{1-\nu}$ and that $(1-\nu)a_{\max}< (1-\nu_H)(a_H+m_H)$ for all $H\in \cM_{nm}(M)$.  Then the heat kernel $H_{\cL}$ of $\cL=-\Delta+ \cR$ with $\cR\ge V:= \frac{\Delta(x^{\frac{a}2})}{x^{\frac{a}2}}$ satisfies the estimate 
$$
    |H_{\cL}(t,z,z')|\le H_{-\Delta+V}(t,z,z')\preceq \frac{x^{\frac{a}2}(z)x^{\frac{a}2}(z')e^{-\frac{cd(z,z')^2}t}}{  \lrp{ \mu_a(B(z,\sqrt{t}))\mu_a(B(z',\sqrt{t}))   }^{\frac12}   }
$$ 
for a positive constant $c$, where $f\preceq g$ if there exists a positive constant $C$ such that $f\le Cg$.
\label{w.40}\end{corollary}

This implies in particular the Sobolev inequality for warped $\QAC$-metrics.

\begin{corollary}[Sobolev inequality]
If $g$ is a warped $\QAC$-metric as in Corollary~\ref{w.40}, then there is a constant $C_S>0$ such that 
\begin{equation}
   \lrp{ \int_{M\setminus \pa M} |u|^{\frac{m}{\frac{m}2-1}}dg}^{1-\frac{2}m}\le C_S\int |df|^2_g dg, \quad \forall u\in \cC^{\infty}_c(M\setminus \pa M).
\label{w.41a}\end{equation}
\label{w.41}\end{corollary}
\begin{proof}
When $a=0$ and $V=0$, we can take $\cR=0$ in Corollary~\ref{w.40}, yielding a Gaussian bound for $H_{-\Delta}$, which is well-known to be equivalent to the Sobolev inequality \eqref{w.41a}; see for instance \cite{Grigoryan}.
\end{proof}

Through Lemma~\ref{w.10}, \eqref{w.9} and \cite[Theorem~3.22]{DM2014}, Corollary~\ref{w.40} yields a corresponding estimate for the Green's function, namely
\begin{equation}
  |G_{\cL}(z,z'|\le G_{-\Delta+V}(z,z')\le G_{-\Delta+V}(z,z') \asymp x(z)^{\frac{a}2}x(z')^{\frac{a}2}\int_{d(z,z')}^{\infty} \frac{sds}{\sqrt{\cV(z,s;a)\cV(z',s;a)}}.
\label{w.42}\end{equation}
This is the estimate, together with our estimates on the volume of balls, that will allow us to determine spaces on which the operator $\cL= -\Delta+\cR$ can be inverted.  In fact, what we really need is to estimate the integral
$$
  \int_{d(z,z')}^{\infty} \frac{sds}{\sqrt{\cV(z,s;a)\cV(z',s;a)}}.
$$
Since $s\ge d(z,z')$, we can use the volume doubling property as in \cite[\S~5.2]{DM2014} to conclude that this integral is equivalent to the slightly simpler integral
$$
    \bbI(z,z')= \int_{d(z,z')}^{\infty} \frac{sds}{\cV(z,s;a)}.
$$
Using Propositions~\ref{w.27} and \ref{w.30}, this simpler integral can be estimated as follows.
\begin{lemma}
Let $M$, $g$ and $a$ be as in Corollary~\ref{w.40} and let $c\in(0,\frac15)$ be a choice of remote parameter.  Suppose also that $(1-\nu)a_{\max}<m-2$ and $m>2$.  Let $z,z'\in M\setminus \pa M$ be given.  If $d(z,z')>c\rho(z)$, then 
$$
     \bbI(z,z')\asymp d(z,z')^{2-m-(\nu-1)a_{\max}}.
$$
If instead $d(z,z')<c\rho(z)$ and $z$ is far from all non-maximal corners, then 
$$
   \bbI(z,z') \asymp \rho(z)^{-(\nu-1)a_{\max}} d(z,z')^{2-m}.
$$
Finally, if  $d(z,z')<c\rho(z)$ and $z$ is close to the non-maximal corner $H_1\cap \ldots \cap H_k$ with $\frac{c\sigma^{\nu_{i+1}(z)}}{v_{i+1}(z)}\le d(z,z')\le \frac{\sigma^{\nu_i}(z)}{v_i(z)}$ for some $i\in \{1,\ldots,k+1\}$ (using the conventions of Proposition~\ref{w.30}), then
$$
\bbI(z,z')\asymp \lrp{ \prod_{j=1}^{i-1} x_j(z)^{-a_j}  } \widetilde{v}_i(z)^{-a_i} \sigma(z)^{\nu_i(a_i-a_{\max})} d(z,z')^{2-m+a_{\max}-a_i-(\nu-\nu_i)a_{\max}}.
$$
\label{w.43}\end{lemma}

In terms of \eqref{w.42}, this gives the following estimate.

\begin{proposition}
Let $M$, $g$ and $a$ be as in Lemma~\ref{w.43} and let $c\in(0,\frac15)$ be a choice of remote parameter.  Let $z,z'\in M\setminus \pa M$ be given.  If $d(z,z')>c\rho(z)$, then 
\begin{equation}
     G_{-\Delta+V}(z,z')\asymp x(z)^{\frac{a}2} x(z')^{\frac{a}2}  d(z,z')^{2-m-(\nu-1)a_{\max}}.
\label{w.44a}\end{equation}
If instead $d(z,z')<c\rho(z)$ and $z$ is far from all non-maximal corners, then 
\begin{equation}
   G_{-\Delta+V}(z,z') \asymp x(z)^{\frac{a}2} x(z')^{\frac{a}2} \rho(z)^{-(\nu-1)a_{\max}} d(z,z')^{2-m}.
\label{w.44b}\end{equation}
Finally, if  $d(z,z')<c\rho(z)$ and $z$ is close to the non-maximal corner $H_1\cap \ldots \cap H_k$ with $\frac{c\sigma^{\nu_{i+1}}(z)}{v_{i+1}(z)}\le d(z,z')\le \frac{c\sigma^{\nu_i}(z)}{v_i(z)}$ for some $i\in \{1,\ldots,k+1\}$ (using the conventions of Proposition~\ref{w.30}), then
\begin{equation}
G_{-\Delta+V}(z,z')\asymp  x(z)^{\frac{a}2} x(z')^{\frac{a}2}  \lrp{ \prod_{j=1}^{i-1} x_j(z)^{-a_j}  } \widetilde{v}_i(z)^{-a_i} \sigma(z)^{\nu_i(a_i-a_{\max})} d(z,z')^{2-m+a_{\max}-a_i-(\nu-\nu_i)a_{\max}}.
\label{w.44c}\end{equation}

\label{w.44}\end{proposition}

As in \cite{DM2014}, we can extract mapping properties from these estimates using the Schur test.  This amounts to obtaining the following estimates.
\begin{proposition}
Let $M$, $g$ and $a$ be as in Lemma~\ref{w.43}.  Let $b$ be a multiweight with $b_H=b_{\max}$ for all $H\in \cM_{\max}(M)$.  Suppose that 
$$
       \frac{2}{1-\nu}+ \frac{a_{\max}}{2} < b_{\max} < \frac{m}{1-\nu} -\frac{a_{\max}}2
$$
and that
$$
           \lrp{\frac{1-\nu}{1-\nu_H} } \lrp{\frac{a_{\max}}2+ b_{\max}} -\frac{a_H}2-m_H< b_H < \frac{a_H}2+ \lrp{\frac{1-\nu}{1-\nu_H}}\lrp{b_{\max}-\frac{a_{\max}}2}-2
$$
for all $H\in\cM_{\nm}(M)$.  Then
\begin{equation}
   \int_{M\setminus \pa M} G_{-\Delta+V}(z,z') x(z')^{b}dg(z') \preceq x(z)^b(\rho(z)w(z))^2,
\label{w.45a}\end{equation}
where $w= \prod_{H\in \cM_{\nm}(M)} x_H$.
\label{w.45}\end{proposition}
\begin{proof}
We follow the strategy of the proof of \cite[Lemma~6.5]{DM2014}.  Thus, fix a remote parameter $c\in (0,\frac15)$ and decompose the region of integration into three subdomains: $\cD_1:=(M\setminus \pa M)\setminus B(o,2\rho(z))$, $\cD_2:=B(o,2\rho(z))\setminus B(z,c\rho(z))$ and $\cD_3:=B(z,c\rho(z))$, the latter being a maximal remote ball.  On the first region, $d(z,z')\asymp \rho(z')$, so by \eqref{w.44a},
\begin{equation}
\begin{aligned}
\int_{\cD_1} G_{-\Delta+V}(z,z')x(z')^b dg(z') &\preceq x(z)^{\frac{a}2} \int_{\cD_1} x(z')^{\frac{a}2+b}d(z,z')^{2-m-(\nu-1)a_{\max}}dg(z') \\
  &\preceq x(z)^{\frac{a}2}\int_{\cD_1} x(z')^{\frac{a}2+b} \rho(z')^{2-m-(\nu-1)a_{\max}}dg(z') \\
   & \preceq x(z)^{\frac{a}2} \int_{2\rho(z)}^{\infty} d\cA\lrp{s; \frac{a}2+b+ \frac{m-2-(1-\nu)a_{\max}}{1-\mathfrak{n}}}  \\
   & \preceq x(z)^{\frac{a}2} \int_{2\rho(z)}^{\infty} s^{m-1} \lrp{ \sum_{H\in \cM_1(M)} s^{(\nu_H-1)(\frac{a_H}2+b_H+ \frac{m-2-(1-\nu)a_{\max}}{1-\nu_H}+ m_H)} }ds,
\end{aligned}
\label{w.46}\end{equation}
where in the last line we have used Proposition~\ref{w.17}.  By our assumption on $b$, we see that 
$$
    m-1+ (\nu-1)\lrp{b_{\max}+ \frac{m-2}{1-\nu}-\frac{a_{\max}}2}<-1
$$
and
$$
      m-1+ (\nu_H-1)\lrp{ \frac{a_H}2+ b_H+ \frac{m-2-(1-\nu)a_{\max}}{1-\nu_H}+m_H  }< m-1+ (\nu-1)\lrp{b_{\max}+ \frac{m-2}{1-\nu}-\frac{a_{\max}}2}<-1,
$$
so the integral at the very end of \eqref{w.46} does not diverge and
\begin{equation}
\begin{aligned}
\int_{\cD_1} G_{-\Delta+V}(z,z')x(z')^b dg(z') &\preceq x(z)^{\frac{a}2}\rho(z)^{m+(\nu-1)(b_{\max}-\frac{a_{\max}}2+\frac{m-2}{1-\nu})} \\
              &\preceq x(z)^{\frac{a}2}\rho(z)^{2+ (\nu-1)(b_{\max}-\frac{a_{\max}}2)} \\
                & \preceq x(z)^{\frac{\widehat{a}}2} \rho(z)^{2+(\nu-1)b_{\max}} , \quad \mbox{where} \; \widehat{a}_H= a_H-\lrp{\frac{1-\nu}{1-\nu_H}}a_{\max}.
\end{aligned}
\label{w.47}\end{equation}
But if $\widehat{b}$ is the multiweight such that $\widehat{b}_H=b_H-\lrp{\frac{1-\nu}{1-\nu_H}}b_{\max}$, then we see from our assumption on $b$ that
\begin{equation}
       \widehat{b}_H+2< \frac{\widehat{a}_H}2 \quad \mbox{for} \quad H\in \cM_{\nm}(M),
\label{w.48}\end{equation}
so that
$$
\begin{aligned}
\int_{\cD_1} G_{-\Delta+V}(z,z')x(z')^b dg(z') & \preceq  x(z)^{\widehat{b}}w^2\rho(z)^{2+(\nu-1)b_{\max}}= x(z)^{b} (\rho(z)w(z))^2
\end{aligned}
$$
as claimed.  In the second region, $d(z,z')\asymp \rho(z)$, so again, by \eqref{w.44a}, 
$$
\begin{aligned}
\int_{\cD_2} G_{-\Delta+V}(z,z') x(z')^{b}dg(z') &\preceq x(z)^{\frac{a}2}\int_{\cD_2} d(z,z')^{2-m-(\nu-1)a_{\max}} x(z')^{\frac{a}2+b}dg(z') \\
     &\preceq x^{\frac{a}2}(z) \rho(z)^{2-m-(\nu-1)a_{\max}} \int_{B(o,2\rho(z))} x(z')^{\frac{a}2+b}dg(z') \\
     &\preceq x(z)^{\frac{a}2}\rho(z)^{2-m-(\nu-1)a_{\max}}\cA\lrp{\rho(z), \frac{a}2+b}.
\end{aligned}
$$
Now, by our assumption on $b$, the multiweight $\frac{a}2+b$ satisfies the assumption of Corollary~\ref{w.20}, so
$$
\begin{aligned}
\int_{\cD_2} G_{-\Delta+V}(z,z') x(z')^{b}dg(z') &\preceq x(z)^{\frac{a}2}\rho(z)^{2-m-(\nu-1)a_{\max}} \rho(z)^{m+(\nu-1)(\frac{a_{\max}}2+b_{\max})} \\
  & \preceq  x(z)^{\frac{a}2}\rho(z)^{2+(\nu-1)(b_{\max}-\frac{a_{\max}}2)} \\
  & \preceq x(z)^{\frac{\widehat{a}}2}\rho(z)^{2+ (\nu-1)b_{\max}}  \\
  & \preceq x(z)^{\widehat{b}}w(z)^2\rho(z)^{2+ (\nu-1)b_{\max}} \quad \mbox{by \eqref{w.48}}, \\
  & \preceq x(z)^{b} (\rho(z)w(z))^2.
\end{aligned}
$$
Finally, in the last region, suppose that $z$ is close to the non-maximal corner $H_1\cap\ldots\cap H_k$.  Then by \eqref{w.44c}, setting $B_i:= B(z,\frac{c\sigma(z)^{\nu_i}}{v_i(z)})$, we see that
\begin{equation}
\int_{\cD_3} G_{-\Delta+V}(z,z')dg(z')  \asymp \sum_{i=1}^{k+1} I_i 
\label{w.49}\end{equation}
with
$$
   I_{k+1}= \int_{B_{k+1}} x(z)^{\frac{a}2}x(z')^{\frac{a}2} \lrp{ \prod_{j=1}^k x_j(z)^{-a_j}   } x_{\max}(z)^{-a_{\max}} d(z,z')^{2-m} x(z')^{b}dg(z')
$$
and
$$
    I_i= \int_{B_i\setminus B_{i+1}} x(z)^{\frac{a}2}x(z')^{\frac{a}2} \lrp{\prod_{j=1}^{i-1}x_j(z)^{-a_j}}  \widetilde{v}_i(z)^{-a_i}\sigma(z)^{\nu_i(a_i-a_{\max})} d(z,z')^{2-m+a_{\max}-a_i-(\nu-\nu_i)a_{\max}} x(z')^b dg(z')
$$
for $i\le k$.  For $I_{k+1}$, we can use our assumption on $b$ and \eqref{w.30b} with $i=k+1$ to obtain
$$
\begin{aligned}
I_{k+1} & \preceq x(z)^{\frac{a}2} \lrp{ \prod_{j=1}^k x_{j}(z)^{-a_j}     } x_{\max}(z)^{-a_{\max}}  \int_0^{\frac{c\sigma(z)^{\nu}}{x_{\max}(z)}}  r^{2-m}d\cR(z,r; \frac{a}2+b)  \\
    & \preceq x(z)^{\frac{a}2} \lrp{ \prod_{j=1}^k x_{j}(z)^{-a_j}     } x_{\max}(z)^{-a_{\max}} \int_0^{\frac{c\sigma(z)^{\nu} }{x_{\max}(z)}} r^{1-m} \lrp{ \prod_{j=1}^{k} x_j(z)^{\frac{a_j}2+b_j}  } x_{\max}(z)^{\frac{a_{\max}}2+ b_{\max}} r^m dr  \\
    & \preceq x(z)^{\frac{a}2} \lrp{ \prod_{j=1}^{k} x_j(z)^{-\frac{a_j}2+b_j}  } x_{\max}(z)^{-\frac{a_{\max}}2+b_{\max}} \int_0^{\frac{c\sigma(z)^{\nu}}{x_{\max}(z)}} r dr \\
    & \preceq x(z)^{\frac{a}2} \lrp{ \prod_{j=1}^{k} x_j(z)^{-\frac{a_j}2+b_j}  } x_{\max}(z)^{-\frac{a_{\max}}2+b_{\max}}  \lrp{\frac{\sigma(z)^{\nu}}{x_{\max}(z)}}^2 \asymp x(z)^b\sigma(z)^2\lrp{ \frac{\sigma(z)^{\nu-1}}{x_{\max}(z)} }^2 \\
    & \preceq x(z)^{b} (\rho(z)w(z))^2.
\end{aligned}
$$

For $i\le k$, again using our hypothesis on $b$, we see from Proposition~\ref{w.30} that 
$$
\begin{aligned}
I_i & \preceq x^{\frac{a}2}(z) \lrp{ \prod_{j=1}^{i-1} x_j(z)^{-a_j}   }  \widetilde{v}_i(z)^{-a_i}\sigma(z)^{\nu_i(a_i-a_{\max})} \int_{\frac{c\sigma(z)^{\nu_{i+1}}}{v_{i+1}}}^{\frac{c\sigma(z)^{\nu_i}}{v_{i}}} r^{2-m+a_{\max}-a_i-(\nu-\nu_i)a_{\max}}d\cR\lrp{z,r,b+\frac{a}2}  \\
     & \preceq  x^{\frac{a}2}(z) \lrp{ \prod_{j=1}^{i-1} x_j(z)^{-a_j}   }  \widetilde{v}_i(z)^{-a_i}\sigma(z)^{\nu_i(a_i-a_{\max})} 
       \int_{\frac{c\sigma(z)^{\nu_{i+1}}}{v_{i+1}}}^{\frac{c\sigma(z)^{\nu_i}}{v_{i}}}  r^{1-m+a_{\max}-a_i-(\nu-\nu_i)a_{\max}}  \cdot \\ 
       & \hspace{1cm}\lrp{  \lrp{ \prod_{j=1}^{i-1}x_j(z)^{\frac{a_j}2+b_j} }  \widetilde{v}_i(z)^{\frac{a_i}2+b_i}\sigma(z)^{\nu_i(\frac{a_{\max}}2 + b_{\max}- \frac{a_i}2-b_i)} r^{m- (\frac{a_{\max}}2+b_{\max}-\frac{a_i}2-b_i)+ (\nu-\nu_i)(\frac{a_{\max}}2+b_{\max})}  }dr \\
     & \preceq  x^{\frac{a}2}(z) \lrp{ \prod_{j=1}^{i-1} x_j(z)^{-\frac{a_j}2+ b_j}  } \widetilde{v}_i(z)^{-\frac{a_i}2+b_i}\sigma(z)^{\nu_i(\frac{a_i}2- \frac{a_{\max}}2+ b_{\max}-b_i)}  \cdot \\
     & \hspace{1cm} \int_{\frac{c\sigma(z)^{\nu_{i+1}}}{v_{i+1}}}^{\frac{c\sigma(z)^{\nu_i}}{v_{i}}}
        r^{1+ \frac{a_{\max}}2-\frac{a_i}2-b_{\max}+ b_i+(\nu-\nu_i)(b_{\max}-\frac{a_{\max}}2)} dr  \\
       & \preceq x^{\frac{a}2}(z) \lrp{ \prod_{j=1}^{i-1} x_j(z)^{-\frac{a_j}2+ b_j}  } \widetilde{v}_i(z)^{-\frac{a_i}2+b_i}\sigma(z)^{\nu_i(\frac{a_i}2- \frac{a_{\max}}2+ b_{\max}-b_i)} \; \cdot \\
          &  \hspace{1cm} \left| \lrp{ \frac{\sigma(z)^{\nu_i}}{ \widetilde{v}_i(z)}  }^{2+\frac{a_{\max}}2 -\frac{a_i}2-b_{\max} + b_i +(\nu-\nu_i)(b_{\max}-\frac{a_{\max}}2) }  -   \lrp{ \frac{\sigma(z)^{\nu_{i+1}}}{ \widetilde{v}_{i+1}(z)}  }^{2+\frac{a_{\max}}2 -\frac{a_i}2-b_{\max} + b_i +(\nu-\nu_i)(b_{\max}-\frac{a_{\max}}2) }  \right|.        
\end{aligned}
$$
Using the facts that $\frac{c\sigma^{\nu_{i+1}}(z)}{ \widetilde{v}_{i+1}(z)}\le \frac{c\sigma^{\nu_i}(z)}{ \widetilde{v}_i(z)}$ and that $ 2+\frac{a_{\max}}2 -\frac{a_i}2-b_{\max} + b_i +(\nu-\nu_i)(b_{\max}-\frac{a_{\max}}2) <0$, this implies that
\begin{multline}
I_i   \preceq  x^{\frac{a}2}(z) \lrp{ \prod_{j=1}^{i-1} x_j(z)^{-\frac{a_j}2+ b_j}  } \widetilde{v}_i(z)^{-\frac{a_i}2+b_i}\sigma(z)^{\nu_i(\frac{a_i}2- \frac{a_{\max}}2+ b_{\max}-b_i)}\;\cdot \\ \lrp{ \frac{\sigma(z)^{\nu_{i+1}}}{ \widetilde{v}_{i+1}(z)}  }^{2+\frac{a_{\max}}2 -\frac{a_i}2-b_{\max} + b_i +(\nu-\nu_i)(b_{\max}-\frac{a_{\max}}2)  }.
\end{multline}
If $\nu_i=\nu$, this gives
$$
\begin{aligned}
I_i & \preceq x(z)^{\frac{a}2} \lrp{ \prod_{j=1}^{i-1} x_j(z)^{-\frac{a_j}2+b_j}   }  \widetilde{v}_i(z)^{-\frac{a_i}2+b_i} \sigma(z)^{2\nu}  \widetilde{v}_{i+1}(z)^{-2-\frac{a_{\max}}2+ \frac{a_i}2+ b_{\max}-b_i} \\
 & \preceq  x(z)^{\frac{a}2} \lrp{ \prod_{j=1}^{i-1} x_j(z)^{-\frac{a_j}2+ b_j}  } x_i(z)^{-\frac{a_i}2+ b_i}  \sigma(z)^{2\nu}  \widetilde{v}_{i+1}(z)^{-2-\frac{a_{\max}}2+ b_{\max}} \\
 & \preceq x(z)^{\frac{a}2} \lrp{ \prod_{j=1}^i x_j(z)^{-\frac{a_j}2+ b_j}   }\sigma(z)^{2\nu}  \widetilde{v}_{i+1}(z)^{-2-\frac{a_{\max}}2+b_{\max}} \\
 & \preceq \lrp{ \prod_{j=i+1}^{k+1} x_j(z)^{\frac{a_j}2-b_j}  }x(z)^{b} \sigma(z)^2 \lrp{ \frac{\sigma(z)^{\nu-1}}{x_{\max}}  }^2 x_{\max}(z)^2  \widetilde{v}_{i+1}(z)^{-2-\frac{a_{\max}}2+b_{\max}} \\
 & \asymp x(z)^b (\rho(z)w(z))^2    \lrp{ \prod_{j=i+1}^{k+1} x_j(z)^{\frac{a_j}2-b_j}  }   x_{\max}(z)^2  \widetilde{v}_{i+1}(z)^{-2-\frac{a_{\max}}2+b_{\max}} \\
 & \preceq x(z)^b (\rho(z)w(z))^2   \lrp{  \prod_{j=i+1}^{k+1} x_j(z)^{\frac{a_j}2-b_j-2-\frac{a_{\max}}2+ b_{\max}}  } x(z)^2_{\max}  \\
 & \preceq x(z)^b (\rho(z)w(z))^2, \quad \mbox{by \eqref{w.48}},
\end{aligned}
$$
completing the proof in this case.  If instead $\nu_i=0$ and $\nu_{i+1}=\nu$, then
$$
\begin{aligned}
I_i &\preceq x(z)^{\frac{a}{2}}\lrp{ \prod_{j=1}^{i-1} x_j(z)^{-\frac{a_j}2+b_j} }  \widetilde{v}_i^{-\frac{a_i}2+b_i} \lrp{\frac{\sigma(z)^{\nu}}{ \widetilde{v}_{i+1}(z)}}^{2+ \frac{a_{\max}}2-\frac{a_i}2-b_{\max}+b_i+ \nu\lrp{b_{\max}-\frac{a_{\max}}2}} \\
    & \asymp x(z)^{\frac{a}{2}}\lrp{ \prod_{j=1}^{i-1} x_j(z)^{-\frac{a_j}2+b_j} } \lrp{ x_i^{-\frac{a_i}2+b_i}  ( \widetilde{v}_{i+1}^{\frac{1}{1-\nu}})^{-\frac{a_i}2+b_i}  }   \lrp{ \widetilde{v}_{i+1}^{-\frac{1}{1-\nu}}}^{2+ \frac{a_{\max}}2-\frac{a_i}2-b_{\max}+b_i+ \nu\lrp{b_{\max}-\frac{a_{\max}}2}}  \\
    &\asymp x(z)^{\frac{a}{2}}\lrp{ \prod_{j=1}^{i} x_j(z)^{-\frac{a_j}2+b_j} } ( \widetilde{v}_{i+1}^{-\frac{1}{1-\nu}})^{2+ (\nu-1)(b_{\max}-\frac{a_{\max}}2)} \\
    &\asymp x(z)^{b} \lrp{ \prod_{j=i+1}^{k+1} x_j(z)^{\frac{a_j}2-b_j} }  \widetilde{v}_{i+1}(z)^{b_{\max}-\frac{a_{\max}}{2}-\frac{2}{1-\nu}}  \asymp x(z)^{b} \lrp{ \prod_{j=i+1}^{k+1} x_j(z)^{\frac{\widehat{a}_j}2-\widehat{b}_j} }  \widetilde{v}_{i+1}(z)^{-\frac{2}{1-\nu}} \\
    &\preceq  x(z)^{b}  \widetilde{v}_{i+1}(z)^{2}    \widetilde{v}_{i+1}(z)^{-\frac{2}{1-\nu}} \quad \mbox{by \eqref{w.48}}, \\
    & \asymp x(z)^b ( \widetilde{v}_{i+1}(z)^{-\frac{\nu}{1-\nu}})^2 \preceq x(z)^b (x_{\max}(z)^{-1} \widetilde{v}_{i+1}(z)^{-\frac{\nu}{1-\nu}})^2 \asymp x(z)^b (\rho(z) w(z))^2,
\end{aligned}
$$
while if $\nu_i=\nu_{i+1}=0$, then
$$
\begin{aligned}
I_i &\preceq x(z)^{\frac{a}{2}}\lrp{ \prod_{j=1}^{i-1} x_j(z)^{-\frac{a_j}2+b_j} }  \widetilde{v}_i(z)^{-\frac{a_i}2+b_i} \lrp{\frac{1}{ \widetilde{v}_{i+1}(z)}}^{2+ \frac{a_{\max}}2-\frac{a_i}2-b_{\max}+b_i+ \nu\lrp{b_{\max}-\frac{a_{\max}}2}} \\
    & \asymp x(z)^{\frac{a}{2}}\lrp{ \prod_{j=1}^{i-1} x_j(z)^{-\frac{a_j}2+b_j} } x_i(z)^{-\frac{a_i}2+b_i}  \widetilde{v}_{i+1}(z)^{-\frac{a_i}2+b_i}  \widetilde{v}_{i+1}(z)^{-2- \frac{a_{\max}}2+\frac{a_i}2+b_{\max}-b_i- \nu\lrp{b_{\max}-\frac{a_{\max}}2}} \\
    &\asymp x(z)^{\frac{a}{2}}\lrp{ \prod_{j=1}^{i} x_j(z)^{-\frac{a_j}2+b_j} }  \widetilde{v}_{i+1}(z)^{-2+(1-\nu)\lrp{b_{\max}-\frac{a_{\max}}2}} \asymp x(z)^b \lrp{ \prod_{j=i+1}^{k+1} x_j(z)^{\frac{\widehat{a}_j}2-\widehat{b}_j} }  \widetilde{v}_{i+1}^{-2} \\
    & \preceq x(z)^b \lrp{ \prod_{j=i+1}^{k+1} x_j(z)^{2} }  \widetilde{v}_{i+1}^{-2} \quad \mbox{by \eqref{w.48},} \\
    & \asymp  x(z)^b \lrp{ \prod_{j=i+1}^{k+1} x_j(z)^{2- \frac{2}{1-\nu_j}} } \asymp  x(z)^b \lrp{ \prod_{j=i+1}^{k+1} x_j(z)^{-\frac{\nu_j}{1-\nu_j}} }^2 \preceq x(z)^b \lrp{x_{\max}^{-1}(z) \prod_{j=i+1}^{k+1} x_j(z)^{-\frac{\nu_j}{1-\nu_j}} }^2 \\
    &\asymp x(z)^b(\rho(z) w(z))^2.
\end{aligned}
$$
Finally, if $z$ is away from all non-maximal corners, we can apply the estimate for $I_{k+1}$ with $k=0$.

\end{proof}

Relying on this estimate, we can now apply the Schur test to obtain the main result of this section.

\begin{theorem}
Let $g$ be an $\mathfrak{n}$-warped $\QAC$-metric of weight function $\mathfrak{n}:\cM_1(M)\to \{0,\nu\}$ for $\nu\in [0,1)$ on the interior of a manifold with fibered corners $M$ of dimension $m>2$.  Suppose that $\pa M$ is connected, as well as $\pa Z_H$ for each $H\in \cM_{\nm}(M)$.  Suppose that $a$ is a multiweight such that $a_{H}= a_{\max}$ for $H\in \cM_{\max}(M)$, that $a_{\max}<\frac{m-2}{1-\nu}$ and that $(1-\nu)a_{\max}< (1-\nu_H)(a_H+ m_H)$ for all $H\in \cM_{\nm}(M)$.  Let also $\delta$ be a multiweight such that $\delta_H=\delta_{\max}$ for all $H\in \cM_{\max}(M)$ with 
\begin{equation}
   \frac{a_{\max}}2< \delta_{\max}< \frac{m-2}{1-\nu}- \frac{a_{\max}}2
\label{w.50a}\end{equation}
and
\begin{equation}
  \lrp{\frac{1-\nu}{1-\nu_H}}\lrp{\delta_{\max}+ \frac{a_{\max}}2} - \frac{a_H}2+ 2-m_H< \delta_H < \frac{a_H}2 + \lrp{ \frac{1-\nu}{1-\nu_H} } \lrp{-\frac{a_{\max}}2+ \delta_{\max}}.
\label{w.50b}\end{equation}
If $\cR\in \rho^{-2}w^{-2}\CI_{\nQb}(M\setminus \pa M)$ is such that $\cR\ge V:= -\frac{-\Delta x^a}{x^a}$, then for all $\ell\in\bbN_0$ and $\alpha\in (0,1)$, the mappings
\begin{equation}
\begin{gathered}
  -\Delta+\cR: x^{\delta+\mathfrak{w}}\sigma^{\frac{\nu m}2}H^{\ell+2}_w(M)\to (\rho w)^{-2}x^{\delta+\mathfrak{w}}\sigma^{\frac{\nu m}2}H^{\ell}_w(M), \\
   -\Delta+\cR: x^{\delta}\cC^{\ell+2,\alpha}_{\nQb}(M\setminus \pa M)\to (\rho w)^{-2} x^{\delta}\cC^{\ell,\alpha}_{\nQb}(M\setminus \pa M),
\end{gathered}
\label{w.50c}\end{equation}
are isomorphisms, where  $H^{\ell}_w(M)$ was introduced in \eqref{Sob.1} and $\mathfrak{w}(H):=\frac{m_H-m}2$ is the multiweight such that $x^{\mathfrak{w}}\sigma^{\frac{\nu m}2}L^2_w(M)=L^2_b(M)$.
\label{w.50}\end{theorem}
\begin{proof}

For the mapping on Sobolev spaces, it suffices by local elliptic estimates to show that the Green's function $G_{\cL}$ defines a bounded map
\begin{equation}
  G_{\cL}: (\rho w)^{-2}x^{\delta+\mathfrak{w}}\sigma^{\frac{\nu m}2}L^2_w(M)\to x^{\delta+ \mathfrak{w}}\sigma^{\frac{\nu m}2}L^2_w(M).
\label{w.51}\end{equation}
This in turn corresponds to the boundedness of 
\begin{equation}
 \cK: L^2_w(M) \to L^2_w(M)
\label{w.52}\end{equation}
with
$$
  \cK(z,z')= x(z)^{-\delta-\mathfrak{w}} \sigma(z)^{-\frac{\nu m}2} G_{\cL}(z,z') (\rho(z')w(z'))^{-2} x(z')^{\delta+\mathfrak{w}}\sigma(z')^{\frac{\nu m}2}.
$$
By the Schur test, this will be the case provided we can find positive measurable functions $f_1$ and $f_2$ such that
$$
  \left|  \int_{M\setminus \pa M} \cK(z,z') f_1(z) dg(z)  \right| \preceq f_2(z') \quad \mbox{and}  \quad \left|  \int_{M\setminus \pa M} \cK(z,z') f_2(z') dg(z')  \right| \preceq f_1(z)
$$
for all $z,z'\in M\setminus \pa M$.  We will take $f_1=f_2= x^{-\mathfrak{w}}\sigma^{-\frac{m\nu}2}$.  Indeed, by Lemma~\ref{w.10}, we have that 
$$
\begin{aligned}
  \left| \int_{M\setminus \pa M}  \cK(z,z') x^{-\mathfrak{w}}(z') \sigma(z')^{-\frac{m\nu}2} dg(z') \right| & \le x(z)^{-\delta-\mathfrak{w}}\sigma(z)^{-\frac{\nu m}2} \int_{M\setminus \pa M} G_{-\Delta+ V} x(z')^{\delta} (\rho(z')w(z'))^{-2}dg(z') \\
      & \preceq  x(z)^{-\delta-\mathfrak{w}} \sigma(z)^{-\frac{\nu m}2} \int_{M\setminus \pa M} G_{-\Delta+V} x(z')^b dg(z')
\end{aligned}  
$$
with 
$$
  x^{b}:= x^{\delta} (\rho w)^{-2}= \lrp{ \prod_{H\in \cM_1(M)} (x_H^{\delta_H+ \frac{2}{1-\nu_H}}) } w^{-2},
$$
so in particular with multiweight $b$ such that $b_{\max}= \delta_{\max}+ \frac{2}{1-\nu}$ and $b_H= \delta_H+ \frac{2}{1-\nu_H}-2$.  In particular, by our assumption on $\delta$, we can apply Proposition~\ref{w.45} so that
$$
\begin{aligned}
 \left| \int_{M\setminus \pa M}  \cK(z,z') x^{-\mathfrak{w}}(z') \rho(z')^{-\frac{m\nu}2} dg(z') \right| & \preceq x(z)^{-\delta-\mathfrak{w}}\sigma(z)^{-\frac{\nu m}2}x(z)^{b} (\rho(z) w(z))^2 \\
   & \preceq x(z)^{-\mathfrak{w}}\sigma(z)^{-\frac{\nu m}2}.
\end{aligned}
$$ 
Similarly,
$$
\begin{aligned}
\left| \int_{M\setminus \pa M}  \cK(z,z') x^{-\mathfrak{w}}(z) \sigma(z)^{-\frac{m\nu}2} dg(z) \right| &\le (\rho(z')w(z'))^{-2}x(z')^{\delta+ \mathfrak{w}}\sigma(z')^{\frac{\nu m}2} \cdot \\
 & \hspace{3cm} \int_{M\setminus \pa M} G_{-\Delta+V}(z,z') x(z)^{-\delta-2\mathfrak{w}}\sigma(z)^{-m\nu} dg(z) \\
  &= (\rho(z')w(z'))^{-2}x(z')^{\delta+ \mathfrak{w}}\sigma(z')^{\frac{\nu m}2} \int_{M\setminus \pa M} G_{-\Delta+V}(z,z') x(z)^{b} dg(z)
\end{aligned}
$$
with this time multiweight $b$ such that
$$
     b_H:= m-m_H + \frac{m\nu_H}{1-\nu_H} -\delta_H= \frac{m}{1-\nu_H}-m_H-\delta_H,
$$
that is, such that
$$
   x^b= x^{-\delta-2\mathfrak{w}}\sigma(z)^{-m\nu} = \prod_H x_H^{-\delta_H-m_H+m + \frac{m\nu_H}{1-\nu_H}}.
$$
By our assumptions on $\delta$ and since $G_{-\Delta+V}(z,z')=G_{-\Delta+V}(z',z)$, we can apply Proposition~\ref{w.45} to conclude that
$$
\begin{aligned}
\left| \int_{M\setminus \pa M}  \cK(z,z') x^{-\mathfrak{w}}(z) \sigma(z)^{-\frac{m\nu}2} dg(z) \right| & \preceq (\rho(z')w(z'))^{-2} x(z')^{\delta+\mathfrak{w}}\sigma(z')^{\frac{\nu m}2} x(z')^{b} (\rho(z')w(z'))^{2} \\
 & \preceq x(z')^{-\mathfrak{w}}\sigma(z')^{-\frac{\nu m}2}.
\end{aligned}
$$
By the Schur test, the mapping \eqref{w.51} is indeed bounded and the result follows.  For the map on H\"older spaces, we can use our estimates on Green's functions essentially as above to show that the Green's function $G_{\cL}$ induces a bounded mapping 
$$
    G_{\cL}: x^{\delta} (\rho w)^{-2}\cC^{\ell,\alpha}_{\nQb}(M\setminus \pa M)\to x^{\delta}L^{\infty}(M\setminus \pa M).
$$
Using Schauder estimates and the fact that an $\nQb$-metric has bounded geometry, we then see that in fact it defines an inverse 
$$
G_{\cL}: x^{\delta} (\rho w)^{-2}\cC^{\ell,\alpha}_{\nQb}(M\setminus \pa M)\to  x^{\delta} \cC^{\ell+2,\alpha}_{\nQb}(M\setminus \pa M)
$$
for the map
$$
    -\Delta+ \cR: x^{\delta} \cC^{\ell+2,\alpha}_{\nQb}(M\setminus \pa M)\to (\rho w)^{-2} x^{\delta} \cC^{\ell,\alpha}_{\nQb}(M\setminus \pa M).
$$ \end{proof}

As in \cite{CR2021}, we will be mostly interested in the case where $a=0$ and $\cR=0$, in which case we have the following.
\begin{corollary}
Let $g$ be an $\mathfrak{n}$-warped $\QAC$-metric of weight function $\mathfrak{n}:\cM_1(M)\to \{0,\nu\}$ for some $\nu\in [0,1)$ on the interior of a manifold with fibered corners $M$.  Suppose that $\pa M$ is connected, as well as $\pa Z_H$ for each $H\in \cM_{\nm}(M)$.  Let $\delta$ be a multiweight such that $\delta_H=\delta_{\max}$ for all $H\in \cM_{\max}(M)$.  Then for all $\ell\in\bbN_0$ and $\alpha\in (0,1)$, the mappings 
\begin{equation}
\begin{gathered}
  \Delta: x^{\delta+\mathfrak{w}}\rho^{\frac{\nu m}2}H^{\ell+2}_w(M)\to (\rho w)^{-2}x^{\delta+\mathfrak{w}}\rho^{\frac{\nu m}2}H^{\ell}_w(M), \\
   \Delta: x^{\delta}\cC^{\ell+2,\alpha}_{\Qb}(M\setminus \pa M)\to (\rho w)^{-2} x^{\delta}\cC^{\ell,\alpha}_{\Qb}(M\setminus \pa M),
\end{gathered}
\label{w.53a}\end{equation}
are isomorphisms provided
$$
  0<\delta_{\max}< \frac{m-2}{1-\nu}\quad \mbox{and} \quad \lrp{\frac{1-\nu}{1-\nu_H}}\delta_{\max}+2-m_H < \delta_H < \lrp{\frac{1-\nu}{1-\nu_H}}\delta_{\max} \quad \forall \; H\in \cM_{\nm}(M).
$$
\label{w.53}\end{corollary}
\begin{remark}
In the case $x^{\delta}=\rho^sw^\tau$, this means that $\delta_{\max}= -\frac{s}{1-\nu}$ and $\delta_H= \tau -\frac{s}{1-\nu_H}$ for $H\in \cM_{\nm}(M)$, so the maps in Corollary~\ref{w.53} will be isomorphisms provided
$$
     2-m<s<0 \quad \mbox{and} \quad 2-m_H< \tau <0, \quad \forall \; H\in \cM_{\nm}(M),
$$
which is consistent with \cite[Theorem~8.4]{DM2014} and \cite[Corollary~5.19]{CR2021}.
\label{w.54}\end{remark}

\section{Weighted blow-ups}\label{wbu.0}

In our construction of examples of Calabi-Yau $\mathfrak{n}$-warped $\QAC$-metrics on smoothing of cones,  a key role will be played by weighted blow-ups.  For this reason, let us take some time to collect basic results.  First, recall that a $n$-tuple of positive integers $w=(w_1,\ldots, w_n)\in \bbN^n$ induces a $\bbC^*$-action on $\bbC^n$, namely the action defined by
\begin{equation}
t\cdot z := (t^{w_1}z_1,\ldots,t^{w_n}z_n)
\label{wbu.1}\end{equation}
for  $t\in\bbC^*$ and $z=(z_1,\ldots, z_n)\in\bbC^n$.  By definition, the \textbf{weighted projective space} $\bbC\bbP^{n-1}_{w}$ associated to $w$ is the quotient of $\bbC^n\setminus \{0\}$ by this action,
\begin{equation}
      \bbC\bbP^{n-1}_w:= (\bbC^n\setminus\{0\})/\bbC^*.
\label{wbu.2}\end{equation}
Unless $w=(1,\ldots,1)$, in which case $\bbC\bbP^{n-1}_w$ is just the usual complex projective space, the weighted projective space $\bbC\bbP^{n-1}_w$ is not smooth, but it is a complex orbifold if the greatest common divisor of $w_1,\ldots, w_n$ is $1$.  As for the usual projective space, one can blow up the origin in $\bbC^n$ by replacing $\{0\}$ by $\bbC\bbP^{n-1}_w$.  More precisely, with respect to the weight $w$, the weighted blow-up of $\bbC^n$ at the origin is the quotient
\begin{equation}
 \operatorname{Bl}_{\{0\}}(\bbC^n,w):=  \lrp{(\bbC\times\bbC^{n})\setminus(\bbC\times \{0\})} /\bbC^*
\label{wbu.3}\end{equation}
with $\bbC^*$ acting on $\bbC^{n+1}$ with weight $(-1,w)$, that is,
$$
       t\cdot(z_0,z)= (t^{-1}z_0,t^{w_1}z_1,\ldots, t^{w_n}z_n)
$$
for $t\in \bbC^*$ and $(z_0,z)\in \bbC\times \bbC^n=\bbC^{n+1}$.  When $w=(1,\ldots,1)$, one can readily check that \eqref{wbu.3} corresponds to the usual blow-up of a point.  If we pick $w_0\in \bbN$ and consider the more general $\bbC^*$-action on $\bbC^{n+1}$ given by 
$$
    t\cdot (z_0,z)= (t^{-w_0}z_0, t^{w_1}z_1,\ldots, t^{w_n}z_n),
$$
notice that the corresponding quotient $\bbC\bbP^{n}_{(-w_0,w)}:= \lrp{(\bbC\times \bbC^{n})\setminus(\bbC\times \{0\})}/ \bbC^*$, called the non-compact weighted projective space in \cite{Apostolov-Rollin}, can also be thought of as the blow-up of the origin in $\bbC^n/\Gamma$, where $\Gamma$ is the cyclic group of $w_0$-roots of unity with generator $e^{\frac{2\pi i}{w_0}}$ acting on $\bbC^n$ by 
$$
   e^{\frac{2\pi i}{w_{0}}}\cdot z:= (e^{\frac{2\pi i w_1}{w_0}}z_1, \ldots, e^{\frac{2\pi i w_n}{w_0}}z_n).
$$
Over the real numbers, the weighted blow-up admits two interesting versions.  The first one, the immediate algebraic analog, is to replace $\bbC$ by $\bbR$ and a weighted action of $\bbC^*$ by a weighted action of $\bbR^*$.  The second one, which is the one on which we will focus, consists of still replacing $\bbC$ by $\bbR$, but with a weighted action of $\bbC^*$ replaced with a weighted action of $\bbR^+=(0,\infty)$.  Only considering the action of $\bbR^+$ adds some flexibility, since we can then consider actions with weight $w\in (\bbR^+)^n$ given by an $n$-tuple of positive real numbers that are not necessarily integers.  Moreover, we can replace $\bbR^n$ by the manifold with corners
\begin{equation}
    \bbR^n_k:= [0,\infty)^k\times \bbR^{n-k}
\label{wbu.4}\end{equation} 
for some $k\in\{0,1,\ldots,n\}$.  On this space, a choice of weight $w\in(\bbR^+)^n$ specifies a $\bbR^+$-action given by
\begin{equation}
     t\cdot x:= (t^{w_1} x_1,\ldots, t^{w_n}x_n) 
\label{wbu.5}\end{equation}
for $t\in\bbR^+$ and $x=(x_1,\ldots,x_n)\in\bbR^n_k$.  Clearly, the corresponding quotient $(\bbR^n_k\setminus \{0\})/\bbR^+$ can be identified with the unit sphere 
\begin{equation}
 \bbS^{n-1}_k:= \{ x=(x_1,\ldots,x_n)\in\bbR^{n}_k \; | \; \sum_{i=1}^{n}x_i^2=1\}.
\label{wbu.6}\end{equation}
\begin{definition}
With respect to a choice of weight $w\in (\bbR^+)^n$, the \textbf{weighted blow-up} of $\{0\}$ in $\bbR^n_k$ is the quotient
$$
     [\bbR^n_k;\{0\}]_w:= (([0,\infty)\times \bbR^n_k)\setminus([0,\infty)\times \{0\}))/\bbR^+
$$ 
with respect to the $\bbR^+$-action given by 
$$
   t\cdot (x_0,x)=(t^{-1}x_0,t^{w_1}x_1,\ldots, t^{w_n}x_n)
$$
for $t\in \bbR^+$ and $(x_0,x)\in[0,\infty)\times \bbR^n_k$.  The corresponding blow-down map $\beta_{\{0\}}: [\bbR^n_k;\{0\}]_w\to \bbR^n_k$ is given by
$$
     \beta_{\{0\}}([x_0:x])= (x_0^{w_1}x_1,\ldots, x_0^{w_n} x_n),
$$
where $[x_0:x]_w$ denotes the class in $[\bbR^n_k;\{0\}]_w$ corresponding to $(x_0,x)\in ([0,\infty)\times \bbR^n_k)\setminus ([0,\infty)\times \{0\})$.
\label{wbu.6a}\end{definition}
Clearly, the weighted blow-up $[\bbR^n_k;\{0\}]_w$ is naturally a manifold with corners diffeomorphic to $\bbS^{n-1}_k\times [0,\infty)$ via the map
$$
     \begin{array}{lccc}
       F: & \bbS^{n-1}_k\times [0,\infty) &\to & [\bbR^n_k;\{0\}]_w \\
           & (\omega,t) & \mapsto & [t:\omega]_w.
     \end{array}
$$
Closely related to the notion of blow-up is the notion of radial compactification.  
\begin{definition}
For $w=(w_1,\ldots,w_n)\in (\bbR^+)^n$, the \textbf{weighted radial compactification} of $\bbR^n_k$ is the quotient
\begin{equation}
    \overline{\bbR^n_{k,w}}:= (([0,\infty)\times \bbR^n_k)\setminus\{0\})/\bbR^+
\label{wbu.6c}\end{equation}
with respect to the $\bbR^+$-action on $[0,\infty)\times \bbR^n_k$ given by 
$$
            t\cdot (x_0,x)= (tx_0,t^{w_1}x_1,\ldots,t^{w_n}x_n) \quad \mbox{for} \quad  t\in\bbR^+, \; (x_0,x)\in [0,\infty)\times \bbR^n_k.
$$
\label{wbu.6b}\end{definition}
Comparing Definitions~\ref{wbu.6a} and \ref{wbu.6b}, we see that the weighted radial compactification can alternatively be obtained by gluing $[\bbR^n_k;\{0\}]$ and $\bbR^n_k$ via the map
$$
  \begin{array}{lccc}
     \varphi: & [\bbR^n_k;\{0\}]\setminus H_{\{0\}} & \to & \bbR^n_k \setminus \{0\} \\
                  & [x_0:x] & \mapsto & (x_0^{-w_1}x_1,\ldots,x_0^{-w_n}x_n),
  \end{array}
$$
where $H_{\{0\}}\cong \bbS^{n-1}_k$ is the boundary hypersurface in $[\bbR^n_k;\{0\}]_w$ created by the weighted blow-up of the origin.  The weighted radial compactification is also naturally diffeomorphic to the new boundary hypersurface created by the weighted blow-up 
$$
    [[0,\infty)\times \bbR^n_k;\{0\}]_{(1,w)}
$$
of $\{0\}$ in $[0,\infty)\times \bbR^n_k$ with respect to the weight $(1,w)$.

More generally, if $M$ is a manifold with corners and $Y\subset M$ is a $p$-submanifold, we can do a weighted blow-up of $M$ along the $p$-submanifold $Y$ provided some extra data is given.  More precisely, suppose that the inner-pointing normal bundle $N^+Y$ of $Y$ admits a decomposition
\begin{equation}
  N^+Y= \bigoplus_{i=1}^{\ell} V_i^+ 
\label{wbu.7}\end{equation}
with $V_i^+$ a subbundle of $N^+Y$ with fiber $\bbR^{n_i}_{k_i}$.  For each $V_i^+$, let $w_i\in\bbR^+$ be a choice of weight.  Let $\cU$ be an open neighborhood of $Y$ inside $N^+Y$ and let 
\begin{equation}
    c: \cU\to \cV\subset M
\label{wbu.8}\end{equation}  
be a collar neighborhood of $Y$ in $M$, that is, $c$ is a diffeomorphism between $\cU$ and an open neighborhood $\cV$ of $Y$ in $M$.  On each fiber of $N^+Y$, we can consider the weighted blow-up of the origin with respect to the decomposition \eqref{wbu.7} and the weight $w=(w_1,\ldots,w_{\ell})\in (\bbR^+)^{\ell}$.  These naturally combine to give the weighted blow-up
$$
       [N^+Y; Y]_w
$$
of the zero section of $N^+Y$ with blow-down map $\beta_{N^+Y}: [N^+Y;Y]\to N^+Y$.  It comes with a natural fiber bundle $\pi: [NY^+;Y]_w\to Y$ with fiber $[\bbR^n_{k};\{0\}]_w$, where
$$
    n=\sum_{i=1}^\ell n_i, \quad k= \sum_{i=1}^{\ell} k_i.
$$  
Using the collar neighborhood \eqref{wbu.8}, one can then define the weighted blow-up of $Y$ in $M$ by 
\begin{equation}
  [M;Y]_{w,c}:= M\setminus Y \cup_c \widetilde{\cU},
\label{wbu.9}\end{equation}
where $\widetilde{\cU}= \beta^{-1}_{N^+Y}(\cU)$, that is, $[M;Y]_{w,c}$ is obtained by replacing $\cU$ by $\widetilde{\cU}$ in $M$ via the gluing map given by \eqref{wbu.8}.   Unfortunately, compared to the usual blow-up, the definition \eqref{wbu.9} seems to be sensitive to the choice of collar neighborhood \eqref{wbu.8} and we will need to be careful about what class of collar neighborhoods we will allow.
\begin{definition}
Two collar neighborhoods $c_i: \cU_i\to \cV_i$ of $Y$ in $M$ for $i\in\{1,2\}$ are \textbf{equivalent} with respect to the decomposition \eqref{wbu.7} provided there are  open neighborhoods $\cU_0$ and $\cV_0$ of $Y$ in $N^+Y$ contained in $\cU_1$ and $\cU_2$ such that 
\begin{equation}
    c_2^{-1}\circ c_1: \cU_0\to \cV_0\subset N^+Y
\label{wbu.10b}\end{equation}
preserves (not necessarily linearly) the fibers of $N^+Y$ and the decomposition \eqref{wbu.7}.  If moreover the map \eqref{wbu.10b} is linear in each fiber, we say that the collar neighborhoods $c_1$ and $c_2$ are \textbf{linearly equivalent}.
\label{wbu.10}\end{definition}

\begin{lemma}
If $c_1$ and $c_2$ are equivalent collar neighborhoods with respect to the decomposition \eqref{wbu.7}, then there is a natural diffeomorphism
$$
      [M;Y]_{w,c_1}\cong [M;Y]_{w,c_2}.
$$
\label{wbu.11}\end{lemma}
\begin{proof}
In this case, one can check in local coordinates that $c_2^{-1}\circ c_1$ extends to a diffeomorphism
$$
       \widetilde{c_2^{-1}\circ c_1}: \widetilde{\cU}_0\to \widetilde{\cV_0}
$$
with $\widetilde{\cU}_0= \beta^{-1}_{N^+Y}(\cU_0)$ and $\widetilde{\cV}_0= \beta^{-1}_{N^+Y}(\cV_0)$ the lifts of $\cU_0$ and $\cV_0$ to $[N^+Y;Y]_w$ respectively.  Using this diffeomorphism, we see that the identity map $\Id: M\setminus Y \to M\setminus Y$ extends uniquely to a diffeomorphism
$$
        [M;Y]_{w,c_1}\cong [M;Y]_{w,c_2}.
$$
\end{proof}

Thus, we will be able to define the weighted blow-up of $Y$ in $M$ provided we can specify in some natural way a class of equivalent collar neighborhoods for $Y$ in $M$.  This will be possible for specific choices of $M$ and $Y$.  In fact, we will even be able to pick a natural linear equivalence class.  Let us consider the simplest possible situation where $M=\bbR^n_k$  with weighted $\bbR^+$-action specified by a weight $w=(w_1,\ldots,w_n)$.  In terms of the canonical coordinates on $\bbR^n_k$, suppose that $Y$ is a \textbf{$p$-linear subspace} of $\bbR^n_k$, that is, a  $p$-submanifold of the form
$$
      Y=\{ (x_1,\ldots,x_n)\in \bbR^n_k \; | \; x_i=0 \quad \mbox{for} \; i\in \cI\}
$$
for some subset $\cI\subset \{1,\ldots,n\}$.  Using the linear structure on $\bbR^n_k$, we have a canonical decomposition 
\begin{equation}
    \bbR^{n}_k = N^+Y=  Y\times \bbR^m_\ell 
\label{wbu.12}\end{equation}
for some $m$ and $\ell$.  Moreover, the weight $w$ induces a weight on $N^+Y$, since the coordinates on the fiber $\bbR^{m}_{\ell}$ of $N^+Y$ are given by $x_i$ for $i\in\cI$ to which we can assign the weight $w_i$.  In this case, \eqref{wbu.12} naturally defines a linear equivalence class of collar neighborhoods with respect to the decomposition $\bbR^m_{\ell}= [0,\infty)^{\ell}\times \bbR^{m-\ell}$.  With respect to this choice of linear equivalence class of collar neighborhoods, we therefore have a well-defined weighted blow-up
\begin{equation}
            [\bbR^n_k; Y]_w= [\bbR^{m}_\ell; \{0\}]_w\times Y \quad \mbox{with blow-down map} \quad \beta_Y: [\bbR^{n}_k;Y]_w\to \bbR^{n}_k.
\label{wbu.13}\end{equation}

More importantly, let $Z\subset \bbR^n_k$ be another $p$-linear subspace of the form 
$$
       Z=\{ (x_1,\ldots,x_n)\in \bbR^{n}_k \; | \; x_i=0 \; \mbox{for} \; i\in \cJ\}
$$
for a subset $\cJ$ strictly contained in $\cI$, so that $Y$ is strictly included in $Z$.  Then the \textbf{lift} of $Z$ in $[\bbR^{n}_k;Y]_w$ is the $p$-submanifold
$$
    \widetilde{Z}= \overline{\beta^{-1}_Y(Z\setminus Y)}.
$$ 
We claim that the inner-pointing normal bundle $N^+\widetilde{Z}$ of $\widetilde{Z}$ has a natural decomposition induced by the $\bbR^+$-action on $\bbR^n_k$ and that there is a corresponding linear equivalence class of collar neighborhoods  for $\widetilde{Z}$ in $[\bbR^n_k;Y]_w$.  To see this, let us first relabel the coordinates $(x_1,\ldots, x_n)\in \bbR^n_k$ so that $\cJ=\{1,\ldots, m_{J}\}$ and $\cI=\{1,\ldots, m\}$ with $m>m_J$.  Then for $i\in \cI\setminus \cJ$, $x_i^{\frac{1}{w_i}}$ (and also $(-x_i)^{\frac{1}{w_i}}$ when $x_i$ can take negative values) defines a local boundary defining function for the boundary hypersurface $H_Y:=\beta_Y^{-1}(Y)$ created by the weighted blow-up of $Y$.  It does so in a region overlapping with $\widetilde{Z}\cap H_Y$, where we can use the coordinates
\begin{equation}
    \zeta_{ij,+}= \left\{ \begin{array}{ll} (x_i)^{\frac{1}{w_i}}, & j=i, \\
                                         x_i^{-\frac{w_j}{w_i}}x_j, & j\in \cI\setminus \{i\}, \\
                                        x_j, & \mbox{otherwise,}
   \end{array} \right.
\label{wbu.14}\end{equation} 
as well as
\begin{equation}
    \zeta_{ij,-}= \left\{ \begin{array}{ll} (-x_i)^{\frac{1}{w_i}}, & j=i, \\
                                         (-x_i)^{-\frac{w_j}{w_i}}x_j, & j\in \cI\setminus \{i\}, \\
                                        x_j, & \mbox{otherwise,}
   \end{array} \right.
\label{wbu.14b}\end{equation} 
when $x_i$ can take negative values.

Under the identification \eqref{wbu.13}, $(\zeta_{i1,\pm},\ldots,\zeta_{in,\pm})$ corresponds to the point 
\begin{equation}
  ([\zeta_{ii,\pm}:\zeta_{i1,\pm}: \ldots : \zeta_{i(i-1),\pm}: \pm 1: \zeta_{i(i+1),\pm}: \ldots: \zeta_{im,\pm}], (\zeta_{i(m+1),\pm},\ldots, \zeta_{in,\pm}))\in [\bbR^m_{\ell};\{0\}]_w\times Y,
\label{wbu.15}\end{equation}
where
$[x_0:\ldots: x_m] \in [\bbR^m_\ell;\{0\}]$ is the class corresponding to $(x_0,x_1,\ldots,x_m)\in [0,\infty)\times \bbR^m_{\ell}$.  In these coordinates, $\widetilde{Z}$ is locally given by 
$$
    \zeta_{i1,\pm}=\cdots= \zeta_{im_J,\pm}=0.
$$
For $i,j\in \cI\setminus \cJ$ with $i\ne j$ and $p,q\in\{+,-\}$, the change for the coordinates $\zeta_{i,p}$ to the coordinates $\zeta_{j,q}$ is given by 
\begin{equation}
   \zeta_{jk,q}= \left\{ \begin{array}{ll} p(q\zeta_{ij,p})^{-\frac{w_i}{w_j}}, & k=i, \\
                                        (q\zeta_{ij,p})^{\frac{1}{w_j}}\zeta_{ii,p}, & k=j, \\
                                        (q\zeta_{ij,p})^{-\frac{w_k}{w_j}} \zeta_{ik,p}, & \mbox{otherwise.}
                                      \end{array} \right.
\label{wbu.16}\end{equation}
\begin{lemma}
The coordinates $(x_1,\ldots,x_n)$ and $(\zeta_{i1,\pm},\ldots, \zeta_{in,\pm})$ for $\cI\setminus \cJ$ induce a decomposition of $N^+\widetilde{Z}$ and a linear equivalence class of collar neighborhoods for $\widetilde{Z}$ in $[\bbR^n_k;Y]_w$.
\label{wbu.17}\end{lemma}
\begin{proof}
In the coordinates $(x_1,\ldots,x_n)$, the inner-pointing normal bundle is trivialized by the coordinates $x_1,\ldots, x_{m_J}$, while in the coordinates $(\zeta_{i1,\pm},\ldots,\zeta_{in,\pm})$ for $i\in \cI\setminus \cJ$, it is trivialized by $\zeta_{i1,\pm},\ldots, \zeta_{im_J,\pm}$.  In both cases, we have a corresponding  local collar neighborhood for $\widetilde{Z}$ in $[\bbR^{n}_k;Y]$ and an induced decomposition.  If we denote them by $c_0$ and $c_{i,\pm}$ respectively, then \eqref{wbu.14}, \eqref{wbu.14b} and \eqref{wbu.16} show that on their overlaps, they are linearly equivalent in the sense of Definition~\ref{wbu.10}.  Since these coordinates cover $\widetilde{Z}$ in $[\bbR^n_k;Y]_w$, this means that they induce a decomposition
$$
         N^+\widetilde{Z}= \bigoplus_{j=1}^{m_J} V_j
$$ 
with $V_j$ spanned locally by $x_j$ and $\zeta_{ij,\pm}$ for $i\in\cI\setminus \cJ$, as well as a corresponding natural linear equivalence class of collar neighborhoods for $\widetilde{Z}$ in $[\bbR^n_k;Y]_w$.  
\end{proof}
Of course, the weight $w$ naturally induces a weight on $V_j$, namely $w_j$.  This means that we can consider the weighted blow-up of $\widetilde{Z}$ in $[\bbR^n;Y]_w$, namely
$$
         [[\bbR^n_k;Y]_w; \widetilde{Z}]_w
$$
with the linear equivalence class of collar neighborhood that of Lemma~\ref{wbu.17}.  As for the usual blow-up, it is convenient to use the notation
$$
  [\bbR^n_k;Y,Z]_w:= [[\bbR^n_k;Y]_w; \widetilde{Z}]_w.
$$
Clearly, if $\cI_p\subsetneq \cdots\subsetneq \cI_1$ is a sequence of strictly embedded subsets of $\{1,\ldots,n\}$ and $Y_1\subsetneq \cdots\subsetneq Y_p$ the corresponding sequence of strictly embedded $p$-linear subspaces, then the above argument can be used to define the sequence of weighted blow-ups of $Y_1,\ldots, Y_p,$
\begin{equation}
  [\bbR^n_k;Y_1,\ldots,Y_p]_w:= [\cdots [ [\bbR^n_k;Y_1]_w;\widetilde{Y}^{(1)}_2]_w\cdots;\widetilde{Y}^{(p-2)}_{p-1}]_w;\widetilde{Y}^{(p-1)}_p]_w
\label{wbu.18}\end{equation}
with $\widetilde{Y}_{i}^{(i-1)}=\overline{\beta_{i-1}^{-1}(Y_i\setminus Y_{i-1})}$ the \textbf{lift} of $Y_i$ to $[\bbR^n_k;Y_1,\ldots, Y_{i-1}]_w$ with blow-down map 
$$
    \beta_{i-1}: [\bbR^n_k;Y_1,\ldots, Y_{i-1}]\to \bbR^n_k.
$$ 
More generally, let $Y_1,\ldots, Y_p$ be a sequence of distinct $p$-linear subspaces of $\bbR^n_k$ with $\cI_1,\ldots, \cI_p$ the corresponding subsets of $\{1,\ldots,n\}$.

\begin{lemma}Suppose that $\{Y_1,\ldots,Y_p\}$ is closed under taking intersections, or equivalently, $\{\cI_1,\ldots,\cI_p\}$ is closed under taking unions.  Then the weighted iterated blow-up 
$$
     [\bbR^n_k;Y_1,\ldots,Y_p]_w
$$
is well-defined provided 
\begin{equation}
     Y_i\subset Y_j\quad \Longrightarrow \quad i\le j.
\label{wbu.19}\end{equation}
Moreover, as long as \eqref{wbu.19} holds, the definition will not depend on the ordering of the $p$-submanifolds $Y_1,\ldots, Y_p$.  
\label{wbu.20}\end{lemma}
\begin{proof}
Let $Y_i$ and $Y_j$ be two of the $p$-submanifolds that are blown up.  If one is included in the other, then the order in which they are blown up is completely determined by \eqref{wbu.19}.  On the other hand, if $Y_i\cap Y_j=Y_q$ for some $q$ with $Y_q\subsetneq Y_i$ and $Y_q\subsetneq Y_j$, then by \eqref{wbu.19}, $q<i$ and $q<j$, so the blow-up of $Y_q$ is performed before those of $Y_i$ and $Y_j$.  But after the blow-up of $Y_q$, the lifts of $Y_i$ and $Y_j$ will be disjoint, so their weighted blow-ups commute, that is, the two orders in which we can blow-up $Y_i$ and $Y_j$ lead to the same space.  
The fact that they are disjoint also indicates more generally that we can essentially  reduce to the case of sequences of strictly included $p$-submanifolds of \eqref{wbu.18}.  

\end{proof}

The manifolds on which we want to perform weighted iterated blow-ups is not quite $\bbR^n_k$, but almost.  We want instead to apply this construction to the weighted radial compactification of $\overline{\bbR^n_{k,w}}$ with respect to some weight $w\in (\bbR^+)^n$.  Let $H_\infty$ be the boundary hypersurface of $\overline{\bbR^n_{k,w}}$ such that $\bbR^n_k= \overline{\bbR^n_{k,w}}\setminus H_{\infty}$.  

\begin{lemma}
Let $\{Y_1,\ldots,Y_p\}$ be a finite set of $p$-linear subspaces of $\bbR^n_k$ and $\cI_1,\ldots,\cI_p$ the corresponding subsets of $\{1,\ldots,n\}$.  Let $\overline{Y}_i$ be the closure of $Y_i$ in $\overline{\bbR^n_{k,w}}$ and consider the $p$-submanifolds $\overline{Y}_1\cap H_{\infty},\ldots, \overline{Y}_p\cap H_{\infty}$.  
Suppose that $\{\overline{Y}_1\cap H_{\infty},\ldots,\overline{Y}_p\cap H_{\infty}\}$ is closed under taking non-empty intersections and that 
$$
     \overline{Y}_i\cap H_{\infty}\subset \overline{Y}_j\cap H_{\infty}\; \Longrightarrow \;  i\le j.
$$ 
Then the weighted iterated blow-up 
\begin{equation}
 [\overline{\bbR^n_{k,w}}; \overline{Y}_1\cap H_{\infty},\ldots,\overline{Y}_p\cap H_{\infty}]_w
\label{wbu.22}\end{equation} 
of $\overline{Y}_1\cap H_{\infty},\ldots,\overline{Y}_p\cap H_{\infty}$ in $\overline{\bbR^n_{k,w}}$ is well-defined, namely, at each step, the lifts of the $p$-submanifolds that must be blown-up come with a natural decomposition of their inner-pointing normal bundle and a natural linear equivalence class of collar neighborhood in the sense of Definition~\ref{wbu.10}.
\label{wbu.21}\end{lemma}
\begin{proof}
Regarding $\overline{\bbR^n_{k,w}}$ as the new boundary hypersurface in $[[0,\infty)\times \bbR^n_k;\{0\}]_{(1,w)}$ created by the blow-up of the origin, we see that this iterated blow-up essentially corresponds to the iterated blow-ups of the lifts of $\{0\}\times Y_1,\ldots, \{0\}\times Y_p$ in $[[0,\infty)\times \bbR^n_k;\{0\}]_{(1,w)}$.  More precisely, if we consider
\begin{equation}
 [ [0,\infty)\times \bbR^n_k;\{0\}, \{0\}\times Y_1,\ldots, \{0\}\times Y_p ]_{(1,w)},
\label{wbu.22}\end{equation}
which is well-defined by Lemma~\ref{wbu.20}, then the boundary hypersurface created by the first weighted blow-up corresponds to the weighted iterated blow-up \eqref{wbu.22} that we want to define.
\end{proof}

\section{Compactification of smoothings of Calabi-Yau cones}\label{sCY.0}

In this section, we will construct a suitable compactification of the smoothing \eqref{cyc.6} by a manifold with fibered corners.  To do this, we will suppose that Assumptions~\ref{smooth.1} and \ref{cyc.6c} hold.  When $\ell>0$ and $N>1$, the zero locus $V_{\{1,\ldots,N\}}$ of $[Q_{q,i}]_{\ell}(z_0)$ described in Assumption~\ref{smooth.1} will play an important role.  On the other hand,  if $\ell= 0$, there is no zero locus, while if $N=1$, there is no need to consider this zero locus, so let us set
\begin{equation}
   V_{\ell,\zero}= \left\{ \begin{array}{ll} \emptyset, & \ell= 0 \; \mbox{or} \; N=1, \\    V_{\{1,\ldots,N\}}, & \ell\ge1 \;\mbox{and} \; N>1.
     \end{array} \right.
\label{pr.10a}\end{equation}

 Let $\overline{\bbC^{m+n}_w}$ be the radial compactification of $\bbC^{m+n}$ with respect to the $\bbR^+$-action specified by the weight $w=(w_{0,1},\ldots,w_{0,m_0+n_0}, w_{1,1},\ldots, w_{1,m_1+n_1},\ldots, w_{N,1}, w_{N,m_N+n_N})$,
$$
         t\cdot (z_0,z_1,\ldots, z_{N})= (tz_0, t^{w_{1,1}}z_{1,1},\ldots, t^{w_{1,m_1+n_1}}z_{1,m_1+n_1},\ldots, t^{w_{N,1}}z_{N,1}, t^{w_{N,m_N+n_N}}z_{N,m_N+n_N}) \quad \mbox{for} \; t>0.
$$
The boundary of $\overline{\bbC^{m+n}_w}$ is the sphere
$$
       \bbS^{2m+2n-1}:= \left\{ z\in \bbC^{m+n} \; | \; \sum_{q=0}^{N}\sum_{j=1}^{m_q+n_q} |z_{q,j}|^2=1 \right\}.
$$
In terms of this sphere, $\overline{\bbC^{m+n}_w}$ is obtained by gluing $\bbS^{2m+2n-1}\times [0,\infty)$ and $\bbC^{m+n}$ via the map
\begin{equation}
\begin{array}{lccl}
\nu: & \bbS^{2m+2n-1} \times (0,\infty) & \to & \bbC^{m+n} \\
               &(\omega,\xi) & \mapsto & (\frac{\omega_{0}}{\xi^{w_{0}}},\frac{\omega_{1,1}}{\xi^{w_{1,1}}},\ldots, \frac{\omega_{N,m_N+n_N}}{\xi^{w_{N,m_N+n_N}}}).
\end{array}
\label{cyc.7}\end{equation} 
In fact, this map extends to give a tubular neighborhood
$$
     \nu:  \bbS^{2m+2n-1}\times [0,\infty)\to \overline{\bbC^{m+n}_w}
$$
of $\bbS^{2m+2n-1}$ in $\bbC^{m+n}$.  In terms of the coordinates induced by this tubular neighborhood, we see that the defining equations of $C_{\epsilon}$ are
\begin{equation}
    \xi^{-d}P_{q,i}(\omega_q)= \epsilon Q_{q,i}( \xi^{-1}\cdot \omega_0 ) \quad \mbox{for} \;q\in\{1,\ldots,N\}, \; i\in \{1,\ldots,n_q\},
\label{cyc.8}\end{equation}
that is, as $\xi\searrow 0$,
\begin{equation}
      P_{q,i}(\omega_q)= \xi^{d-\ell}[Q_{q,i}]_{\ell}(\omega_0)+ o(\xi^{d-\ell}) \quad \mbox{for} \; i\in \{1,\ldots, n_q\}.
\label{cyc.9}\end{equation}
The closure $\overline{C}_{\epsilon}$ of $C_{\epsilon}$ in $\overline{\bbC^{m+n}_w}$ is obtained by  taking also $\xi=0$ in \eqref{cyc.8}.  In particular, in light of \eqref{cyc.9}, on $\pa\overline{\bbC^{m+n}_w}$, we see that
$\pa\overline{C}_{\epsilon}= \overline{C}_{\epsilon}\cap \pa \overline{\bbC^{m+n}_w}$ is given by the equations
$$
         P_{q,i}(\omega_q)=0,  \quad \omega\in \bbS^{2m+2n-1}=\pa \overline{\bbC^{m+n}_w},  \quad q\in\{1,\ldots,N\}, \; i\in \{1,\ldots,n\}.
$$
This does not depend on $\epsilon$ and coincides with $\pa \overline{C}_0= \overline{C}_0\cap \pa \overline{\bbC^{m+n}_w}$, which corresponds to the fact that the polynomial $Q_{q,i}$ has degree strictly smaller than that of $P_{q,i}$ for each $i$.  In particular, if $C_0$ has a singular cross-section, that is, if $\pa \overline{C}_0$ is singular, then $\pa \overline{C}_{\epsilon}$ is also singular, even if $C_{\epsilon}$ itself is smooth.  To resolve these singularities, consider in $\bbC^{m+n}$ the subsets
\begin{equation}
       \bbC^{m+n}_{w,\sing}:= \bigcup_{\{0\}\subset \mathfrak{q}\subsetneq \{1,\ldots,N\}} V_{\mathfrak{q}} 
\label{singset.1}\end{equation}
and $V_{\ell,\zero}$ defined in \eqref{pr.10a}.
 Let $\overline{\bbC^{m+n}_{w,\sing}}$ and $\overline{V}_{\ell,\zero}$ be their closure in $\overline{\bbC^{m+n}_w}$.  If we set $\pa \overline{\bbC^{m+n}_{w,\sing}}= \overline{\bbC^{m+n}_{w,\sing}}\cap \pa \overline{\bbC^{m+n}_w}$ and  $\pa \overline{V}_{\ell,\zero}= \overline{V}_{\ell,\zero}\cap \pa \overline{\bbC^{m+n}_w}$, then $\pa \overline{\bbC^{m+n}_{w,\sing}}$ is naturally a stratified space in $\pa\overline{\bbC^{m+n}_w}$.  Our strategy to resolve the singularities of $\overline{C}_{\epsilon}$ will essentially be to blow up the strata of $\pa \overline{\bbC^{m+n}_{w,\sing}}$ in an order compatible with the partial ordering of the strata.  However, when $\pa \overline{V}_{\ell,\zero}$ is not empty, the part of $\pa \overline{V}_{\ell,\zero}$ inside $\pa \overline{\bbC^{m+n}_{w,\sing}}$ needs to be blown up differently.  

For this reason, the stratification we will consider on $\pa \bbC^{m+n}_{w,\sing}$ is as follows.  If $\pa \overline{V}_{\ell,\zero}=\emptyset$, the closed strata will be given by $\pa \overline{V}_{\mathfrak{q}}:= \bV_{\mathfrak{q}}\cap \pa\overline{\bbC^{m+n}_w}$ for $\{0\}\subset \mathfrak{q}\subsetneq\{0,\ldots,N\}$, where $\bV_{\mathfrak{q}}$ is the closure of $V_{\mathfrak{q}}$ in $\overline{\bbC^{m+n}_w}$.  If instead $\pa\overline{V}_{\ell,\zero}\ne \emptyset$, we will also consider for $\{0\}\subsetneq \mathfrak{q}\subsetneq \{0,\ldots,N\}$ the closed strata 
$$
      \pa\overline{V}_{\mathfrak{q}}\cap \pa \overline{V}_{\ell,\zero}= \pa \overline{V}_{\mathfrak{q}\setminus \{0\}}.
$$
In other words, when $\pa\overline{V}_{\ell,\zero}\ne\emptyset$, the closed strata of $\pa\overline{\bbC^{m+n}_{w,\sing}}$ are given by 
$$
      \pa \overline{V}_{\mathfrak{q}}\quad \mbox{for non-empty subsets } \quad \mathfrak{q}\subsetneq \{0,\ldots,N\}.
$$ 
In both cases, the partial order on strata is given by inclusion,
\begin{equation}
  \pa\overline{V}_{\mathfrak{q}} \le \pa\overline{V}_{\mathfrak{p}} \quad \Longleftrightarrow \quad \pa\overline{V}_{\mathfrak{q}} \subseteq \pa\overline{V}_{\mathfrak{p}}  \quad \Longleftrightarrow \quad \mathfrak{q} \subseteq \mathfrak{p}.
\label{pr.12}\end{equation}
If $\pa \overline{V}_{\ell,\zero}\ne \emptyset$, that is, if $\ell>0$ and $N>1$, then by Lemma~\ref{wbu.21}, we can unambiguously consider the weighted iterated blow-up 
\begin{equation}
      X:= [\overline{\bbC^{m+n}_w}; \{\pa \overline{V}_{\mathfrak{q}} \; | \; \mathfrak{q}\subsetneq \{1,\ldots,N\}, \; \mathfrak{q}\ne \emptyset\}]_w
\label{pr.13}\end{equation}
provided we blow up the subspaces $\pa \overline{V}_{\mathfrak{q}}$ in an order compatible with the partial order \eqref{pr.12}.  If instead $\pa V_{\ell,\zero}=\emptyset$, that is, if instead $\ell=0$, we simply set
\begin{equation}
  X:= \overline{\bbC^{m+n}_w}.
\label{pr.14}\end{equation}

  In either case, let $H_{\max}$ be the boundary hypersurface of $X$ corresponding to (the lift of) $\pa \overline{\bbC^{m+n}_w}$.  By Lemmas~\ref{wbu.17} and \ref{wbu.21}, $H_{\max}$ comes with a natural linear equivalence class of collar neighborhoods in the sense of Definition~\ref{wbu.10}.   Let $x_{\max}$ be a choice of boundary defining function for $H_{\max}$ compatible with this linear equivalence class of collar neighborhoods.  Let $\wX$ be the manifold with corners, which, as a topological space, is identified with $X$, but with smooth functions on $\wX$ corresponding to smooth functions on $X\setminus H_{\max}$ having a smooth expansion at $H_{\max}$ in integer powers of $x_{\max}^{\frac{d-\ell}d}$ (instead of integer powers of $x_{\max}$).  Since we require $x_{\max}$ to be compatible with the natural linear equivalence class of collar neighborhoods of $H_{\max}$, notice that $\wX$ is well-defined in that it does not depend on the choice of $x_{\max}$.  Let us denote by $\widetilde{H}_{\max}$ the boundary hypersurface $H_{\max}$ seen as a boundary hypersurface of $\tX$.  Let $\widetilde{x}_{\max}:=x_{\max}^{\frac{d-\ell}d}$ be the corresponding boundary defining function.  Let $\widetilde{\bbC^{m+n}_{w,\sing}}$ be the closure of $\bbC^{m+n}_{w,\sing}$ in $\tX$.  For $\{0\}\subsetneq \mathfrak{q}\subset \{0,\ldots,N\}$, let $\tV_{\mathfrak{q}}$ be the closure of $V_{\mathfrak{q}}$ in $\tX$.  Then the lift of the strata $\pa\overline{V}_{\mathfrak{q}}$ of $\pa\overline{\bbC^{m+n}_{w,\sing}}$ to $\tH_{\max}$ for $\{0\}\subsetneq\mathfrak{q}\subset \{0,\ldots,N\}$ is given by 
  $$
     \tV_{\mathfrak{q}}\cap \tH_{\max}.
  $$
Relying again on Lemma~\ref{wbu.21}, we can then consider the weighted iterated blow-up
$$
     \hX:= [\tX; \{\tV_{\mathfrak{q}}\cap\tH_{\max}\; | \; \{0\}\subset \mathfrak{q}\subsetneq \{0,\ldots,N\}\} ]_{\widetilde{w}}
$$
provided the blow-ups are performed in an order compatible with the partial ordering of the strata, 
where the lift of the coordinate $\omega_{q,j}$ has weight $\widetilde{w}_{q,j}=w_{q,j}$, while (the lift of) the boundary defining function $\widetilde{x}_{\max}$ has weight $1$.  When $\nu:= \frac{\ell}{d}$ is positive, the set $\cM_1(\hX)$ of boundary hypersurfaces of $\hX$ decomposes as 
\begin{equation}
\cM_1(\hX)= M_{1,0}(\hX)\cup M_{1,\nu}(\hX),
\label{cyc.10}\end{equation}
where $\cM_{1,0}(\hX)$ is the set of boundary hypersurfaces corresponding to the lift of boundary hypersurfaces of $\tX$ coming from blow-ups of the strata  $\pa\overline{V}_{\mathfrak{q}}\cap \pa \overline{V}_{\ell,\zero}$ for subsets $\{0\}\subsetneq \mathfrak{q}\subsetneq \{0,\ldots,N\}$, while $\cM_{1,\nu}(\hX)$ corresponds to the remaining boundary hypersurfaces.  Notice that $\cM_{1,0}(\hX)$ will be empty when $\pa \overline{V}_{\ell,\zero}$ is.  When $\nu=\frac{\ell}{d}=0$, $\pa\overline{V}_{\ell,\zero}$ is empty, but it will be convenient in this case to use the notation 
$$ 
              \cM_{1,\nu}(\hX)=\cM_{1,0}(\hX):= \cM_1(\hX) \quad \mbox{if} \; \nu=\frac{\ell}{d}=0.
$$

For each blow-up performed to construct $\hX$, notice that the corresponding blow-down map induces a fiber bundle on the corresponding boundary hypersurface.  On the other hand, on the lift $\hH_{\max}$ of the boundary hypersurface $\pa\overline{\bbC^{m+n}_w}$ to $\hX$, the natural fiber bundle we consider is that given by the identity map.  All these fiber bundles combine to confer $\hX$ with an iterated fibration structure.   In other words, $\hX$ is naturally a $\QAC$-manifold with fibered corners.  Given the order in which we blew up strata, one important feature of the induced partial order on $\cM_1(\hX)$ when $\nu>0$ is that if $H\in\cM_{1,0}(\hX)$ and $G\in \cM_{1,\nu}(\hX)$ are such that $H\cap G\ne \emptyset$, then $H<G$.  

Let $\hC_{\epsilon}$ be the closure of $C_{\epsilon}\subset \bbC^{m+n}_{w}$ in $\hX$.  For each $\hH\in \cM_1(\hX)$, there is a corresponding boundary hypersurface $\hH_{\epsilon}:= \hH\cap \hC_{\epsilon}$ of $\hC_{\epsilon}$.  For $\epsilon=0$, $C_0$ is a singular affine variety and $\hH_0$ will be singular as well.  Since the base $S_{\hH}$ of the fiber bundle $\phi_{\hH}: \hH\to S_{\hH}$ corresponds to a resolution of a singular stratum of $\pa \overline{C}_0$, notice that $\phi_H$ restricts to $\hC_0$ and to $\hC_{\epsilon}$ for $\epsilon\ne 0$ to induce a fiber bundle 
\begin{equation}
        \phi_{\hH_{\epsilon}}: \hH_{\epsilon}\to S_{\hH}.
\label{cyc.11}\end{equation}
We need to distinguish two situations.  First, if $\hH\in \cM_{1,0}(\hX)$ corresponds to the subspace $V_{\mathfrak{q}}$ for $\mathfrak{q}\subsetneq \{1,\ldots,N\}$ with $\mathfrak{q}\ne \emptyset$, then on the interior of $\hH$, the interior of the fibers of $\phi_{\hH}: \hH\to S_{\hH}$ are naturally identified with the complementary subspace
$$
  V_{\mathfrak{q}}^{\perp}= V_{\mathfrak{q}^c}
$$
 with natural coordinates given by $z_q$ for $q\in \mathfrak{q}^c$.  Intuitively, this corresponds to the fact that the weighted blow-up of the stratum $\mathfrak{s}_{\hH}$ corresponding to $\hH$ `undoes' the radial compactification along $V_{\mathfrak{q}}$ in the directions transverse to $V_{\mathfrak{q}}$.   For $\omega_{\mathfrak{q}}\in \pa \overline{V}_{\mathfrak{q}}$ corresponding to an interior point of $S_{\hH}$, the interior of the fiber $\phi_{\hH}(\omega_{\mathfrak{q}})$ corresponds to $V_{\mathfrak{q}^c}$, with coordinate $z_{\mathfrak{q}^c}$.  Such a fiber will have a non-empty intersection with $\hC_{\epsilon}$ provided
 \begin{equation}
          P_{q,i}(\omega_{q})=\xi^d \epsilon Q_{q,i}(z_0) \quad \mbox{for} \quad q\in\mathfrak{q}
\label{pr.15}\end{equation}
at $\xi=0$, that is, provided
\begin{equation}
          P_{q,i}(\omega_{q})=0 \quad \mbox{for} \quad q\in\mathfrak{q},
\label{pr.16}\end{equation}
where $\omega_q$ is the component of $\omega_{\mathfrak{q}}$ in $V_q\subset V_{\mathfrak{q}}$.  The interior of the fiber of $\phi^{-1}_{\hH_{\epsilon}}(\omega_{\mathfrak{q}})$ is then the affine variety $W_{\mathfrak{q}^c,\epsilon}\subset V_{\mathfrak{q}^c}$ of \eqref{pr.10} with  $V_{\mathfrak{q}^c}$ seen as the interior of $\phi^{-1}_{\hH}(\omega_{\mathfrak{q}})$.  In particular, by Assumption~\ref{cyc.6c}, the interior of $\phi^{-1}_{\hH_{\epsilon}}(\omega_{\mathfrak{q}})$ is smooth.

If instead $\hH\in \cM_{1,\nu}(\hX)$ (with $\nu=\frac{\ell}d >0$) corresponds to the subspace $V_{\mathfrak{q}}$ with subset $\{0\}\subsetneq \mathfrak{q}\subset \{0,\ldots,N\}$,  then, in the interior of $\hH$, the interiors of the fibers of $\phi_{\hH}:\hH\to S_{\hH}$ are again identified with $V_{\mathfrak{q}^c}$, this time however using the rescaled coordinates 
\begin{equation}
  \zeta_{q,j}:= \frac{\omega_{q,j}}{\xi^{\frac{d-\ell}d w_{q,j}}}=\frac{\xi^{w_{q,j}}z_{q,j}}{\xi^{\frac{d-\ell}d w_{q,j}}}= \xi^{\frac{\ell}dw_{q,j}}z_{q,j} \quad \mbox{for} \; q\in\mathfrak{q}^c, \; j\in \{1,\ldots,m_q+n_q\}.
\label{cyc.14}\end{equation} 
Again, such a fiber $\phi^{-1}_{\hH}(\omega_{\mathfrak{q}})$ will have a non-empty intersection with $\hC_{\epsilon}$ provided 
\begin{equation}
  P_{q,i}(\omega_{q})=0 \quad \mbox{for} \quad q\in\mathfrak{q}\setminus \{0\}.
\label{pr.17}\end{equation} 
Correspondingly, in terms of the coordinates \eqref{cyc.14}, the interior of $\phi_{\hH_{\epsilon}}^{-1}(\omega_{\mathfrak{q}})$ is the affine variety  
\begin{equation}
W^{\perp}_{\mathfrak{q},\omega_{\mathfrak{q}},\epsilon}= W_{\mathfrak{q}^c,\omega_{\mathfrak{q}},\epsilon}= \{ \zeta_{\mathfrak{q}^c}\in V_{\mathfrak{q}^c}\; | \; \forall q\in \mathfrak{q}^c, P_{q,i}(\zeta_q)=\epsilon[Q_{q,i}](\omega_{0}), \; i\le k_q, \quad P_{q,i}(\zeta_q)=0,  \; i>k_q\}.
\label{pr.18}\end{equation} 
corresponding to \eqref{pr.11}.
 Again, by Assumption~\ref{cyc.6c}, the interior of $\phi^{-1}_{\hH_{\epsilon}}(\omega_{\mathfrak{q}})$ will be smooth.  In fact, in both cases, the interior of $\phi^{-1}_{\hH_{\epsilon}}(\omega_{\mathfrak{q}})$ is a smoothing of the cone $W_{\mathfrak{q}^c}$ identified with the interior of $\phi^{-1}_{\hH_0}(\omega_{\mathfrak{q}})$ .

Now, in both cases, $\phi^{-1}_{\hH}(\omega_{\mathfrak{q}})$ provides a compactification for $V_{\mathfrak{q}}$ in the same way that $\hX$ provides a compactification for $\bbC^{m+n}_w$.  In the case where $\hH\in \cM_{1,\nu}(\hX)$, this corresponds to such a compactification for $\bbC^{m+n}_w$ when  $\nu=0$.  Moreover, the closure of the interior of $\phi^{-1}_{\hH_{\epsilon}}(\omega_{\mathfrak{q}})$ in $\phi_{\hH}^{-1}(\omega_{\mathfrak{q}})$ is precisely $\phi^{-1}_{\hH_{\epsilon}}(\omega_{\mathfrak{q}})$.   For $\epsilon\in\bbC\setminus \{0\}$ as in Assumption~\ref{cyc.6c}, this discussion can be extended to show that $\hC_{\epsilon}$ is a manifold with corners inheriting from $\hX$ a natural iterated fibration structure.  However, as a manifold with corners, it will typically not be of class $\CI$, but of class $\cC^{\lfloor d\rfloor}$ or lower.  For instance, if $d$ is not an integer and the polynomials $Q_{q,i}$ are all homogeneous, then it will be of class $\cC^{\lfloor d\rfloor}$.  Nevertheless, by restriction from $\hX$, there is on $C_{\epsilon}= \hC_{\epsilon}\setminus \pa\hC_{\epsilon}$ a natural ring of `smooth' functions $\CI(\hC_{\epsilon})$, a ring $\cA_{\phg}(\hC_{\epsilon})$ of bounded polyhomogeneous functions and a ring of $\nQb$-smooth functions $\CI_{\nQb}(\hC_{\epsilon})$.

Now, the set of boundary hypersurfaces of $\hC_{\epsilon}$ is in bijection with that  of $\hX$.  In particular, there is a decomposition
\begin{equation}
  \cM_1(\hC_{\epsilon})= \cM_{1,0}(\hC_{\epsilon})\cup \cM_{1,\nu}(\hC_{\epsilon}).
\label{cyc.16}\end{equation}
Again, when $\nu= \frac{\ell}d>0$, this is a partition, while when $\nu=0$, $\cM_{1,0}(\hC_{\epsilon})=\cM_{1,\nu}(\hC_{\epsilon})=\cM_{1}(\hC_{\epsilon})$.  This suggests to consider the function
\begin{equation}
\begin{array}{llcl} \mathfrak{n}:& \cM_1(\hC_{\epsilon}) & \to & \{0,\nu\} \\
     & \hH_{\epsilon} & \mapsto & \nu_{\hH_{\epsilon}}
\end{array}     
\label{cyc.17}\end{equation}
given by
$$
       \nu_{\hH_{\epsilon}}:= \left\{ \begin{array}{ll}  \nu, & \hH_{\epsilon}\in \cM_{1,\nu}(\hC_{\epsilon}), \\
           0, & \mbox{otherwise.} \end{array}  \right.
$$
By definition of \eqref{cyc.16},
$$
     \hH_{\epsilon}< \hG_{\epsilon} \quad \Longrightarrow \quad \nu_{\hH_{\epsilon}}\le \nu_{\hG_{\epsilon}},
$$
so $\mathfrak{n}$ is a weight function in the sense of Definition~\ref{wqb.1}.  This induces a corresponding weight function $\mathfrak{n}: \cM_1(\hX)\to \{0,\nu\}$.  Let $\beta: \hX\to \overline{\bbC^{m+n}_w}$ denote the blow-down map and let $u\in\CI(\overline{\bbC^{m+n}_w})$ be a choice of boundary defining function compatible with the $\bbR^+$-action near the boundary in the sense that $u(t\cdot z)=\frac{u(z)}{t}$ for $t\in\bbR^+$ and $z\in \bbC^{m+n}$ sufficiently large.  In other words, choose $u$ to be compatible with the natural linear equivalence class of collar neighborhood of $\overline{\bbC^{m+n}_w}$.   Then 
\begin{equation}
\rho^{-1}:=\beta^*u
\label{kw.17b}\end{equation}
 is an $\mathfrak{n}$-weighted total boundary defining function.  Using Lemma~\ref{wqb.5}, we readily see that the $\nQAC$-equivalence class of the $\mathfrak{n}$-weighted total boundary defining function does not depend on the choice of $u$, so  there is a well-defined notion of $\mathfrak{n}$-warped $\QAC$-metrics on $\overline{\bbC^{m+n}_w}$.  By restriction, this induces a class of metrics on $\hC_{\epsilon}$.  We will declare those to be \textbf{the class of $\mathfrak{n}$-warped $\QAC$-metrics on $\hC_{\epsilon}$}.  It will be smooth, polyhomogeneous or $\nQb$-smooth if the associated metric on $\hX$ is.  Proceeding in this way, we avoid having to deal with the fact that $\hC_{\epsilon}$ is possibly not of class $\CI$.  Still, since the differential of the polynomial $P_{q,i}-Q_{q,i}$, seen as a $b$-differential on $\hX$, is always at least of class $\cC^0$, we can also define ${}^{\mathfrak{n}}T\hC_{\epsilon}$ (respectively ${}^{w}T\hC_{\epsilon}$ and $^{\nQb}T\hC_{\epsilon}$) by considering the elements of ${}^{\mathfrak{n}}T\hX|_{\hC_{\epsilon}}$ (respectively ${}^{w}T\hX|_{\hC_{\epsilon}}$ and $^{\nQb}T\hX|_{\hC_{\epsilon}}$) that are in the kernels of the differentials of the polynomials $(P_{q,i}-Q_{q,i})$ for all $q$ and $i$.  
We say that $\hX$ and $\hC_{\epsilon}$ are \textbf{$\nQAC$-compactifications} for respectively $\bbC^{m+n}$ and $C_{\epsilon}$.  Similarly,  we have natural $\nQAC$-compactifications for $V_{q}$, $W_{\mathfrak{q},\epsilon}$ and $W^{\perp}_{\mathfrak{q}, \omega_{\mathfrak{q}},\epsilon}= W_{\mathfrak{q}^c,\omega_{\mathfrak{q}},\epsilon}$ in \eqref{pr.10} and \eqref{pr.11} that we denote by $\hV_{q}$,   $\hW_{\mathfrak{q},\epsilon}$ and $\hW^{\perp}_{\mathfrak{q}, \omega_{\mathfrak{q}},\epsilon}= \hW_{\mathfrak{q}^c,\omega_{\mathfrak{q}},\epsilon}$ respectively.

We are interested in examples of warped $\QAC$-metrics that are K\"ahler.  To study those and see in particular that they exist, it will be helpful to have complex coordinates adapted to the geometry.  Near $\hH\in\cM_{1,0}(\hX)$ corresponding to $V_{\mathfrak{q}}$, we see from the discussion above that we can take the coordinates 
$$
      z_q, \quad q\in \mathfrak{q}^c,
$$
in the fibers of $\phi_{\hH}: \hH\to S_{\hH}$, while on the base, instead of using the real coordinates $\omega_{q}$, we  can use the holomorphic coordinates
\begin{equation}
 \varpi_{q,j}= \xi^{w_{q,j}}z_{q,j}, \quad q\in \mathfrak{q}, \quad j\in \{1,\ldots,m_q+n_q\},
\label{kw.1}\end{equation}
with $(\xi, \varpi_{\mathfrak{q}})$ corresponding to $ z_{\mathfrak{q}}$ with $z_{q,j}= \xi^{-w_{q,j}}\varpi_{q,j}$,
where $\xi$ is now taking values in a sector of the complex plane, so that its logarithm and its complex powers can be well-defined.  To obtain a holomorphic coordinate chart near $\hH$, we can then take $z_q$ for $q\in\mathfrak{q}^c$, the coordinate $\xi$ and all the coordinates $\varpi_{q,j}$ for $q\in\mathfrak{q}$ and $j\in \{1,\ldots, m_q+n_q\}$ (except one such $\varpi_{q,j}$ that we declare to be equal to one).  In this case, away from the boundary of $\hH$, the $1$-forms 
\begin{equation}
\frac{d\xi}{\xi^2}, \frac{d\varpi_{q,i}}{\xi}, dz_{p,j},  \quad q\in \mathfrak{q}, \;i\in\{1,\ldots,m_q+n_q\}, \quad p\in\mathfrak{q}^c, \;j\in\{1,\ldots,m_p+n_p\},
\label{kw.2}\end{equation}    
except one of the $d\varpi_{q,i}$ omitted, form, together with their complex conjugates, a local basis of sections of ${}^{\mathfrak{n}}T^*\hX\otimes_{\bbR} \bbC$, giving a complex analog of \eqref{wqb.7} with $k=1$.  Since $\mathfrak{n}(\hH)=0$, we are also in the setting of \eqref{w.5} with $\nu_H=0$, so it is also a local basis of sections of ${}^{w}T^*\hX\otimes \bbC$.  

If instead $\hH\in \cM_{1,\nu}(\hX)$ corresponds to $V_{\mathfrak{q}}$ with $\{0\}\subsetneq \mathfrak{q}\subset \{0,1,\ldots,N\}$ and $\nu>0$, then again regarding $\xi$ as taking values in a sector of the complex plane where its logarithm can be well-defined, we can regard the coordinate \eqref{cyc.14} as holomorphic.  Using again the coordinates \eqref{kw.1} on the base with one $\varpi_{0,j}$ omitted, as well as \eqref{cyc.14}, we get holomorphic coordinates near $\hH$, and in this case the $1$-forms 
\begin{equation}
\xi^{\nu}\frac{d\xi}{\xi^2}, \xi^{\nu}\frac{d\varpi_{q,i}}{\xi}, d\zeta_{p,j}, \quad q\in \mathfrak{q}, \;i\in\{1,\ldots,m_q+n_q\}, \quad p\in\mathfrak{q}^c, \;j\in\{1,\ldots,m_p+n_p\},
\label{kw.3}\end{equation}
with one $\xi^{\nu}\frac{d\varpi_{0,i}}{\xi}$ omitted, combined with their complex conjugates, form a local basis of sections of ${}^{\mathfrak{n}}T^*\hX\otimes_{\bbR} \bbC$ and are a complex analog of \eqref{wqb.7} with $k=1$.  Correspondingly, 
\begin{equation}
\frac{d\xi}{\xi^2}, \frac{d\varpi_{q,i}}{\xi}, \xi^{-\nu}d\zeta_{p,j}, \quad q\in \mathfrak{q}, \;i\in\{1,\ldots,m_q+n_q\}, \quad p\in\mathfrak{q}^c, \;j\in\{1,\ldots,m_p+n_p\},
\label{kw.4}\end{equation}
with one $\frac{d\varpi_{0,i}}{\xi}$ omitted, combined with their complex conjugates, form a local basis of sections of ${}^{w}T^*\hX\otimes_{\bbR} \bbC$.

More generally, we can introduce holomorphic coordinates near an intersection of boundary hypersurfaces as follows.  Let $\mathfrak{q}_1\subset \cdots\subset \mathfrak{q}_k\subset \{0,\ldots,N\}$ be a sequence of embedded subsets and let $\hH_i\in \cM_1(\hX)$ be the boundary hypersurface corresponding to $V_{\mathfrak{q}_i}$ so that 
$$
      \hH_i < \hH_j \quad \Longrightarrow \quad i<j.
$$
In particular, the intersection $\cap_{i=1}^k \hH_i$ is not empty.  Pick as before $\xi_1=\xi$ to be the coordinate of the $\bbR^+$-action of $V_{\mathfrak{q}_1}$ seen as  taking values in some sector of the complex plane so that its logarithm can be well-defined.  For $q\in \mathfrak{q}_1$, consider again the homogeneous coordinates
\begin{equation}
  \varpi_{q,j}= \xi^{w_{q,j}}z_{q,j}, \quad q\in\mathfrak{q}_1, \; j\in\{1,\ldots,m_q+n_q\}.  
\label{kw.5}\end{equation}
In the fibers of $\hH_1$, we can initially consider the coordinates $z_{\mathfrak{q}^c_1}$ if $0\notin \mathfrak{q}_1$ or else the rescaled coordinates $\zeta_{\mathfrak{q}^c_1}$ of \eqref{cyc.14} if $0\in \mathfrak{q}_1$.  In terms of those, we can take $\xi_2$ to be the coordinates of the $\bbR^+$-action on $V_{\mathfrak{q}_2\setminus \mathfrak{q}_1}$ (seen as taking values in some sector of the complex plane) and define corresponding homogenous coordinates $\varpi_{q,j}$.  Iterating this construction, we obtain holomorphic coordinates
\begin{equation}
  \xi_1, \varpi_{\mathfrak{q}_1}, \ldots, \xi_k, \varpi_{\mathfrak{q}_k}, \zeta_k
\label{kw.6}\end{equation}
with $\zeta_k$ denoting possibly rescaled coordinates in the fibers of $\hH_k$, where for each $\varpi_{\mathfrak{q}_i}$, one $\varpi_{q,j}$ with $q\in \mathfrak{q}_i\setminus \mathfrak{q}_{i-1}$ is omitted, where we use the convention that $\mathfrak{q}_0=\emptyset$.  More precisely, if $0\in\mathfrak{q}_k$, let $k'$ be the smallest integer such that $0\in\mathfrak{q}_{k'}$, and otherwise set $k'=k+1$.  Then $\xi_i$ is the holomorphic coordinate corresponding to the $\bbR^+$-action on $V_{\mathfrak{q}_i\setminus \mathfrak{q}_{i-1}}$, while

\begin{equation}
  \varpi_{q,j}= \left\{ \begin{array}{ll} \xi_i^{w_{q,j}}z_{q,j}, & q\in (\mathfrak{q}_{i}\setminus \mathfrak{q}_{i-1})\subset \mathfrak{q}_{k'}, \\
                       \xi_i^{w_{q,j}}(\xi_{k'}^{\nu w_{q,j}} z_{q,j}), & q\in \mathfrak{q}_i\setminus \mathfrak{q}_{i-1}, \quad k'<i,  \end{array} \right.
\label{kw.7}\end{equation}
and $\zeta_k= \{z_{q,j} \; | \; q\in \mathfrak{q}_k^c, \; j\in \{1,\ldots,m_q+n_q\}  \}$ if $k'=k+1$, and otherwise corresponds to the rescaled coordinates
$$
  \xi_{k'}^{\nu w_{q,j}}z_{q,j} \quad \mbox{for} \; q \in \mathfrak{q}_k^c,  \quad j\in \{1,\ldots,m_q+n_q\}.
$$ 
When $k'\le k$, we will further assume that in \eqref{kw.6}, the coordinate in $\varpi_{\mathfrak{q}_{k'}}$ that is omitted is  $\varpi_{0}=\xi_{k'}z_0$.  This will ensure that $\xi_{k'}$ can be seen as a holomorphic coordinate coming from the $\bbR^+$-action on $V_0\subset V_{\mathfrak{q}_{k'}\setminus\mathfrak{q}_{k'-1}}$.

In terms of \eqref{kw.6}, when $k'\le k$, a local basis of sections of ${}^{\mathfrak{n}}T^*\hX\otimes_{\bbR} \bbC$ is given by 
\begin{equation}
\xi_{k'}^{\nu}\frac{d\xi_1}{\xi_1^2}, \xi_{k'}^{\nu}\frac{d\varpi_{\mathfrak{q}_1}}{\xi_1}, \ldots, \xi_{k'}^\nu \frac{d\xi_{k'}}{\xi_{k'}^2}, \xi_{k'}^{\nu}\frac{d\varpi_{\mathfrak{q}_{k'}}}{\xi_{k'}}, \frac{d\xi_{k'+1}}{\xi_{k'+1}^2}, \frac{d\varpi_{\mathfrak{q}_{k'+1}}}{\xi_{k'+1}}, \ldots, \frac{d\xi_k}{\xi_k^2}, \frac{d\varpi_{\mathfrak{q}_k}}{\xi_k}, d\zeta_k,
\label{kw.8}\end{equation}
and their complex conjugates, where to lighten notation, $d\varpi_{\mathfrak{q}_i}$ stands for
$$
    \{ d\varpi_{q,j} \; | \; q\in\mathfrak{q}_i\setminus \mathfrak{q}_{i-1}, \; j\in \{1,\ldots,m_q+n_q\}\}
$$
with one $d\varpi_{q,j}$ omitted for each $d\varpi_{\mathfrak{q}_i}$.  From \eqref{kw.8}, we see that the $1$-forms
\begin{equation}
\frac{d\xi_1}{\xi_1^2}, \frac{d\varpi_{\mathfrak{q}_1}}{\xi_1}, \ldots,  \frac{d\xi_{k'}}{\xi_{k'}^2}, \frac{d\varpi_{\mathfrak{q}_{k'}}}{\xi_{k'}}, \xi_{k'}^{-\nu} \frac{d\xi_{k'+1}}{\xi_{k'+1}^2},  \xi_{k'}^{-\nu}\frac{d\varpi_{\mathfrak{q}_{k'+1}}}{\xi_{k'+1}}, \ldots,  \xi_{k'}^{-\nu} \frac{d\xi_k}{\xi_k^2},  \xi_{k'}^{-\nu}\frac{d\varpi_{\mathfrak{q}_k}}{\xi_k},  \xi_{k'}^{-\nu}d\zeta_k,
\label{kw.10}\end{equation}
together with their complex conjugates, form a local basis of ${}^{w}T\hX\otimes_{\bbR} \bbC$. 

If instead $k'=k+1$, then instead of \eqref{kw.8} we see that
\begin{equation}
\frac{d\xi_1}{\xi_1^2}, \frac{d\varpi_{\mathfrak{q}_1}}{\xi_1}, \ldots,  \frac{d\xi_k}{\xi_k^2}, \frac{d\varpi_{\mathfrak{q}_k}}{\xi_k}, d\zeta_k,
\label{kw.9}\end{equation}
together with their complex conjugates, form a local basis of sections of ${}^{\mathfrak{n}}T^*\hX\otimes_{\bbR} \bbC$ and ${}^{w}T\hX\otimes_{\bbR} \bbC$.  Notice that the coordinates \eqref{kw.6} are only valid in the interior of $\hX$ and are not coordinates on $\hX$ as a manifold with corners. However, as in \eqref{wqb.7}, the sections \eqref{kw.8} and \eqref{kw.9} naturally extend to the boundary of $\hX$ as sections of ${}^{\mathfrak{n}}T^*\hX\otimes \bbC$.  
 
 \begin{lemma}
 The complex structure $J$ of $\bbC^{m+n}$ naturally extends to a section
 $$
        J\in \CI(\hX;\End({}^wT\hX)) \quad \mbox{with} \quad J^2=-\Id.  
 $$
 \label{kw.11}\end{lemma}
\begin{proof}
This is clear from the local descriptions \eqref{kw.10} and \eqref{kw.9}.
\end{proof}
\begin{remark}
Since $\nQb$-metrics and $\nQAC$-metrics are conformal to $\mathfrak{n}$-warped $\QAC$-metrics, notice that in fact
$$
    J\in \CI(\hX;\End({}^wT\hX))=\CI(\hX;\End({}^{\mathfrak{n}}T\hX))=\CI(\hX;\End({}^{\nQb}T\hX)).
$$
\label{kw.12}\end{remark}
By restriction to $\hC_{\epsilon}$, Lemma~\ref{kw.11} shows that the complex structure $J_{\epsilon}$ of $C_{\epsilon}$ extends to a section
\begin{equation}
    J_{\epsilon}\in \CI(\hC_{\epsilon},\End({}^{w}T\hC_{\epsilon}))
\label{kw.13}\end{equation} 
which is `smooth' up to the boundary in the sense that it comes from the restriction of $J$ to $\hC_{\epsilon}$.

To show that a K\"ahler metric on $C_{\epsilon}$ is a warped $\QAC$-metric, Lemma~\ref{kw.11} indicates that it suffices to check that its K\"ahler form is an element of $\CI_{\nQb}(\hC_{\epsilon}; \Lambda^{1,1}({}^{w}T^*\hC_{\epsilon}\otimes_{\bbR} \bbC))$.
In the remainder of this section, we will construct examples of K\"ahler warped $\QAC$-metrics with $\pa\db$-exact K\"ahler forms, that is, with K\"ahler form $\omega= \frac{\sqrt{-1}}2 \pa \db U$ for some potential $U$.  To do so, as in \cite[Corollary~4.4]{CR2021}, we will use a convexity argument of van Coevering \cite[Lemma~4.3]{vanC} to extend a model at infinity to a K\"ahler form on the entire space.  
\begin{lemma}
Let $\hZ_{\epsilon}$ denote $\hC_{\epsilon}$, the $\nQAC$-compactification $\hW_{\mathfrak{q},\epsilon}$ of \eqref{pr.10} or the $\nQAC$-compactification $\hW^{\perp}_{\mathfrak{q},\omega_{\mathfrak{q}},\epsilon}$ of \eqref{pr.11} for $\epsilon\ne 0$ as in Assumption~\ref{cyc.6c}.  Suppose that $u_{\epsilon}$ is a smooth positive proper function  on $Z_{\epsilon}:=\hZ_{\epsilon}\setminus \pa\hZ_{\epsilon}$ such that 
$$
       \omega_{\epsilon}= \frac{\sqrt{-1}}2 \pa \db u_{\epsilon}
$$
is the K\"ahler form of a K\"ahler $\mathfrak{n}$-warped $\QAC$-metric $g_{\epsilon}\in \CI(\overline{\hZ_{\epsilon}\setminus K_{\epsilon}}; S^2({}^wT^*\hZ_{\epsilon}))$ on the complement of a compact set $K_{\epsilon}\subset Z_{\epsilon}$.  Then there exists a potential $\widetilde{u}_{\epsilon}$ defined on $Z_{\epsilon}$ and agreeing with $u_{\epsilon}$ outside a compact set such that 
$$
     \widetilde{\omega}_{\epsilon}= \frac{\sqrt{-1}}2 \pa\db \widetilde{u}_{\epsilon}
$$
is the K\"ahler form of a global K\"ahler $\mathfrak{n}$-warped $\QAC$-metric $g_{\epsilon}\in\CI(\hZ_{\epsilon};S^2({}^w T^*\hZ_{\epsilon}))$ on $Z_{\epsilon}$. 
\label{kw.14}\end{lemma}
\begin{proof}
By assumption $\frac{\sqrt{-1}}2 \pa \db u_{\epsilon}>0$ outside the compact set $K_{\epsilon}$.  Since $u_{\epsilon}$ is proper, this can be formulated as saying that there exists a constant $C>0$ such that $\frac{\sqrt{-1}}2\pa\db u_{\epsilon}>0$ for all $p\in Z_{\epsilon}$ such that $u_{\epsilon}(p)>C$.  Let $\eta\in\CI(\bbR)$ be a non-decreasing convex function such that
\begin{equation}
   \eta(t)= \left\{   \begin{array}{ll} t, & \mbox{if} \; t\ge C+2, \\ C+\frac{3}2, & \mbox{if} \; t\le C+1.   \end{array}   \right.
\label{kw.15}\end{equation}
For such a choice,  $\eta\circ u_{\epsilon}$ agrees with $u_{\epsilon}$ outside a compact set.  Since $\eta',\eta''\ge 0$, we see also that
\begin{equation}
\frac{\sqrt{-1}}2 \pa \db \eta\circ u_{\epsilon}= \frac{\sqrt{-1}}2 \eta''(u_{\epsilon})\pa u_{\epsilon}\wedge \db u_{\epsilon}+ \frac{\sqrt{-1}}{2} \eta'(u_{\epsilon}) \pa\db u_{\epsilon}\ge \frac{\sqrt{-1}}{2}\eta'(u_{\epsilon})\pa\db u_{\epsilon}\ge 0.
\label{kw.16}\end{equation}
Now, $Z_{\epsilon}$ is an affine variety in $\bbC^k$ for some $k\in\bbN$.  Let $w_{\epsilon}$ be the restriction to $Z_{\epsilon}$ of the Euclidean potential $ \sum_{i=1}^k |z_i|^2$ on $\bbC^k$ and let $\phi_{\epsilon}\in\CI_c(Z_{\epsilon})$ be a nonnegative function such that 
$$
   \phi_{\epsilon}(p)= \left\{  \begin{array}{ll} 1, & \mbox{if} \; u_{\epsilon}(p)\le C+2, \\ 0, & \mbox{if} \; u_{\epsilon}\ge C+3.  \end{array}  \right.
$$ 
We claim that it suffices to take 
$$
     \widetilde{u}_{\epsilon}:= \eta\circ u_{\epsilon} + \delta \phi w_{\epsilon}
$$
with $\delta>0$ sufficiently small.  Indeed, by construction, $\widetilde{u}_{\epsilon}$ is equal to $u_{\epsilon}$ outside a compact set.  By \eqref{kw.16}, at a point $p$ where $u_{\epsilon}(p)<C+2$,
$$
     \frac{\sqrt{-1}}2 \pa\db \widetilde{u}_{\epsilon} \ge \delta \frac{\sqrt{-1}}2 \pa \db w_{\epsilon}>0,
$$
while at a point $p$ where $u_{\epsilon}>C+3$, 
$$
      \frac{\sqrt{-1}}{2}\pa\db \widetilde{u}_{\epsilon} = \frac{\sqrt{-1}}2 \pa\db u_{\epsilon}>0.
$$
On the other hand, in the compact region where 
$$
     C+2\le u_{\epsilon}(p)\le C+3,
$$
notice that 
$$
  \frac{\sqrt{-1}}2 \pa\db (\eta\circ u_{\epsilon})= \frac{\sqrt{-1}}2 \pa \db u_{\epsilon}>0,
$$
so taking $\delta>0$ sufficiently small, we can ensure that $\frac{\sqrt{-1}}2 \pa \db \widetilde{u}_{\epsilon}>0$ in this region as well.
\end{proof}

If we set 
$$
    r_{\mathfrak{q}}:= \sqrt{\sum_{q\in\mathfrak{q}} r_q^2},
$$
then
$$
     \omega_{W_{\mathfrak{q}}}:= \frac{\sqrt{-1}}2 \pa \db r^2_{\mathfrak{q}}
$$
is the K\"ahler form of the natural Calabi-Yau cone metric on $W_{\mathfrak{q}}$.  For each $q$, consider a smooth extension $r_q'$ of $r_q$ to $V_q\setminus\{0\}$ as a homogeneous positive function of degree $1$ with respect to the $\bbR^+$-action.  Cut it off near $\{0\}$ using a partition of unity and denote by $\widetilde{r}_q$ the resulting function.  This function is smooth, but no longer homogeneous.  However, it agrees with $r_q'$ outside a compact set. More generally, let
$$
     \widetilde{r}_{\mathfrak{q}}= \sqrt{\sum_{q\in \mathfrak{q}} \widetilde{r}_q^2}
$$ 
be the corresponding function on $V_{\mathfrak{q}}$.  For $\mathfrak{q}=\{0,\ldots,N\}$, we will also use the notation
$$
     \widetilde{r}= \sqrt{\sum_{q=0}^N \widetilde{r}_q^2}.
$$
\begin{lemma}
The function $\frac{1}{\widetilde{r}}$ is an $\mathfrak{n}$-weighted total boundary defining function $\nQAC$-equivalent to  that of \eqref{kw.17b}.
\label{kw.17}\end{lemma}
\begin{proof}
Let $\mathfrak{q}_1\subset \cdots \subset \mathfrak{q}_k\subset \{0,\ldots,N\}$ be a sequence of embedded subsets and let $\hH_i\in\cM_1(\hX)$ be the boundary hypersurface corresponding to $V_{\mathfrak{q}_i}$.  Then near a compact region of the interior of $\hH_1\cap\ldots\cap \hH_k$, the function $\widetilde{r}_{\mathfrak{q}_k}$ is homogeneous of degree $1$, so $\frac{1}{\widetilde{r}_{\mathfrak{q}_k}}$ is clearly $\nQAC$-equivalent to \eqref{kw.17b}.  On the other hand, 
$$
   \frac{1}{\widetilde{r}}= \frac{1}{\sqrt{\widetilde{r}^2_{\mathfrak{q}_k}+  \widetilde{r}^2_{\mathfrak{q}^c_k}}}= \frac{1}{\widetilde{r}_{\mathfrak{q}_k}} \frac{1}{\sqrt{1+ \frac{\widetilde{r}^2_{\mathfrak{q}^c_k}}{\widetilde{r}^2_{\mathfrak{q}_k}}}  }
$$
and clearly $\frac{\widetilde{r}^2_{\mathfrak{q}^c_k}}{r^2_{\mathfrak{q}}}=0$ on $\hH_1, \ldots, \hH_{k-1}$ and $\hH_k$.  By Lemma~\ref{wqb.5}, $\frac{1}{\widetilde{r}}$ is $\nQAC$-equivalent to $\frac{1}{\widetilde{r}_{\mathfrak{q}_k}}$ and \eqref{kw.17b} near a compact region of the interior of $\hH_1\cap\ldots \cap \hH_{k}$.  Since $\hH_1\cap\ldots \cap \hH_{k}$ was an arbitrary corner of $\hX$, the result follows. 
\end{proof}

On $\hC_0$, notice that $\widetilde{r}^2$ agrees with $r^2$ near the maximal boundary hypersurface $\hH_{\max}$ of $\hX$, so
$$
        \frac{ \sqrt{-1}}{2} \pa\db \widetilde{r}^2= \omega_{C_0}
$$
on $\hC_0$ near $\hH_{\max}$.  In fact, in the coordinates \eqref{kw.6} near $H_{\max}$, so with $H_k=H_{\max}$ and with no coordinate $\zeta_k$, the potential $\widetilde{r}^2$ takes the form 
\begin{equation}
  \widetilde{r}^2= \sum_{i=1}^{k'} |\xi_i|^{-2}f_i(\varpi_{\mathfrak{q}_i})+ |\xi_{k'}|^{-2\nu} \sum_{i=k'+1}^k f_i(\varpi_{\mathfrak{q}_i})
\label{kw.18}\end{equation}
for some smooth functions $f_i$.  Clearly, in terms of \eqref{kw.10}, using the fact that $1-\nu>0$, we see that 
$$
    \frac{\sqrt{-1}}2 \pa\db\widetilde{r}^2\in \CI(\hX; \Lambda^{1,1}({}^{w}T^*\hX\otimes_{\bbR} \bbC)).
$$
Since $\hC_{\epsilon}$ and $\hC_0$ are tangent to order $\widehat{x}_{\max}^d$ at $\hH_{\max}$, we see also that $\frac{\sqrt{-1}}2 \pa\db \widetilde{r}^2$ remains positive definite on $\hC_{\epsilon}$ near $\hH_{\max}$, hence defines there a K\"ahler $\mathfrak{n}$-warped $\QAC$-metric.  However, there is no reason a priori for $\frac{\sqrt{-1}}{2}\pa\db \widetilde{r}^2$ to be positive definite everywhere on $\hC_{\epsilon}$.  Nevertheless, using Lemma~\ref{kw.14}, we will modify $\widetilde{r}^2$ to obtain the potential of a K\"ahler $\mathfrak{n}$-warped $\QAC$-metric on $\hC_{\epsilon}$.  
\begin{theorem}
There exists a smooth positive function $\phi_{\epsilon}$ on $\hC_{\epsilon}$ agreeing with $\widetilde{r}^2$ near $\hH_{\max}$ such that $\frac{\sqrt{-1}}2\pa\db \phi_{\epsilon}$ is the K\"ahler form of an $\mathfrak{n}$-warped $\QAC$-metric on $\hC_{\epsilon}$.  Moreover, if $\hH$ corresponds to the subset $\mathfrak{q}$, then in terms of the decomposition $\bbC^{m+n}= V_{\mathfrak{q}}\times V_{\mathfrak{q}^c}$, near $\hH$, 
\begin{equation}
            \phi_{\epsilon}= \widetilde{r}_{\mathfrak{q}}^2+ \phi_{\mathfrak{q}^c}(z_{\mathfrak{q}^c},\overline{z}_{\mathfrak{q}^c} )
\label{kw.19a}\end{equation}
if $H\in\cM_{1,0}(\hX)$, where $\frac{\sqrt{-1}}{2}\pa \db \phi_{\mathfrak{q}^c}(z_{\mathfrak{q}^c},\overline{z}_{\mathfrak{q}^c} )$ is the K\"ahler form of an $\mathfrak{n}$-warped $\QAC$-metric on $W_{\mathfrak{q}^c,\epsilon}$.  If instead $H\in \cM_{1,\nu}(\hX)$ with $\nu>0$, then 
\begin{equation}
 \phi_{\epsilon}= \widetilde{r}_{\mathfrak{q}}^2+ |z_0|^{2\nu}\phi_{\mathfrak{q}^c}(\zeta_{\mathfrak{q}^c},\overline{\zeta}_{\mathfrak{q}^c} ),
\label{kw.19b}\end{equation}
where $\zeta_{\mathfrak{q}^c}$ are the rescaled coordinates of \eqref{cyc.14} in $V_{\mathfrak{q}^c}$ and $\frac{\sqrt{-1}}{2}\pa \db \phi_{\mathfrak{q}^c}(\zeta_{\mathfrak{q}^c},\overline{\zeta}_{\mathfrak{q}^c} )$ is the K\"ahler form of a $\QAC$-metric on $W_{\mathfrak{q}^c, \omega_{\mathfrak{q}},\epsilon}$. 

\label{kw.19}\end{theorem}
\begin{proof}
Let $\hH\in\cM_1(\hX)$ be a boundary hypersurface such that $\hH<\hH_{\max}$ and 
$$
    \hH<\hG\; \Longrightarrow \; \hG=\hH_{\max}.  
$$
The fibers of $\phi_{\hH}:  \hH\to S_{\hH}$ are then manifolds with boundary and $\hH\in\cM_{1,\nu}(\hX)$.  If $\hH$ corresponds to the subset $\mathfrak{q}\subset \{0,\ldots,N\}$, then in terms of the decomposition 
$$
   \bbC^{m+n}= V_{\mathfrak{q}}\times V_{\mathfrak{q}^c},
$$
we have that 
\begin{equation}
         \widetilde{r}^2 = \widetilde{r}^2_{\mathfrak{q}}+ \widetilde{r}^2_{\mathfrak{q}^c}.
\label{kw.20}\end{equation}
Now, the function $\widetilde{r}_{\mathfrak{q}_c}$ agrees with $r_{\mathfrak{q}^c}$ outside a compact set and so is homogenous there.  In terms of the coordinates \eqref{kw.6} with $k=1$, $\hH_1=\hH$ and $\xi_1$ is the function of the $\bbR^+$-action on $V_{\{0\}}$.   This means that in this region,
$$
    \widetilde{r}_{\mathfrak{q}_c}(z_{\mathfrak{q}^c})= |\xi_1|^{-2\nu}\widetilde{r}_{\mathfrak{q}^c}(\zeta_{\mathfrak{q}^c}).
$$
By Lemma~\ref{kw.14}, we can find $\phi_{\mathfrak{q}}(\zeta_{\mathfrak{q}^c},\overline{\zeta}_{\mathfrak{q}^c})$ with $\phi_{\mathfrak{q}^c}= \widetilde{r}_{\mathfrak{q}^c}^2$ outside a compact set such that 
$\frac{\sqrt{-1}}2 \pa\db \phi_{\mathfrak{q}^c}(\zeta_{\mathfrak{q}^c},\overline{\zeta}_{\mathfrak{q}^c} )$
is the K\"ahler form of a $\QAC$-metric, in fact of an $\AC$-metric, on $W_{\mathfrak{q}^c,\omega_{\mathfrak{q}},\epsilon}$ for each $\omega_{\mathfrak{q}}$ and $\epsilon \ne 0$.  Now, replacing $\widetilde{r}_{\mathfrak{q}}^2$ in \eqref{kw.20} by $|\xi_1|^{-2\nu}\phi_{\mathfrak{q}^c}(\zeta_{\mathfrak{q}^c},\overline{\zeta}_{\mathfrak{q}^c} )$, we obtain a potential 
$$
    \psi_{\epsilon}:= \widetilde{r}^2_{\mathfrak{q}}+ |\xi_1|^{-2\nu}\phi_{\mathfrak{q}^c}(\zeta_{\mathfrak{q}^c},\overline{\zeta}_{\mathfrak{q}^c} )
$$
on $\hC_{\epsilon}$ which agrees with $\widetilde{r}^2$ near $\hH_{\max}$ and such that $\frac{\sqrt{-1}}2\pa\db\psi_{\epsilon}$ is the K\"ahler form of an $\mathfrak{n}$-warped $\QAC$-metric not only near $\hH_{\max}$, but also near $\hH$.  Indeed, in any coordinates \eqref{kw.6} near $\hH$, $\xi_{k'}$ corresponds to the coordinate of the $\bbR^+$-action on $V_{0}$ and we compute that 
\begin{equation}
\begin{aligned}
  \frac{\sqrt{-1}}2 \pa\db |\xi_{k'}|^{-2\nu}\phi_{\mathfrak{q}^c}(\zeta_{\mathfrak{q}^c},\overline{\zeta}_{\mathfrak{q}^c})&= |\xi_{k'}|^{-2\nu}\frac{\sqrt{-1}}2 \left(\pa\db \phi_{\mathfrak{q}^c} + -\nu \pa\phi_{\mathfrak{q}^c}\wedge \overline{\xi}_{k'}\frac{d\overline{\xi}_{k'}}{\overline{\xi}_{k'}^2} - \nu \xi_{k'}\frac{d\xi_{k'}}{\xi_{k'}^2} \wedge \db\phi_{\mathfrak{q}^c} \right. \\
  & \hspace{4cm} \left. +\nu^2\phi_{\mathfrak{q}^c} \xi_{k'}  \frac{d\xi_{k'}}{\xi_{k'}^2}\wedge \overline{\xi}_{k'}\frac{d\overline{\xi}_{k'}}{\overline{\xi}_{k'}^2}  \right) \\
  &= |\xi_{k'}|^{-2\nu}\frac{\sqrt{-1}}2 \pa\db \phi_{\mathfrak{q}^c} + |\xi_{k'}|^{1-\nu} x_{\max}^{-1}\cA_{\phg}(\hX;\Lambda^{1,1}({}^{w}T^*\hX\otimes_{\bbR}\bbC)) \\
  &= |\xi_{k'}|^{-2\nu}\frac{\sqrt{-1}}2 \pa\db \phi_{\mathfrak{q}^c}+ x_{\hH}\cA_{\phg}(\hX;\Lambda^{1,1}({}^{w}T^*\hX\otimes_{\bbR}\bbC)).
\end{aligned}  
\label{kw.21}\end{equation}
In particular, since $\frac{\sqrt{-1}}2 \pa\db \phi_{\mathfrak{q}^c}$ is a K\"ahler form on $W_{\mathfrak{q}^c,\omega_{\mathfrak{q}},\epsilon}$, this shows that $\frac{\sqrt{-1}}2 \pa\db \psi_{\epsilon}$ is an $\mathfrak{n}$-warped $\QAC$-K\"ahler form when $x_{\hH}$ is small enough, that is, for points sufficiently close to $\hH$.  We can iterate this construction near each boundary hypersurface of $\hC_{\epsilon}$ proceeding in a non-increasing order  with respect to the partial order of $\cM_1(\hX)$.  More precisely, fix $\hH\in\cM_1(\hX)$ and assume that we have a potential $\psi_{\epsilon}$ agreeing with $\widetilde{r}^2$ near $\hH_{\max}$ and such that $\frac{\sqrt{-1}}2 \pa\db \psi_{\epsilon}$ is a warped $\QAC$-K\"ahler form near $\hG$ for each $\hG>\hH$.  Suppose that $\hH$ corresponds to the subset $\mathfrak{q}\subset \{0,\ldots,N\}$.  In terms of the decomposition $\bbC^{m+n}=V_{\mathfrak{q}}\times V_{\mathfrak{q}^c}$, we can assume by induction that near $\hH\cap\lrp{\bigcup_{\hG>\hH}\hG}$,
\begin{equation}
     \psi_{\epsilon}= \widetilde{r}_{\mathfrak{q}}^2+ \psi_{\epsilon,\mathfrak{q}^c}
\label{kw.22}\end{equation}
with $\frac{\sqrt{-1}}2\pa\db \widetilde{r}^2_{\mathfrak{q}}$ a warped $\QAC$-K\"ahler form on $\hW_{\mathfrak{q}}$ near $S_{\hH}$ seen as a boundary hypersurface of $\hV_{\mathfrak{q}}$.  If $0\in \mathfrak{q}$, then using the coordinates \eqref{kw.6} with $\hH_k=\hH$, $\xi_{k'}$ a coordinate corresponding to the $\bbR^+$-action on $V_{\{0\}}$ and $\zeta_k=\zeta_{\mathfrak{q}^c}$ corresponding to a rescaled coordinate on $V_{\mathfrak{q}^c}$, we can suppose more precisely that outside a compact set of $V_{\mathfrak{q}^c}$,
$$
        \psi_{\epsilon,\mathfrak{q}^c}(\xi_{k'}, \overline{\xi}_{k'},\zeta_{\mathfrak{q}^c},\overline{\zeta}_{\mathfrak{q}^c}) := |\xi_{k'}|^{-2\nu}\psi_{\epsilon,\mathfrak{q}^c}(\zeta_{\mathfrak{q}^c},\overline{\zeta}_{\mathfrak{q}^c})
$$
with $\frac{\sqrt{-1}}2 \pa\db \psi_{\epsilon,\mathfrak{q}^c}(\zeta_{\mathfrak{q}^c},\overline{\zeta}_{\mathfrak{q}^c})$ a K\"ahler form of a $\QAC$-metric outside a compact set of $W_{\mathfrak{q}^c,\omega_{\mathfrak{q}},\epsilon}$.  Using Lemma~\ref{kw.14}, we can change $\psi_{\epsilon,\mathfrak{q}^c}$ to a function $\phi_{\epsilon,\mathfrak{q}^c}$ within a compact set so that $\frac{\sqrt{-1}}2 \pa \db \phi_{\epsilon,\mathfrak{q}^c}(\zeta_{\mathfrak{q}^c},\overline{\zeta}_{\mathfrak{q}^c})$ is positive definite everywhere on $W_{\mathfrak{q}^c,\omega_{\mathfrak{q}},\epsilon}$ for each $\omega_{\mathfrak{q}}$ and for $\epsilon\ne 0$ as in Assumption~\ref{cyc.6c}.  We can then take the new potential
$$
  \widetilde{\psi}_{\epsilon}= \widetilde{r}^2_{\mathfrak{q}}+ |\xi_{k'}|^{-2\nu}\phi_{\epsilon,\mathfrak{q}^c}(\zeta_{\mathfrak{q}^c},\overline{\zeta}_{\mathfrak{q}^c}).
$$
By construction, $\widetilde{\psi}_{\epsilon}=\psi_{\epsilon}$ near $\bigcup_{\hG>\hH} \hG$ and as in \eqref{kw.21}, one can check that 
\begin{equation}
\frac{\sqrt{-1}}2 \pa\db \widetilde{\psi}_{\epsilon}= \frac{\sqrt{-1}}2 \pa\db \widetilde{r}^2_{\mathfrak{q}}+|\xi_{k'}|^{-2\nu}\frac{\sqrt{-1}}2 \pa\db \phi_{\mathfrak{q}^c}+ x_{\hH}\cA_{\phg}(\hX;\Lambda^{1,1}({}^{w}T^*\hX\otimes_{\bbR}\bbC)),
\label{kw.22b}\end{equation} 
so is an $\mathfrak{n}$-warped $\QAC$-K\"ahler form on $\hC_{\epsilon}$ near $\hH$.  If instead $0\notin \mathfrak{q}$, then we still have a decomposition \eqref{kw.22}, but with the difference that on $V_{\mathfrak{q}^c}$, we can directly use the coordinates $z_{\mathfrak{q}^c}$ instead of the rescaled coordinates $\zeta_{\mathfrak{q}^c}$ and assume that
$$
      \psi_{\epsilon,\mathfrak{q}^c}= \psi_{\epsilon,\mathfrak{q}^c}(z_{\mathfrak{q}^c},\overline{z}_{\mathfrak{q}^c}).
$$
  Moreover, this time, $\frac{\sqrt{-1}}2 \pa\db \psi_{\epsilon,\mathfrak{q}^c}(z_{\mathfrak{q}^c},\overline{z}_{\mathfrak{q}^c})$ is an $\mathfrak{n}$-warped $\QAC$-metric outside a compact set of $W_{\mathfrak{q}^c,\epsilon}$.  We can therefore use again Lemma~\ref{kw.14} to change $\psi_{\epsilon,\mathfrak{q}^c}$ on a compact set and obtain a new potential $\phi_{\epsilon,\mathfrak{q}^c}$ with $\frac{\sqrt{-1}}2 \pa \db \phi_{\epsilon,\mathfrak{q}^c}$ positive definite everywhere on $W_{\mathfrak{q}^c,\epsilon}$.  It suffices then to take 
$$
    \widetilde{\psi}_{\epsilon}= \psi_{\epsilon,\mathfrak{q}}+ \phi_{\epsilon,\mathfrak{q}^c}.
$$
Indeed, by construction, $\widetilde{\psi}_{\epsilon}=\psi_{\epsilon}$ near $\bigcup_{\hG>\hH} \hG$ and we can easily check that $\frac{\sqrt{-1}}2 \pa\db \widetilde{\psi}_{\epsilon}$ is also the K\"ahler form of an $\mathfrak{n}$-warped $\QAC$-metric on $\hC_{\epsilon}$ near $\hH$.  

This completes the inductive step and shows that we can find a potential $\psi_{\epsilon}$ agreeing with $\widetilde{r}^2$ near $\hH_{\max}$ and such that $\frac{\sqrt{-1}}2\pa\db \psi_{\epsilon}$ is the K\"ahler form of an $\mathfrak{n}$-warped $\QAC$-metric on $\hC_{\epsilon}$ outside a compact set.  Using Lemma~\ref{kw.14} one last time, we can thus find a potential $\phi_{\epsilon}$ agreeing with $\psi_{\epsilon}$ outside a compact set such that $\frac{\sqrt{-1}}2\pa\db \phi_{\epsilon}$ is the K\"ahler form of an $\mathfrak{n}$-warped $\QAC$-metric on $\hC_{\epsilon}$.

\end{proof}
\begin{remark}
In the setting of Example~\ref{cr.1}, the proof of Theorem~\ref{kw.19} can also be easily adapted to construct K\"ahler $\mathfrak{n}$-warped  $\QAC$-metrics on $K_{F_k}$ in any compactly supported K\"ahler class.  Indeed, the proof of Theorem~\ref{kw.19} at least yields such a metric away from the zero section.  Hence, to obtain one also well-defined near the zero section of $K_{F_k}$,  it suffices to apply the convexity argument of the proof Lemma~\ref{kw.14} with the Euclidean metric replaced by any K\"ahler metric on $K_{F_k}$ with compactly supported K\"ahler class, since by \cite[Corollary~A.3]{CH2013}, such a K\"ahler form is $\pa\db$-exact away from the zero section.
\label{cr.2}\end{remark}

Let us formalize the type of metrics that we have obtained in the following analog of \cite[Definition~3.6]{CDR2016}.

\begin{definition}
A K\"ahler $\mathfrak{n}$-warped $\QAC$-metric $g\in \CI_{\nQb}(C_{\epsilon}; {}^{w}T^*\hC_{\epsilon} \otimes {}^{w}T^*\hC_{\epsilon})$ is \textbf{asymptotic with rate $\delta$} to the Calabi-Yau cone metric $\omega_{C_0}$ if:
\begin{enumerate}
\item Near $\hH_{\max}$, $\omega-\frac{\sqrt{-1}}2 \pa\db\widetilde{r}^2\in \widehat{x}_{\max}^{\delta}\CI_{\nQb}(C_{\epsilon}; \Lambda^{1,1}({}^wT^*\hC_{\epsilon}))$;
\item Near $\hH\in \cM_{1}(\hX)$ corresponding to the subset $\mathfrak{q}$, 
$$
         \omega- \frac{\sqrt{-1}}2\pa\db \widetilde{r}_{\mathfrak{q}}^2- |z_0|^{2\nu_{\hH}}\omega_{\hH}\in x_{\hH}\CI_{\nQb}(\hC_{\epsilon}; \Lambda^{1,1}({}^wT^*\hC_{\epsilon}))
$$
with $\omega_{\hH}$ a closed $(1,1)$-form on $\hH$ which restricts on each fiber $\hW_{\mathfrak{q}^c,\varpi_{\mathfrak{q}},\epsilon}$ of $\phi_{\hH_{\epsilon}}: \hH_{\epsilon}\to S_{\hH}$ to the K\"ahler form of a K\"ahler $\mathfrak{n}_{Z_{\hH_{\epsilon}}}$-warped $\QAC$-metric asymptotic with rate $\delta$ to $g_{W_{\mathfrak{q}^c}}$, the natural Calabi-Yau cone metric on $W_{\mathfrak{q}^c}$ obtained by restricting $g_{C_0}$.  Moreover, as a family of $(1,1)$-forms parametrized by $S_{\hH_{\epsilon}}$, $\omega_{\hH}$ is smooth up to $\pa S_{\hH_{\epsilon}}$.
\end{enumerate}
Notice that the definition is not circular, since by induction on the depth of $\hC_{\epsilon}$, we can assume that the notion of a K\"ahler $\mathfrak{n}_{Z_{\hH_{\epsilon}}}$-warped $\QAC$-metric asymptotic to $g_{W_{\mathfrak{q}^c}}$ with rate $\delta$ has already been defined.
\label{kw.23}\end{definition}
This yields the following characterization of the K\"ahler metrics of Theorem~\ref{kw.19}.
\begin{corollary}
 If $d>1$, then the K\"ahler $\mathfrak{n}$-warped $\QAC$-metrics of Theorem~\ref{kw.19} are asymptotic with rate $d$ to the Calabi-Yau cone metric $g_{C_0}$.
\label{kw.24}\end{corollary}
\begin{proof}
By construction, $\phi_{\epsilon}$ agrees with $\widetilde{r}^2$ near $\hH_{\max}$, which implies (1) in Definition~\ref{kw.23}.  Computations as in \eqref{kw.22b}, \eqref{kw.19a} and \eqref{kw.19b} imply (2) of Definition~\ref{kw.23} so the result follows.  
\end{proof}

To construct Calabi-Yau examples, the key property of the K\"ahler metrics of Definition~\ref{kw.23} that we will use can be formulated in terms of the following analog of \cite[Definition~5.1]{CDR2016}.
\begin{definition}
For $\hH<\hH_{\max}$, let $\CI_{\nQb}(\hH_{\epsilon}/S_{\hH_{\epsilon}})$ be the space of smooth functions on 
$$
        \hH_{\epsilon}\setminus \lrp{ \bigcup_{\hG_{\epsilon}>\hH_{\epsilon}} \hG_{\epsilon}\cap\hH_{\epsilon} }
$$
which restricts on each fiber $\phi_{\hH_{\epsilon}}^{-1}(s)$ to a function in $\CI_{\nQb}(\phi_{\hH_{\epsilon}}^{-1}(s))$.  A function $f\in \widehat{x}_{\max}^{\delta}\CI_{\nQb}(C_{\epsilon})$ is said to \textbf{restrict} to $\pa \hC_{\epsilon}$ to order $r>0$ if for each $\hH<\hH_{\max}$, there is $f_{\hH_{\epsilon}}\in \widehat{x}_{\max}^{\delta}\CI_{\nQb}(\hH_{\epsilon}/S_{\hH_{\epsilon}})$ such that
$$
        f-f_{\hH_{\epsilon}}\in \widehat{x}_{\max}^{\delta} x_{\hH}^r\CI_{\nQb}(C_{\epsilon}).
$$
We denote by $ \widehat{x}_{\max}^{\delta}\CI_{\nQb,r}(C_{\epsilon})$ the subspace of functions in $ \widehat{x}_{\max}^{\delta}\CI_{\nQb}(C_{\epsilon})$ that restrict to $\pa \hC_{\epsilon}$ to order $r$.
\label{qbr.1}\end{definition}
Indeed, as the next lemma shows, the Ricci potential of the K\"ahler metrics of Definition~\ref{kw.23} is an example of such a function.  

\begin{lemma}
If $d>1$ and $\omega_{\epsilon}$ is the K\"ahler form of  a K\"ahler $\mathfrak{n}$-warped $\QAC$-metric asymptotic to $g_{C_0}$ with rate $d$, then
its Ricci potential
\begin{equation}
            \mathfrak{r}_{\epsilon}= \log\lrp{ \frac{\omega_{\epsilon}^m}{c_m \Omega^m_{C_{\epsilon}}\wedge \overline{\Omega}^m_{C_{\epsilon}}}   }
\label{kw.25b}\end{equation}
is an element of $\widehat{x}_{\max}^{d}\CI_{\nQb,1}(C_{\epsilon}).$  
\label{kw.25}\end{lemma}
\begin{proof}
Let us first check that $\mathfrak{r}_{\epsilon}$  restricts to $\pa\hC_{\epsilon}$ and consider the forms
\begin{equation}
dz:= dz_0\wedge (dz_{1,1}\wedge\ldots\wedge dz_{1,m_1+n_1})\wedge\ldots\wedge(dz_{N,1}\wedge\ldots\wedge dz_{N,m_N+n_N})
\label{kw.27}\end{equation}
and 
\begin{equation}
  \cP_{\epsilon}:= (d(P_{1,1}-\epsilon Q_{1,1})\wedge \ldots d(P_{1,n_1}-\epsilon Q_{1,n_1}))\wedge\ldots\wedge (d(P_{N,1}-\epsilon Q_{N,1})\wedge \ldots\wedge d(P_{N,n_N}-\epsilon Q_{N,n_N})). 
\label{kw.28}\end{equation}
For $q\in\{1,\ldots,N\}$, consider also the form
$$
  \cP_{\epsilon,q}:= d(P_{q,1}-\epsilon Q_{q,1})\wedge \ldots d(P_{q,n_q}-\epsilon Q_{q,n_q}).
$$
Then from \eqref{pr.3d} and the definition of $\mathfrak{r}_{\epsilon}$, we see that there exists a constant $c_{m,n}\in\bbC^*$ depending only on $m$ and $n$ such that
\begin{equation}
   \mathfrak{r}_{\epsilon}= \left. \log \lrp{ \frac{\omega_{\epsilon}^m\wedge \cP_{\epsilon}\wedge \overline{\cP}_{\epsilon}}{c_{m,n} dz\wedge d\overline{z}}   }\right|_{\hC_{\epsilon}}.
\label{kw.29}\end{equation}
To study the behavior of $\mathfrak{r}_{\epsilon}$ near $\hH\in \cM_1(\hX)$, let $\mathfrak{q}\subset \{0,\ldots,N\}$ be the subset corresponding to $\hH$.  If $0\notin \mathfrak{q}$, that is, if $\hH\in\cM_{1,0}(\hX)$, then in the coordinates \eqref{kw.1}, we see that for $q\in \mathfrak{q}^c\setminus \{0\}$, the $1$-form
$$
      d(P_{q,j}(z_q)-\epsilon Q_{q,j}(z_0))
$$
naturally restricts to $\hH$ while for $q\in \mathfrak{q}$, 
$$
      \cP_{\epsilon,q}= \cP_{0,q}+ \mathcal{O}(\xi^d),
$$
where $\mathcal{O}(\xi^d)$ corresponds to a sum of $n_q$-forms, each homogeneous in $\xi$ of degree at least $\displaystyle \left(-\sum_{j=1}^{n_q}d_{q,j}\right)+d $ (instead of $\displaystyle -\sum_{j=1}^{n_q}d_{q,j}$ for $\cP_{q,0}$ ).  Keeping in mind \eqref{cyc.4}, the latter is to be compared with
$$
   dz_{q,j}= d(\xi^{-w_{q,j}}\varpi_{q,j}).
$$
If instead $0\in \mathfrak{q}$, then in terms of the coordinates used in \eqref{kw.4}, we see that for $q\in\mathfrak{q}^c$ and $j\le k_q$, 
\begin{equation}
\begin{aligned}
d(P_{q,j}(z_q)-\epsilon Q_{q,j}(z_0))&=  d \lrp{ \xi^{-\nu d}P_{q,j}(\zeta_q)-\epsilon Q_{q,j}(\xi^{-1}\varpi_0)   } = d \lrp{ \xi^{-\ell}P_{q,j}(\zeta_q)-\epsilon Q_{q,j}(\xi^{-1}\varpi_0)   }\\
   &= d \lrp{ \xi^{-\ell}(P_{q,j}(\zeta_q)-\epsilon [Q_{q,j}](\xi^{-1}\varpi_0))   } + \mathcal{O}(\xi^{\alpha+1-\ell}), \quad \mbox{for some}\; \alpha>0, \\
   &= d \lrp{ \xi^{-\ell}(P_{q,j}(\zeta_q)-\epsilon [Q_{q,j}](\xi^{-1}\varpi_0))   } + \mathcal{O}(x_{\hH}^{\frac{(\alpha+1)d}{d-\ell}}\xi^{-\ell}),
\end{aligned}
\label{kw.30}\end{equation}
where $ \mathcal{O}(\xi^{\alpha+1-\ell})=  \mathcal{O}(x_{\hH}^{\frac{(\alpha+1)d}{d-\ell}}\xi^{-\ell})$ is with respect to the local sections of $1$-forms $\frac{d\xi}{\xi^2}$ and $\frac{d\varpi_{0,j}}{\xi }$ in \eqref{kw.4}.  If $0\in\mathfrak{q}$ and $q\in \mathfrak{q}^c$, but $j>k_q$, then
\begin{equation}
\begin{aligned}
d(P_{q,j}(z_q)-\epsilon Q_{q,j}(z_0))&=  d \lrp{ \xi^{-\nu d_{q,j}}P_{q,j}(\zeta_q)-\epsilon Q_{q,j}(\xi^{-1}\varpi_0)   } \\
   &= d \lrp{ \xi^{-\nu d_{q,j}}(P_{q,j}(\zeta_q)} + \mathcal{O}(\xi^{\alpha+1-\nu d_{q,j}}), \quad \mbox{for some}\; \alpha>0, \\
   &= d \lrp{ \xi^{-\nu d_{q,j}}(P_{q,j}(\zeta_q)  } + \mathcal{O}(x_{\hH}^{\frac{(\alpha+1)d}{d-\ell}}\xi^{-\nu d_{q,j}}),
\end{aligned}
\label{kw.30a}\end{equation}
where $ \mathcal{O}(\xi^{\alpha+1-\nu d_{q,j}})=  \mathcal{O}(x_{\hH}^{\frac{(\alpha+1)d}{d-\ell}}\xi^{-\nu d_{q,j}})$ is again with respect to the local sections of $1$-forms $\frac{d\xi}{\xi^2}$ and $\frac{d\varpi_{0,j}}{\xi }$ in \eqref{kw.4}. 
For $q\in\mathfrak{q}$, we have instead that 
\begin{equation}
   d(P_{q,j}(z_q)-\epsilon Q_{q,j}(z_0))= d(\xi^{-d_{q,j}}P_{q,j}(\varpi))-\epsilon dQ_{q,j}(\xi^{-1}\varpi_0), 
\label{kw.31}\end{equation}
so that 
\begin{equation}
      \cP_{\epsilon,q}= \cP_{0,q}+ \mathcal{O}(\xi^{d-\ell})=\cP_{0,q}+ \mathcal{O}(x_{\hH}^d)
\label{kw.32}\end{equation}
with $\mathcal{O}(\xi^{d-\ell})=\mathcal{O}(x_{\hH}^d)$ corresponding to a sum of $n_q$-forms, each homogeneous in $\xi$ of degree at least $\displaystyle -\lrp{\sum_{j=1}^{n_q}d_{q,j}}+ d-\ell$, instead of $\displaystyle -\sum_{j=1}^{n_q}d_{q,j}$ for $\cP_{0,q}$.  Thanks to property (2) of Definition~\ref{kw.23} and keeping in mind \eqref{cyc.4}, we see from \eqref{kw.29} that $\mathfrak{r}_{\epsilon}$ restricts to $\pa C_{\epsilon}$ to order $1$ in both coordinate charts.  Similar computations can be done in the coordinates \eqref{kw.6}, from which the result follows.

At $\hH_{\max}$, notice that $\mathfrak{q}=\{0,\ldots,N\}$, so $\mathfrak{q}^c=\emptyset$, that is, \eqref{kw.30} does not arise and we only need to consider \eqref{kw.32}.  On the other hand, since $\omega_{C_{0}}$ is Calabi-Yau, we know that 
\begin{equation}
            \mathfrak{r}_{0}= \log\lrp{ \frac{\omega_{C_0}^m}{c_m \Omega^m_{C_{0}}\wedge \overline{\Omega}^m_{C_{0}}}   }=0.
\label{kw.26}\end{equation}
Combined with property (1) of Definition~\ref{kw.23}, this gives the desired decay at $\hH_{\max}$.
\end{proof}

\section{Solving the complex Monge-Amp\`ere equation} \label{mae.0}

Suppose that $d>1$ and let $g_{\epsilon}$ be a K\"ahler $\mathfrak{n}$-warped $\QAC$-metric asymptotic to the Calabi-Yau cone metric $g_{C_0}$ with rate $d$.  By Theorem~\ref{kw.19} and Corollary~\ref{kw.24}, such a metric exists provided $d>1$.  If $\omega_{\epsilon}$ is the K\"ahler form of $g_{\epsilon}$,  then $\widetilde{\omega}_{\epsilon}:= \omega_{\epsilon}+\sqrt{-1}\pa\db u$ will be the K\"ahler form of a Calabi-Yau metric provided $u$ is a solution of the complex Monge-Amp\`ere equation
\begin{equation}
 \log\lrp{ \frac{(\omega_{\epsilon}+\sqrt{-1}\pa\db u)^m}{\omega_{\epsilon}^m}  }= -\mathfrak{r}_{\epsilon},
\label{mae.1}\end{equation}
where $\mathfrak{r}_{\epsilon}$ is the Ricci potential of $g_{\epsilon}$ as defined in \eqref{kw.25b}.  By Lemma~\ref{kw.25}, this is a particular case of the more general complex Monge-Amp\`ere equation
\begin{equation}
 \log\lrp{ \frac{(\omega_{\epsilon}+\sqrt{-1}\pa\db u)^m}{\omega_{\epsilon}^m}  }= f \quad \mbox{for} \; f\in\widehat{x}^d_{\max}\CI_{\nQb,1}(\hC_{\epsilon}).
\label{mae.2}\end{equation}
As in \cite[\S~5]{CDR2016}, we will solve this equation by first solving the corresponding equation on each non-maximal boundary hypersurface of $\hC_{\epsilon}$.  This will allow us to proceed by induction on the depth of $\hC_{\epsilon}$.  In the induction step, we can in fact assume that $f\in \widehat{x}^d_{\max}w\CI_{\nQb}(C_{\epsilon})$, where 
$$
   w= \prod_{\hH\in \cM_{\nm}(\hX)} x_{\hH}.
$$
\begin{lemma}
Let $\beta$ be a real number such that $\frac{2}{(1-\nu)}<\beta< \frac{2m}{(1-\nu)}$ and $\beta\le 2m_{\hH_{\epsilon}}$ for each $\hH_{\epsilon}\in\cM_{\nm}(\hC_{\epsilon})$, where $2m_{\hH_{\epsilon}}$ is the real dimension of the fibers of $\phi_{\hH_{\epsilon}}: \hH_{\epsilon}\to S_{\hH_{\epsilon}}$.  Let $\delta$ be a multiweight such that $\delta_{\max}=\beta-\frac{2}{1-\nu}$, $-\frac{2\nu_{\hH}}{1-\nu_{\hH}}<\delta_{\hH}<1-\frac{2\nu_{\hH}}{1-\nu_{\hH}}$ and
$$
              \lrp{\frac{1-\nu}{1-\nu_{\hH}}}\beta-\frac{2}{1-\nu_{\hH}}+2-2m_{\hH_{\epsilon}}< \delta_{\hH} < \lrp{\frac{1-\nu}{1-\nu_{\hH}}}\beta - \frac{2}{1-\nu_{\hH}}
$$
 for $\hH\in\cM_{\nm}(\hX)$. If $f\in \widehat{x}^{\beta'}_{\max}w\CI_{\nQb}(C_{\epsilon})$ for some $\beta'>\beta$, then there exists $v\in x^{\delta}\CI_{\nQb}(C_{\epsilon})$ such that $\widetilde{\omega}_{\epsilon}= \omega_{\epsilon}+ \sqrt{-1}\pa\db v$ is the K\"ahler form of an $\mathfrak{n}$-warped $\QAC$-metric $\widetilde{g}_{\epsilon}$ asymptotic to $g_{C_0}$ with rate $\min\{d,\beta\}$ and with
\begin{equation}
  \widetilde{f}:= f-\log\lrp{ \frac{\widetilde{\omega}_{\epsilon}^m}{\omega_{\epsilon}^m}  }\in \CI_c(C_{\epsilon}).
\label{mae.3a}\end{equation}
\label{mae.3}\end{lemma}
\begin{proof}
We follow the approach of \cite[Proposition~25]{Szekelyhidi} and apply a fixed point argument outside a sufficiently large compact set of $C_{\epsilon}$.  For $A>0$, consider the closed subset 
$$
     \rho^{-1}([A,\infty))\subset C_{\epsilon}.
$$  
Denote by $x^{\delta}\cC^{k,\alpha}_{\nQb}(\rho^{-1}([A,\infty)))$ the subspace of functions on $\rho^{-1}([A,\infty))$ obtained by restriction of $x^{\delta}\cC^{k,\alpha}_{\nQb}(C_{\epsilon})$ to $\rho^{-1}([A,\infty))$.  The norm of a function in $x^{\delta}\cC^{k,\alpha}_{\nQb}(\rho^{-1}([A,\infty)))$ can be defined as the infimum over the corresponding norms of the possible extensions in $x^{\delta}\cC^{k,\alpha}_{\nQb}(C_{\epsilon})$.  For $\delta$ as in the statement of the lemma, consider the subset 
$$
   \cB:= \{ u\in x^{\delta}\cC^{k+2,\alpha}_{\nQb}(\rho^{-1}([A,\infty))) \; | \; \|u\|_{x^{\delta}\cC^{k+2,\delta}_{\nQb}}\le \epsilon_0\}
$$
for some $\epsilon_0>0$ chosen sufficiently small so that $\omega_{\epsilon}+ \sqrt{-1}\pa\db u$ is positive definite in $\rho^{-1}([A,\infty))$.  To solve the complex Monge-Amp\`ere-equation near infinity, consider the nonlinear operator 
$$
\begin{array}{lccc}
  F: & \cB & \to & (\rho w)^{-2}x^{\delta}\cC^{k,\alpha}_{\nQb}(\rho^{-1}([A,\infty))) \\
      & u & \mapsto & \log\lrp{ \frac{(\omega_{\epsilon}+\sqrt{-1}\pa\db u)^m}{\omega_{\epsilon}^m}  }-f.
\end{array}
$$
This can be rewritten as 
$$
     F(u)= F(0)+ \frac12 \Delta_{g_{\epsilon}}u+ Q(u),
$$
where $\Delta_{g_{\epsilon}}= g_{\epsilon}^{ij}\nabla_i\nabla_j$ is the Laplacian associated to $g_{\epsilon}$ and $Q(u)$ is the nonlinear part of $F(u)$.  Let $\mu\in\CI(\bbR)$ be a function such that $\mu(t)\equiv 0$ for $t<1$ and $\mu(t)\equiv 1$ for $t>2$.  To solve \eqref{mae.2}, it suffices then to solve the equation
\begin{equation}
u=2 \Delta_{g_{\epsilon}}^{-1} \left[ \mu(\rho-A)(-F(0)-Q(u))\right],
\label{contr.1}\end{equation}
where $\Delta_{g_{\epsilon}}^{-1}$ is the inverse provided by Corollary~\ref{w.53}.  Indeed, if $u\in x^{\delta}\cC^{k+2,\alpha}_{\nQb}(\rho^{-1}([A,\infty)))$ is a solution of \eqref{contr.1}, then it is a solution of \eqref{mae.2} on $\rho^{-1}([A+2,\infty))$.  Now, to solve \eqref{contr.1}, we can look for a fixed point of the operator 
\begin{equation}
    N(u):= 2\Delta_{g_{\epsilon}}^{-1}\left[ \mu(\rho-A) (-F(0)-Q(u))  \right].
\label{contr.2}\end{equation}
This can be achieved by showing that $N$ is a contraction.  First, as in the proof of \cite[Proposition~25]{Szekelyhidi}, we see from the explicit formula for $Q$ that for $u,v\in \cB$, 
$$
\begin{aligned}
\| Q(u)-Q(v)\|_{(\rho w)^{-2}x^{\delta}\cC^{k,\alpha}} &\le C \lrp{  \| \pa\db u\|_{\cC^{k,\alpha}_{\nQb}(\rho^{-1}([A,\infty);\Lambda^{1,1}({}^{w}T^*\hC_{\epsilon}))}+\| \pa\db v\|_{\cC^{k,\alpha}_{\nQb}(\rho^{-1}([A,\infty);\Lambda^{1,1}({}^{w}T^*\hC_{\epsilon}))}   } \cdot \\
  & \hspace{8cm} \| u-v\|_{x^{\delta}\cC^{k+2,\alpha}_{\nQb}}  \\
 & \le \widetilde{C} \lrp{ \|u\|_{(\rho w)^2\cC^{k+2,\alpha}_{\nQb}} +  \|u\|_{(\rho w)^2\cC^{k+2,\alpha}_{\nQb}}}\| u-v\|_{x^{\delta}\cC^{k+2,\alpha}_{\nQb}}
\end{aligned}
$$
for some positive constants $C$ and $\widetilde{C}$ depending on $g_{\epsilon}$, but independent of the choice of $\epsilon_0$ in the definition of $\cB$ if we assume without loss of generality that $\epsilon_0<1$.  Since
$\delta_{\max}>\frac{-2}{1-\nu}$ and $\delta_{\hH}>-\frac{2\nu_{\hH}}{1-\nu_{\hH}}$, there is a continuous inclusion 
$$
      x^{\delta}\cC^{k+2,\alpha}_{\nQb}(\rho^{-1}([A,\infty)) \subset  (\rho w)^2 \cC^{k+2,\alpha}_{\nQb}(\rho^{-1}([A,\infty)), 
$$
so changing $\widetilde{C}$ if necessary, 
\begin{equation}
     \| Q(u)-Q(v)\|_{(\rho w)^{-2}x^{\delta}\cC^{k,\alpha}}  \le \widetilde{C} \lrp{ \|u\|_{x^{\delta}\cC^{k+2,\alpha}_{\nQb}} +  \|u\|_{x^{\delta}\cC^{k+2,\alpha}_{\nQb}}}\| u-v\|_{x^{\delta}\cC^{k+2,\alpha}_{\nQb}} \quad \forall u,v\in\cB.
\label{contr.3}\end{equation}
Since the $\cC^{k,\alpha}_{\nQb}$-norms of $\mu(\rho-A)$ can be controlled independently of the choice of $A$, we see from \eqref{contr.2}, \eqref{contr.3} and Corollary~\ref{w.53} that by taking $\epsilon_0>0$ sufficiently small in the definition of $\cB$ but independent of the choice of $A$, we will have that
$$
      \| N(u)-N(v)\|_{x^{\delta}\cC^{k+2,\alpha}_{\nQb}} < \frac{1}2 \| u-v\|_{x^{\delta}\cC^{k+2,\alpha}_{\nQb}} \quad \forall u,v\in \cB.
$$  
To apply the Banach fixed point theorem, we also need to check that $N$ maps $\cB$ to itself.  First note that 
$$
          F(0)=-f \in \hx_{\max}^{\beta'}w\CI_{\nQb}(C_{\epsilon}).
$$
Since $\beta'>\beta$ and $\delta_{\hH}<1-\frac{2\nu_{\hH}}{1-\nu_{\hH}}$, this means that there are constants $C>0$ and $\tau>0$ such that 
$$
       \| F(0)\|_{(\rho w)^{-2}x^{\delta}\cC^{k,\alpha}_{\nQb}}< CA^{-\tau}.
$$
Hence, by Corollary~\ref{w.53},  for $u\in\cB$, there is a constant $K>0$ independent of $A$ such that 
$$
\begin{aligned}
\| N(u)\|_{x^{\delta}\cC^{k+2,\alpha}_{\nQb}} & \le \| N(0)\|_{x^{\delta}\cC^{k+2,\alpha}_{\nQb}} + \| N(u)- N(0)\|_{x^{\delta}\cC^{k+2,\alpha}_{\nQb}} \\
     &\le K \| F(0)\|_{(\rho w)^{-2}x^{\delta}\cC^{k,\alpha}_{\nQb}}+ \frac12  \| u\|_{x^{\delta}\cC^{k+2,\alpha}_{\nQb}} \\
     &\le KCA^{-\tau} + \frac{\epsilon_0}2.
\end{aligned}
$$
Thus, for $A>0$ sufficiently large, $N(\cB)\subset \cB$ and $N$ has a fixed point $u$ as required.  Initially, $u$ is only defined on $\rho^{-1}([A,\infty))$, but using Lemma~\ref{kw.14}, we can find a new function $v$ agreeing with $u$ outside a compact set such that 
$$
       \widetilde{\omega}_{\epsilon}= \omega_{\epsilon}+ \sqrt{-1}\pa\db v
$$
is positive definite everywhere.  Using Corollary~\ref{w.53} and \eqref{mae.2}, we can bootstrap to see that in fact $v\in x^{\delta}\CI_{\nQb}(C_{\epsilon})$.

\end{proof}

The decay in \eqref{mae.3a} is exactly what we need to solve the complex Monge-Amp\`ere equation via the approach of \cite{Tian-Yau1991} or its parabolic version \cite{Chau-Tam2011}.
\begin{theorem}
For the K\"ahler form $\widetilde{\omega}_{\epsilon}$ and the function $\widetilde{f}$ of Lemma~\ref{mae.3}, the complex Monge-Amp\`ere equation 
\begin{equation}
  \log \lrp{  \frac{(\widetilde{\omega}_{\epsilon}+ \sqrt{-1}\pa\db u)^m}{\widetilde{\omega}_{\epsilon}^m}   }= \widetilde{f}
\label{mae.4a}\end{equation}
has a unique solution $u$ in $x^{\widetilde{\delta}}\CI_{\nQb}(C_{\epsilon})$, where $\widetilde{\delta}$ is any multiweight satisfying the conditions of Corollary~\ref{w.53} such that $\widetilde{\delta}_{\max}>0$ and $\widetilde{\delta}_{\hH}>0$ for $\hH\in\cM_{\nm}(\hX)$.
\label{mae.4}\end{theorem}
\begin{proof}
The proof is quite similar to those of \cite[Theorem~5.4]{CDR2016} and \cite[Theorem~6.2]{CR2021}.  For the convenience of the reader, we will go over the argument putting emphasis on the new features.  The overall strategy is to apply the continuity method to the equation
\begin{equation}
    \log \lrp{  \frac{(\widetilde{\omega}_{\epsilon}+ \sqrt{-1}\pa\db u)^m}{\widetilde{\omega}_{\epsilon}^m}   }= t\widetilde{f}
\label{mae.5}\end{equation}
for $t\in [0,1]$.  This will be achieved by showing that the set 
$$
       S:= \{s\in [0,1]\; | \; \mbox{there is a solution} \; u_{\epsilon}\in x^{\widetilde{\delta}}\CI_{\nQb}(C_{\epsilon}) \; \mbox{of \eqref{mae.5} for} \; t=s\}
$$
is all of $[0,1]$, that is, by showing that $S$ is nonempty, open and closed.  Clearly, $u_0=0$ is a solution of \eqref{mae.5} with $t=0$, so $S$ is not empty.  On the other hand, the openness of $S$ follows from Corollary~\ref{w.53}.

To see that $S$ is also closed, suppose that $[0,\tau)\subset S$ for some $0<\tau \le 1$.  We need to show that \eqref{w.53} has a solution $u_\tau\in x^{\widetilde{\delta}}\CI_{\nQb}(C_{\epsilon})$ for $t=\tau$.  By Corollary~\ref{w.41}, the Sobolev inequality holds for $\mathfrak{n}$-warped $\QAC$-metrics on $C_{\epsilon}$, so we can apply  Moser iteration to derive an a priori $C^0$-bound on solutions of \eqref{mae.5}.  The argument of Yau then provides uniform bounds on $\sqrt{-1}\pa\db u_t$, which by the result of Evans-Krylov, yields a priori $C^{2,\gamma}_w(C_{\epsilon})$-bounds on solutions.  Taking an increasing sequence $t_{i}\nearrow \tau$, we can apply the Arzel\`a-Ascoli theorem to extract a subsequence of $\{u_{t_i}\}$ converging in $\cC^2_w(C_{\epsilon})$ to some solution $u_{\tau}$ of \eqref{mae.5} for $t=\tau$. Bootstrapping  shows that $u_{\tau} \in\CI_w(C_{\epsilon})$.

To show that $u_{\tau}\in x^{\delta}\CI_{\nQb}(C_{\epsilon})$, we can first apply a Moser iteration with weight as in \cite[\S~8.6.2 and \S~9.6.2]{Joyce} to obtain an a priori bound in $\rho^{-\mu_1}\cC^0_w(C_{\epsilon})$ for $\mu_1>0$ such that $\frac{\mu_1}{1-\nu_{\hH}}<\delta_{\hH}$ for each $\hH\in\cM_1(\hX)$.  As in \cite[(5.11)]{CDR2016}, equation \eqref{mae.5} for $t=\tau$ can be rewritten as 
\begin{equation}
  \Delta_u u=\tau\widetilde{f},
\label{mae.6}\end{equation}
where 
$$
    \Delta_u v= \frac12 \int_0^{1} (\Delta_{u,t}v)dt 
$$
with $\Delta_{u,t}$ the Laplacian associated to the K\"ahler form $\widetilde{\omega}_{\epsilon}+ t\sqrt{-1}\pa\db u$.  By Proposition~\ref{w.2}, we can apply the Schauder estimate to \eqref{mae.6} in terms of $\mathfrak{n}$-warped $\QAC$-metrics to bootstrap and obtain that in fact $u_{\tau}\in \rho^{-\mu_1}\CI_w(C_{\epsilon})$.  Now by Lemma~\ref{hol.1}, we know that
$$
       \rho^{-\mu_1}\cC^1_w(C_{\epsilon})\subset C^{0,\mu_1}_{\nQb}(C_{\epsilon}).
$$ 
In particular, this implies that $\| \pa\db u_{\tau}\|_{\widetilde{g}_{\epsilon}}\in \cC^{0,\mu_1}_{\nQb}(C_{\epsilon})$.  Rewriting \eqref{mae.6} in terms of an elliptic $\nQb$-operator
\begin{equation}
 (\rho w)^2\Delta_u u=(\rho w)^2\tau\widetilde{f}
\label{mae.7}\end{equation}
and using Lemma~\ref{wqb.16}, we can apply the Schauder estimate and bootstrap to see that $u_{\tau}\in \rho^{-\mu_1}\CI_{\nQb}(C_{\epsilon})$.  Finally, using the inclusion $x^{\widetilde{\delta}}\CI_{\nQb}(C_{\epsilon})\subset \rho^{-\mu_1}\CI_{\nQb}(C_{\epsilon})$ for $\mu_1>0$ small enough, we see from Corollary~\ref{w.53} applied with $\widetilde{\delta}$ above and with a positive multiweight $\delta'$ small enough such that $\rho^{-\mu_1}\CI_{\nQb}(C_{\epsilon})\subset x^{\delta'}\CI_{\nQb}(C_{\epsilon})$ that $u\in x^{\delta}\CI_{\nQb}(C_{\epsilon})$.  This shows that $S$ is closed and completes the proof of existence.  For uniqueness, we can proceed as in \cite[Proposition~7.13]{Aubin} but using the isomorphism \eqref{w.53a} instead of the maximum principle.

\end{proof}

This allows us to solve the complex Monge-Amp\`ere \eqref{mae.2} when $f\in \hx^\beta_{\max}w\CI_{\nQb}(C_{\epsilon})$.

\begin{corollary}
Suppose that $\beta$ is a real number such that 
$$
   \frac{2}{(1-\nu)}<\beta< \frac{2m}{(1-\nu)}     
$$
and $\beta\le 2m_{\hH_{\epsilon}}$ for each $\hH_{\epsilon}\in\cM_{\nm}(\hC_{\epsilon})$.  Let $\delta$ be the multiweight as in Lemma~\ref{mae.3}.  If $f\in \hx_{\max}^{\beta'}w\CI_{\nQb}(C_{\epsilon})$ for some $\beta'>\beta$, then the complex Monge-Amp\`ere equation \eqref{mae.2} has a unique solution $u\in x^{\delta}\CI_{\nQb}(C_{\epsilon})$.
\label{mae.8}\end{corollary}
\begin{proof}
The existence is given by Lemma~\ref{mae.3} and Theorem~\ref{mae.4}.  Uniqueness follows again from \cite[Proposition~7.13]{Aubin} using the isomorphism \eqref{w.53a} instead of the maximum principle.
\end{proof}

Using these results, we can finally solve the complex Monge-Amp\`ere equation \eqref{mae.2} in full generality.  
\begin{theorem}
Suppose  that $\beta:=\min\{d,2m_1,\ldots,2m_N\}$ is such that 
$$
         \beta> \frac{2}{1-\nu}.
$$
Then for $f\in x^{d}_{\max}\CI_{\nQb,1}(C_{\epsilon})$, the complex Monge-Amp\`ere equation \eqref{mae.2} has a unique solution 
$$
u\in \hx_{\max}^{\beta-2}\sigma^{2\nu}\CI_{\nQb,1}(C_{\epsilon}),
$$
where $\sigma:= \prod_{\hH\in\cM_{1,\nu}(\hX)} x_{\hH}^{-\frac{1}{1-\nu}}$ if $\nu>0$ and $\sigma:=1$ otherwise, so that $\sigma^{\nu}=\rho^{\mathfrak{n}}$.
\label{mae.9}\end{theorem}
\begin{proof}
To construct $u$, we will first construct what should be the restriction of $\sigma^{-2\nu}u$ to $\hH_{\epsilon}$ for $\hH_{\epsilon}\in \cM_{\nm}(\hC_{\epsilon})$  such that 
$$
     \hH<\hG\; \Longrightarrow \; \hG=\hH_{\max}.
$$
By property (2) of Definition~\ref{kw.23}, using the rescaled complex coordinates \eqref{cyc.14} in the fibers of $\phi_{\hH_{\epsilon}}: \hH_{\epsilon}\to S_{\hH_{\epsilon}}$, this restriction that we will denote by $u_{\hH_{\epsilon}}$ should satisfy fiberwise the complex Monge-Amp\`ere equation
\begin{equation}
 \log \lrp{ \frac{(\omega_{\hH}+\sqrt{-1}\pa\db u_{\hH})^{m_{\hH_{\epsilon}}}}{\omega_{\hH_{\epsilon}}^{m_{\hH_{\epsilon}}}}}= f|_{\hH_{\epsilon}}.
\label{mae.10}\end{equation}
This is a family of complex Monge-Amp\`ere equations in the asymptotically conical setting.  Applying the asymptotically conical analog of Corollary~\ref{mae.8}, see for instance \cite[Theorem~2.1]{CH2013}, we can solve it in each fiber to find a unique solution $u_{\hH_{\epsilon}}\in\hx^{\beta-2}_{\max}\CI_{\nQb}(\hH_{\epsilon}/S_{\hH_{\epsilon}})$.  Extend $u_{\hH_{\epsilon}}$ smoothly off of $\hH_{\epsilon}$ to a function $\widehat{u}_{\hH_{\epsilon}}$  and consider the new closed $(1,1)$-form 
\begin{equation}
   \omega_{\epsilon,\hH}= \omega_{\epsilon}+ \sqrt{-1}\pa\db (\sigma^{2\nu}\widehat{u}_{\hH_{\epsilon}}).
\label{mae.11}\end{equation}
By property (2) of Definition~\ref{kw.3} and \eqref{mae.10}, it is positive definite outside a compact set.  By Lemma~\ref{kw.14}, changing $\widehat{u}_{\hH_{\epsilon}}$ on a compact set of $C_{\epsilon}$, we can assume that $\omega_{\epsilon,\hH}$ is positive definite everywhere.  By \eqref{mae.10} and computing as in \eqref{kw.21}, using the fact that
$$
    \sigma^{2\nu}\hx^{\beta-2+\frac{2}{1-\nu}}_{\max}= \mathcal{O}(x_{\max}^{\beta}. (\sigma^{2\nu}\hx_{\max}^{\frac{2\nu}{1-\nu}})),
$$
we see that 
\begin{equation}
    f_1:= f-\log\lrp{\frac{\omega^{m}_{\epsilon,\hH}}{\omega_{\epsilon}^{m}}}\in x_{\max}^{\beta}x_{\hH}\CI_{\nQb,1}(C_{\epsilon}).
\label{mae.12}\end{equation}
Thus, replacing $\omega_{\epsilon}$ by $\omega_{\epsilon,\hH}$ and $f$ by $f_1$, the complex Monge-Amp\`ere equation \eqref{mae.2} corresponds to solving 
\begin{equation}
   \log\lrp{ \frac{(\omega_{\epsilon,\hH}+\sqrt{-1}\pa\db u)^m}{\omega_{\epsilon,\hH}^m} }= f_1\in \hx^\beta_{\max}x_{\hH} \CI_{\nQb,1}(C_{\epsilon}).
\label{mae.13}\end{equation}
Let $\cK$ be the subset of $\cM_{\nm}(\hX)$ consisting of boundary hypersurfaces $\hH$ such that
$$
    \hH< \hG \; \Longrightarrow \; \hG=\hH_{\max}.
$$ 
Performing this argument at each $\hH\in \cK$, we can reduce to the case where 
\begin{equation}
    f\in \hx^{\beta}_{\max}\lrp{\prod_{\hH\in \cK}x_{\hH}}\CI_{\nQb,1}(C_{\epsilon}).
\label{mae.14}\end{equation}
Knowing that \eqref{mae.14} holds, this argument can be iterated.  Namely, if $\hH\in \cM_{\nm}(\hX)\setminus \cK$ is such that 
$$
       \hH<\hG \; \Longrightarrow \; \hG \in \cK\cup \{\hH_{\max}\},
$$
 then we can find the restriction of $\sigma^{-2\nu}u$ by solving \eqref{mae.10} again, this time however with Corollary~\ref{mae.8} (with $\nu=0$ if $\hH\in\cM_{1,\nu(\hX)}$) instead of \cite[Theorem~2.1]{CH2013}.  Proceeding in an order which is non-increasing with respect to the partial order on $\cM_{\nm}(\hX)$, this argument can be iterated to reduce to the case where $f\in \hx^{\beta}_{\max}w\CI_{\nQb}(C_{\epsilon})$.  One can then apply Corollary~\ref{mae.8} once more to get the desired solution on $C_{\epsilon}$.
 
 To show the solution is unique, we again apply \cite[Proposition~7.13]{Aubin} with the maximum principle replaced by the isomorphism \eqref{w.53a} with multiweight $\delta$ such that $\delta_{\hH}= -\frac{2\nu_H}{1-\nu_H}$ for $\hH\in\cM_{\nm}(\hX)$ and 
 $$
     \delta_{\max}= \beta -\frac{2}{1-\nu}- \lambda
 $$
 with $\lambda\ge 0$ possibly positive but small.

\end{proof}

Applying this result to the complex Monge-Amp\`ere equation \eqref{mae.1} yields the main result of this paper.
\begin{corollary}
Suppose that $d>1$ and that $\beta:=\min\{d,2m_1,\ldots,2m_N\}$ is such that 
$$
     \beta> \frac{2}{1-\nu}.
$$
Then for $\epsilon\ne 0$ as in Assumption~\ref{cyc.6c}, $C_{\epsilon}$ admits a Calabi-Yau $\mathfrak{n}$-warped $\QAC$-metric asymptotic to $g_{C_0}$ with rate $\beta$.  
\label{mae.15}\end{corollary}
\begin{proof}
Since $d>1$, we know by Corollary~\ref{kw.24} that there exists on $C_{\epsilon}$ a K\"ahler $\mathfrak{n}$-warped $\QAC$-metric asymptotic to $g_{C_0}$ with rate $d$.  By Lemma~\ref{kw.25}, its Ricci potential $\mathfrak{r}_{\epsilon}$ is in $\hx_{\max}^d\CI_{\nQb,1}(C_{\epsilon})$, so by Theorem~\ref{mae.9}, the complex Monge-Amp\`ere equation \eqref{mae.1} has a unique solution $u\in \hx_{\max}^{\beta-2}\sigma^{2\nu}\CI_{\nQb,1}(C_{\epsilon})$ and 
$$
     \widetilde{\omega}_{\epsilon}= \omega_{\epsilon}+ \sqrt{-1}\pa\db u
$$
is the desired Calabi-Yau metric.

\end{proof}

When $N=1$, it is possible to improve slightly the result as follows.
\begin{corollary}
Suppose that $N=1$, that $d>1$ and that $\beta:=\min\{d,2m_1\}$ is such that either $ \beta> \frac{2}{1-\nu}$, or else
\begin{equation}
      3<\beta\le \frac{2}{1-\nu}< 2m_1+5.
\label{mae.16a}\end{equation}
Then for $\epsilon\ne 0$ as in Assumption~\ref{cyc.6c}, $C_{\epsilon}$ admits a Calabi-Yau $\mathfrak{n}$-warped $\QAC$-metric asymptotic to $g_{C_0}$ with rate $\beta$.  
\label{mae.16}\end{corollary}
\begin{proof}
If $\beta>\frac{2}{1-\nu}$, this is a particular case of Corollary~\ref{mae.15}.  If instead \eqref{mae.16a} holds, we can follow the strategy of \cite[\S~7]{CR2021}.  More precisely, as in the first part of the proof of Theorem~\ref{mae.9}, we can first solve the complex Monge-Amp\`ere equation on the fibers of $\hH_{\epsilon}\in \cM_{\nm}(C_{\epsilon})$ to reduce to the case where $\mathfrak{r}_{\epsilon}\in \hx^{\beta}_{\max}w\CI_{\nQb}(C_{\epsilon})$.
By \cite[Corollary~5.4]{CMR2015}, we know in fact that $\mathfrak{r}_{\epsilon}\in \hx_{\max}^{\beta} w\cA_{\phg}(\hC_{\epsilon})$.  Proceeding as in \cite[Lemma~7.1]{CR2021}, we can then eliminate the part of the polyhomogeneous expansion of $\mathfrak{r}_{\epsilon}$ at $\hH_{\max}$ of order $\frac{2}{1-\nu}$ or less.  Indeed, if $(\hx_{\max}^{\mu}w)e$ is a term of order $\mu\le \frac{2}{1-\nu}$ in this expansion at $\hH_{\max}$, then since 
$$
     \hx_{\max}^{\mu}w= v^{\mu}w^{1-\mu}= \rho^{-\mu(1-\nu)}w^{1-\mu},
$$ 
we need to solve $I(B,\lambda)f_{\lambda}=w^{1-\mu}e$ as in \cite[(7.6)]{CR2021} with $\lambda=\mu(1-\nu)-2$, which can be achieved through \cite[Corollary~A.5]{CR2021} with 
$$
    a-2=1-\mu.
$$
But in our case, $f=2m_1-1$ in the statement of \cite[Corollary~A.5]{CR2021}, so for this corollary to apply, we need $-f+1<a<0$, which in terms of $\mu$ and $m_1$ translates into
\begin{equation}
    3<\mu<2m_1+5.
\label{mae.17}\end{equation}
By assumption, $\beta\le \mu\le \frac{2}{1-\nu}$, so \eqref{mae.17} holds thanks to \eqref{mae.16a}.  We can therefore proceed as in \cite[Lemma~7.1]{CR2021} to reduce to the case 
$$
         \mathfrak{r}_{\epsilon}\in \hx_{\max}^{\mu}w\cA_{\phg}(C_{\epsilon})
$$
with $\mu>\frac{2}{1-\nu}$.  We can then rely on Corollary~\ref{mae.8} to conclude the proof.
\end{proof}

\section{Singular Calabi-Yau warped $\QAC$-metrics} \label{sing.0}

The Calabi-Yau metric of Example~\ref{int.2} with $\ell=2$ yields a Calabi-Yau metric as predicted in \cite{YangLi2018}, while Example~\ref{int.2b} provides a generalization in higher dimensions.  Following \cite{YangLi2018}, we can try to extract a singular Calabi-Yau metric out of these examples.  To do so, we will slightly change the setting of these examples by considering instead the smoothing $C_{\epsilon}$ given by 
\begin{equation}
  z_{1,1}^k+\cdots+ z_{1,m_1+1}^k=\epsilon
\label{sing.1}\end{equation}  
of the cone $C_0$ defined by 
\begin{equation}
  z_{1,1}^k+\cdots+ z_{1,m_1+1}^k=0,
\label{sing.2}\end{equation}
where $k$ and $m_1$ are chosen as in Example~\ref{int.2} or Example~\ref{int.2b}.  To lighten the notation, we will set $m:= m_1+m_0=m_1+1$, $z_0:=z_{0,1}$ and $z_i:=z_{1,i}$ for $i\in\{1,\ldots,m_1+1\}$.  Then $C_{\epsilon}$ is also a smoothing of the affine hypersurface given by
\begin{equation}
      \cH_0= \{ (z_0,z_1,\ldots, z_m)\in \bbC\times \bbC^m\; | \; z_0^k+\cdots +z_m^{k}=0 \}.
\label{sing.3}\end{equation}
To emphasize this relationship, it will be convenient to use the notation
\begin{equation}
      \cH_\epsilon= \{ (z_0,z_1,\ldots, z_m)\in \bbC\times \bbC^m\; | \; z_0^k+\cdots+z_m^{k}=\epsilon\}
\label{sing.3a}\end{equation}
for $\epsilon\ne 0$.  Hence $\cH_{\epsilon}=C_\epsilon$ for $\epsilon\ne 0$, but $\cH_0$ and $C_0$ corresponds to distinct cones.  In this case, the compactification $\hX$ of $\bbC^{m+1}$ described in \S~\ref{sCY.0} has two boundary hypersurfaces $\hH_1$ and $\hH_2$ with $\hH_2$ maximal.  Moreover, in the setting of Corollary~\ref{w.53}, $\nu=\frac{m-k}{m-1}$ with $\nu_{H_1}=\nu_{H_2}=\nu$, so the weight function $\mathfrak{n}$ is given by $\mathfrak{n}(\hH_1)=\mathfrak{n}(\hH_2)=\nu= \frac{m-k}{m-1}$.  We will let $x_1$ and $x_2$ be choices of boundary defining functions for $\hH_1$ and $\hH_2$ with $\rho^{-1}= (x_1x_2)^{\frac{1}{1-\nu}}$ an $\mathfrak{n}$-weighted total boundary defining function compatible with the class specified by \eqref{kw.17b}.   Without loss of generality, we can assume that $x_1$ and $x_2$ are constant near the origin. There is a corresponding compactification $\widehat{\cH}_{\epsilon}$ for $\cH_{\epsilon}$ by taking its closure in $\hX$.

Since we assume that $2\le k\le m$, it is well-known that $\cH_0$ admits a Calabi-Yau cone metric $g_{\cH_0}$, see for instance \cite[Example~2.1]{CR2021}.  On $\cH_{\epsilon}$, let 
\begin{equation}
   \omega_{\epsilon}= \frac{\sqrt{-1}}2 \pa \db \phi_{\epsilon} |_{\cH_\epsilon}
\label{sing.4}\end{equation}
be the $\mathfrak{n}$-warped $\QAC$-metric constructed in Theorem~\ref{kw.19} and Corollary~\ref{kw.24}.  Notice that the construction still works on $\cH_0$, more precisely, the convexity argument of Lemma~\ref{kw.14} can be achieved in a uniform way in $\epsilon\in [0,1]$ with $\phi_{\epsilon}$ corresponding to the restriction of the potential of the Euclidean metric on $\bbC\times \bbC^m$ within a fix compact set $K\subset \bbC\times \bbC^m$ containing an open neighborhood of the origin.
With more care, we can assume that $\omega_{\epsilon}$ is already Calabi-Yau outside a large compact set.  Indeed, for $\epsilon\in [0,1]$, let $\cH_{\epsilon}'$ be a smooth family of smooth non-compact manifolds with compact sets $K_{\epsilon}'\subset \cH_{\epsilon}'$   and $K_{\epsilon}\subset \cH_{\epsilon}$
such that there are canonical identifications 
$$
   \cH_{\epsilon}'\setminus K_{\epsilon}' = \cH_{\epsilon}\setminus K_{\epsilon}
$$
as smooth families of manifolds.  This is possible by changing $\cH_{\epsilon}$ in a neighborhood of the origin.  Pick a smooth family of $\mathfrak{n}$-warped $\QAC$-metrics $g_{\epsilon}'$ on $\cH_{\epsilon}'$ such that
$$
      g_{\epsilon}'|_{\cH_{\epsilon}'\setminus K_{\epsilon}'}= g_{\epsilon}|_{\cH_{\epsilon}\setminus K_{\epsilon}},
$$
 where $g_{\epsilon}$ is our initial family of K\"ahler $\mathfrak{n}$-warped $\QAC$-metrics with K\"ahler form \eqref{sing.4}.  Then Corollary~\ref{w.53} applies to the family of Laplacians $\Delta_{g_{\epsilon}'}$ associated  to $g_{\epsilon}'$.  We can then use the surjectivity of this family of Laplacians as in the proof of Theorem~\ref{mae.9} to improve the decay of the Ricci potential of $g_{\epsilon }$ (working on the complement of $K_{\epsilon}$) to be in $\widehat{x}_{\max}w\CI_{\nQb}(\cH_{\epsilon})$.  Using the uniform invertibility of $\Delta_{g_{\epsilon}'}$ for $\epsilon\in [0,1]$, we can then apply the proof of Lemma~\ref{mae.3} to ensure that for $\epsilon\in [0,1]$, $g_{\epsilon}$ is in fact a Calabi-Yau $\mathfrak{n}$-warped $\QAC$-metric outside some large compact set.  On the other hand, still using the convexity argument of Lemma~\ref{kw.14}, we can assume that $g_{\epsilon}$ is still induced by the restriction of the Euclidean metric near the origin.    In other words, without loss of generality, we can assume that the function
 \begin{equation}
    f_{\epsilon}:=  -\log\lrp{ \frac{\omega_{\epsilon}^m}{c_m\Omega_{\cH_{\epsilon}}\wedge \overline{\Omega}_{\cH_{\epsilon}}} }
 \label{sing.4a}\end{equation}
 has compact support.

In particular, notice that near the origin, $\omega_{\epsilon}$ (with $\epsilon$ possibly equal to zero) has negative bisectional curvature by \cite[\S~5, p. 819]{Vitter} since it is the restriction of the Euclidean metric there.  Since outside a small open set containing the origin, the bisectional curvature of $\omega_{\epsilon}$ can be uniformly bounded above and below as $\epsilon\searrow 0$, we have the following.

\begin{lemma}
The bisectional curvature of $\omega_{\epsilon}$ has a uniform upper bound independent of $\epsilon\in[0,1]$.  
\label{sing.5}\end{lemma}
This upper bound on the bisectional curvature will be crucial in solving the complex Monge-Amp\`ere equation on $\cH_0$.  On the other hand, this is a manifestation of the fact that $\omega_0$ is quite different from the K\"ahler form $\omega_{\cH_0}$ of the Calabi-Yau cone metric $g_{\cH_0}$ near the origin.  To compare the two metrics, let 
\begin{equation}
   P_k(z)= z_0^k+\cdots+ z_m^k
\label{sing.6}\end{equation}  
be the polynomial such that 
$$
          \cH_\epsilon= P_k^{-1}(\epsilon).
$$
Then the holomorphic volume form $\Omega_{\cH_0}$ on $\cH_0$ is defined implicitly by 
\begin{equation}
dz_0\wedge\ldots \wedge dz_m= \Omega_{\cH_0}\wedge dP_k.
\label{sing.7}\end{equation}
The fact that $(\cH_0,g_{\cH_0})$ is Calabi-Yau means that there is a constant $c_{m}\in\bbC\setminus\{0\}$ depending only on $m$ such that 
\begin{equation}
  \omega_{\cH_0}^m= c_m \Omega_{\cH_0}\wedge \overline{\Omega}_{\cH_0}.
\label{sing.8}\end{equation}
On the other hand, the natural $\bbR^+$-action on the cone $(\cH_0,g_{\cH_0})$ is obtained by restricting an $\bbR^+$-action of the form
\begin{equation}
      \begin{array}{llcl} \bbR^+\ni t: & \bbC^{m+1} & \to & \bbC^{m+1} \\ z & \mapsto & t^w z\end{array}
\label{sing.9}\end{equation} 
for some positive weight $w$.  By \eqref{sing.7}, $\Omega_{\cH_0}$ is then homogenous of degree $(m+1-k)w$ with respect to this $\bbR^+$-action, while $\omega_{\cH_0}$ is by definition of degree $2$, so that from \eqref{sing.8}, we see that
\begin{equation}
    w= \frac{m}{m+1-k}>1.
\label{sing.10}\end{equation}
This means that the Euclidean metric is homogeneous of degree $2w>2$.  If $r_E$ is the distance from the origin in terms of the Euclidean metric and $r_{\cH_0}$ is the distance from the origin with respect to $g_{\cH_0}$, then $r_{E}$ is homogenous of degree $w>1$ and $r_{\cH_0}$ is homogenous of degree $1$, so that
\begin{equation}
      r_{\cH_0} \asymp r_{E}^{\frac{1}{w}}=r_E^{1-\frac{k-1}{m}}
\label{sing.10b}\end{equation}
and near the origin,
$$
       \omega_{\cH_0}\asymp r_E^{-\frac{2(k-1)}{m}}\omega_0 \quad \mbox{and} \quad \Omega_{\cH_0}\wedge \overline{\Omega}_{\cH_0}\asymp r_E^{-2(k-1)}\omega_0^m.
$$
In fact, in the same way, near the origin, we have that
\begin{equation}
  \Omega_{\cH_{\epsilon}}\wedge\overline{\Omega}_{\cH_\epsilon} \asymp r^{-2(k-1)}_E \omega_{\epsilon}^m \quad \mbox{as} \; \epsilon\searrow 0.
\label{sing.11}\end{equation}
Now, our goal is to obtain a singular Calabi-Yau $\mathfrak{n}$-warped $\QAC$-metric on $\cH_0$ by finding a solution to the complex Monge-Amp\`ere equation 
\begin{equation}
   (\omega_0+\sqrt{-1}\pa\db u)^m= c_m \Omega_{\cH_0}\wedge \overline{\Omega}_{\cH_0}.
\label{sing.12}\end{equation}
Compared to \S~\ref{mae.0}, the main difference is that $\cH_0$ has a singularity at the origin.  As pointed out in \cite{Hein-Sun}, the fact that the expected model singularity $(\cH_0,g_{\cH_0})$ does not have bounded bisectional curvature seriously compromises the possibility of having a fine understanding of a solution of \eqref{sing.12} near the singularity working directly on $\cH_0$  with adapted weighted H\"older spaces.  Following \cite{Hein-Sun}, we will take advantage of the smoothing $\cH_{\epsilon}$ of $\cH_0$ and produce a solution of \eqref{sing.12} by extracting a converging subsequence of solutions $u_{\epsilon}$ as $\epsilon\searrow 0$ of the corresponding complex Monge-Amp\`ere equation
\begin{equation}
  (\omega_{\epsilon}+\sqrt{-1}\pa\db u_{\epsilon})^m= c_m \Omega_{\cH_{\epsilon}}\wedge \overline{\Omega}_{\cH_{\epsilon}}.
\label{sing.13}\end{equation}  
This will be possible provided we can get good uniform control on these solutions as $\epsilon\searrow 0$.  This will be achieved adapting recent ideas of \cite{CGT} and \cite{Sun-Zhang} to our setting.  The solution of \eqref{sing.12} we will obtain will have the expected $\mathfrak{n}$-warped $\QAC$-behavior at infinity, but will initially come with very little control near the origin.  To identify the precise behavior of the metric near the singularity, we will apply the continuity method of \cite{Hein-Sun}, which requires solving another family of complex Monge-Amp\`ere equations.  Let us describe this family before starting to construct solutions.  

On $\cH_0$, we could also have applied the convexity argument of Lemma~\ref{kw.14} with $\omega_{\cH_0}$ instead of the Euclidean metric to obtain a K\"ahler metric $\widetilde{\omega}_{0}$ which agrees with $\omega_{\cH_0}$ near the origin and with $\omega_0$ outside some compact set.  Since $\omega_{\cH_0}$ is Calabi-Yau and $\omega_0$ is Calabi-Yau outside a compact set, this means that 
\begin{equation}
              \widetilde{\omega}_0^m= e^Fc_m\Omega_{\cH_0}\wedge \overline{\Omega}_{\cH_0}
\label{sing.13a}\end{equation}
with $F\in\CI_c(\cH_{0}\setminus\{0\})$.  Extending $F$ from $\cH_0\setminus\{0\}$ to a smooth function in $\CI_c(\bbC^{m+1}\setminus\{0\})$ that we will also denote by $F$, the family of complex Monge-Amp\`ere equations we will consider is 
\begin{equation}
  (\omega_{\epsilon}+\sqrt{-1}\pa\db u_{t,\epsilon})^m= e^{tF}c_m \Omega_{\cH_{\epsilon}}\wedge \overline{\Omega}_{\cH_\epsilon}
\label{sing.14}\end{equation}
for $t,\epsilon\in [0,1]$.  For $t=\epsilon=0$, this corresponds to the complex Monge-Amp\`ere equation \eqref{sing.8} we want to solve, while for $t=1$ and $\epsilon=0$, we see from \eqref{sing.13a} that a solution is given by taking a compactly supported potential $u_{1,0}$ such that
$$
    \omega_0+ \sqrt{-1}\pa\db u_{1,0}=\widetilde{\omega}_0
$$
and has the expected singular behavior modelled on $(\cH_0,g_{\cH_0})$ near the singularity.  On the other hand, for $\epsilon>0$, we know that \eqref{sing.14} has a unique solution by Theorem~\ref{mae.4}.  Our first step will be to produce solutions to \eqref{sing.14} when $\epsilon=0$ by taking limits of solutions as $\epsilon\searrow 0$.  First, we can rewrite \eqref{sing.14} as
\begin{equation}
    (\omega_{\epsilon}+\sqrt{-1}\pa\db u_{t,\epsilon})^m= e^{tF}e^{f_{\epsilon}} \omega_{\epsilon},
\label{sing.15}\end{equation}
where $f_{\epsilon}$ is the function of \eqref{sing.4a}.  By Theorem~\ref{mae.4}, the complex Monge-Amp\`ere equation \eqref{sing.15} has a unique solution in $x^{\widetilde{\delta}}\CI_{\nQb}(\cH_{\epsilon})$ for $\epsilon \ne 0$ with $\widetilde{\delta}=(\widetilde{\delta}_1,\widetilde{\delta}_2)$ a multiweight satisfying the conditions of Corollary~\ref{w.53} and such that $\widetilde{\delta}_{1}>0$ and $\widetilde{\delta}_{2}>0$.  We need some uniform control on the norm of these solutions to extract a convergent subsequence as $\epsilon\searrow 0$.  By \eqref{sing.11} and the fact that $f_{\epsilon}$ and $F$ have support contained in a fixed large ball in $\bbC^{m+1}$, notice that there is a positive constant $A$ such that
\begin{equation}
  \sup_{\cH_{\epsilon}} | e^{-tF}e^{-f_{\epsilon}}-1 |\le A \quad \forall \ \epsilon\in [0,1], \; t\in [0,1],
\label{sing.16}\end{equation}
and 
\begin{equation}
   \| e^{-tF}e^{-f_{\epsilon}}-1 \|_{L^q(\mu_{\epsilon})} \le A \quad \forall \ \epsilon\in [0,1], \; t\in [0,1], \; q>q_0,
\label{sing.17}\end{equation}
for some $q_0>0$ depending on $\widetilde{\delta}$, where $\mu_{\epsilon}= c_m\Omega_{\cH_{\epsilon}}\wedge \overline{\Omega}_{\cH_{\epsilon}}$.  On the other hand, following an idea of Tosatti \cite{Tosatti2009}, noticing that $\widehat{\omega}_{0,\epsilon}:= \omega_{\epsilon}+ \sqrt{-1}\pa \db u_{0,\epsilon}$ is Calabi-Yau with tangent cone at infinity $(C_0, g_{C_0})$ of \eqref{pr.1}, we can show as in the proof of \cite[Proposition~4.1]{CGT} that there is a uniform Sobolev inequality for those metrics, namely there is a constant $C$ independent of $\epsilon\in (0,1]$ such that
\begin{equation}
  \lrp{ \int_{\cH_{\epsilon}} |u|^{\frac{2m}{m-1}}c_m \Omega_{\cH_{\epsilon}}\wedge \overline{\Omega}_{\cH_{\epsilon}} }^{\frac{m-1}m} \le C \int_{\cH_{\epsilon}} |du|^2_{\widehat{\omega}_{0,\epsilon}} c_m \Omega_{\cH_{\epsilon}}\wedge \overline{\Omega}_{\cH_{\epsilon}}
\label{sing.18}\end{equation}
for all $u\in \CI_c(\cH_{\epsilon})$ and $\epsilon\in (0,1]$.  The estimates \eqref{sing.17} and \eqref{sing.18} yield the following uniform control on the solutions.
\begin{lemma}
  There is a constant $K$ such that for all $\epsilon\in (0,1]$ and $t\in [0,1]$, the unique solution $u_{t,\epsilon}$ of \eqref{sing.15} is such that
 \begin{equation}
   \| u_{t,\epsilon}\|_{L^{\infty}}\le K
 \label{sing.19a}\end{equation}
 and
 \begin{equation}
    K^{-1}\omega_{\epsilon} \le \widehat{\omega}_{t,\epsilon} \le K e^{tF}e^{f_{\epsilon}}\omega_{\epsilon},
 \label{sing.19b}\end{equation}
 where $\widehat{\omega}_{t,\epsilon}= \omega_{\epsilon}+ \sqrt{-1}\pa\db u_{t,\epsilon}$.
\label{sing.19}\end{lemma}
\begin{proof}
The $L^{\infty}$-bound \eqref{sing.19a} is completely standard, see for instance \cite{Joyce}.  To obtain a uniform $L^{\infty}$-bound, it suffices to notice that the uniform bounds \eqref{sing.17} and \eqref{sing.18} allow us to obtain a uniform bound \eqref{sing.19a} independent of $\epsilon\in(0,1]$ and $t\in [0,1]$ proceeding as in  \cite[\S~4]{CGT}.  Let us explain the main differences in our setting.  First, at infinity, our metrics are not asymptotically conical, but $\mathfrak{n}$-warped quasi-asymptotically conical, and by Theorem~\ref{mae.4}, there exists $\gamma>0$ such that $u_{t,\epsilon}\in \rho^{-\gamma}\CI_{\nQb}(\cH_{\epsilon})$.  This can be used as in the proof of \cite[Proposition~4.2]{CGT} to justify the integration by parts leading to the estimate.  The other difference is that $\omega_{t,\epsilon}$ is not Calabi-Yau for $t>0$, but we can first proceed as in \cite[\S~4]{CGT} when $t=0$.   Then, writing 
$$
     \widehat{\omega}_{t,\epsilon}= \widehat{\omega}_{0,\epsilon}+ \sqrt{-1}\pa\db (u_{t,\epsilon}-u_{0,\epsilon}),
$$   
we can obtain the uniform $L^{\infty}$-bound for $u_{t,\epsilon}-u_{0,\epsilon}$ proceeding as in \cite[\S~4]{CGT}.  

For \eqref{sing.19b},  notice that the inequality on the right-hand side follows from the one on the left-hand side and the elementary inequality
$$
      \tr_{\beta}(\alpha)\le \frac{\alpha^m}{\beta^m} (\tr_{\alpha}(\beta))^{m-1}
$$ 
for any positive $(1,1)$-forms $\alpha$ and $\beta$.  On the other hand, the inequality on the right is a consequence of the Chern-Lu inequality.  Indeed, since $F$ is supported away from the origin, we see that the Ricci curvature of $\widehat{\omega}_{t,\epsilon}$ is uniformly bounded in $\epsilon\in (0,1]$ and $t\in [0,1]$, while as mentioned earlier, the bisectional curvature of $\omega_{\epsilon}$ is uniformly bounded from above.  Hence, by \cite[Proposition~7.1]{Jeffres-Mazzeo-Rubinstein}, there is a Chern-Lu inequality of the form 
\begin{equation}
  \Delta_{\widehat{\omega}_{t,\epsilon}}(\log \tr_{\widehat{\omega}_{t,\epsilon}}\omega_{\epsilon}- B u_{t,\epsilon})\ge \tr_{\widehat{\omega}_{t,\epsilon}}\omega_{\epsilon} -Bm
\label{sing.20}\end{equation}
for some constant $B$ independent of $\epsilon$ and $t$.  Since $u_{t,\epsilon}\to 0$ at infinity, notice that  $\log \tr_{\widehat{\omega}_{t,\epsilon}}\omega_{\epsilon}- B u_{t,\epsilon}$ tends to $\log m$ at infinity.  Thus,  if its maximum is not attained on $\cH_{\epsilon}$, we see that 
$$
      \log \tr_{\widehat{\omega}_{t,\epsilon}}\omega_{\epsilon} \le B u_{t,\epsilon} + \log m \le |B| K + \log m
$$
by \eqref{sing.19a}.  If instead its maximum is attained at $p_{t,\epsilon}\in \cH_{\epsilon}$, then by \eqref{sing.20}, at this point
$$
       \lrp{ \tr_{\widehat{\omega}_{t,\epsilon}}\omega_{\epsilon} }(p_{t,\epsilon}) \le Bm, 
$$
hence
$$
\begin{aligned}
\log  \tr_{\widehat{\omega}_{t,\epsilon}}\omega_{\epsilon} &= \log  \tr_{\widehat{\omega}_{t,\epsilon}}\omega_{\epsilon}- Bu_{t,\epsilon} + B u_{t,\epsilon} \\
  & \le \lrp{ \log  \tr_{\widehat{\omega}_{t,\epsilon}}\omega_{\epsilon}- Bu_{t,\epsilon}  }(p_{t,\epsilon}) + Bu_{t,\epsilon} \\
  &\le \log(Bm)-Bu_{t,\epsilon}(p_{t,\epsilon}) +Bu_{t,\epsilon} \le \log(Bm) + 2|B| K
\end{aligned}
$$
by \eqref{sing.19a}.  In both cases, this yield the desired uniform estimate.
\end{proof}

By the definition of $f_{\epsilon}$ in \eqref{sing.4a} and  the estimate \eqref{sing.11}, notice that the upper bound in \eqref{sing.19b} degenerates as we approach the origin.  Similarly, the injectivity radius and the curvature of the family of metrics $\omega_{\epsilon}$ degenerate as we approach the origin.  However, if we fix a compact set $K\subset \bbC^{m+1}$ containing an open neighborhood of the origin, then the upper bound in \eqref{sing.19b} is uniform in $t\in [0,1]$ and $\epsilon\in (0,1]$ on $\cH_{\epsilon}\setminus (K\cap \cH_{\epsilon})$.  Similarly, there is a uniform lower bound on the injectivity radius of $\omega_{\epsilon}$ on $\cH_{\epsilon}\setminus (K\cap \cH_{\epsilon})$, as well as uniform bounds on its curvature and its covariant derivatives.  Thus, Lemma~\ref{sing.19} and the result of Evans-Krylov, combined with Schauder estimates and bootstrapping yields  $\cC^k_{w}(\cH_{\epsilon})$-bounds on $u_{t,\epsilon}|_{\cH_{\epsilon}\setminus (\cH_{\epsilon}\cap K)}$ uniform in $t\in [0,1]$ and $\epsilon\in (0,1]$.   This can be improved to obtain decay at infinity as follows.
\begin{proposition}
As in Theorem~\ref{mae.4}, let $\widetilde{\delta}$ be a positive multiweight satisfying the conditions of Corollary~\ref{w.53}.  If $K\subset \bbC^{m+1}$ is a compact set containing an open neighborhood of the origin, then on $\cH_{\epsilon}\setminus K$, there is for each $k$ a positive constant $C_k$ such that
\begin{equation}
  \| x^{-\widetilde{\delta}}u_{t,\epsilon}\|_{\cC^k_{\nQb}(\cH_{\epsilon}\setminus(\cH_{\epsilon}\cap K))}\le C_k
\label{sing.21a}\end{equation}
for all $t\in [0,1]$ and $\epsilon\in (0,1]$.
\label{sing.21}\end{proposition}
\begin{proof}
As in the proof of Theorem~\ref{mae.4}, the idea is to adapt the Moser iteration with weight of \cite[\S~8.6.2 and \S~9.6.2]{Joyce} to obtain a uniform estimate in $t$ and $\epsilon$.  However, there is an obstacle in that the upper bound in \eqref{sing.19b} degenerates as we approach the origin while the statement of \cite[Proposition~8.6.7]{Joyce} assumes a global uniform upper bound to obtain the estimate \cite[(8.23)]{Joyce}.  However, as noticed in \cite[Proposition~4.10]{CGT}, this can be easily adjusted, since assuming without loss of generality that $\rho$ ($r$ in the notation of \cite{CGT}) is constant on the compact set $K\subset \bbC^{m+1}$, we only need a uniform upper bound on $\widehat{\omega}_{t,\epsilon}$ on $\cH_{\epsilon}\setminus (\cH_{\epsilon}\cap K)$ for \cite[(8.23)]{Joyce} to hold; see the second half of \cite[p.889]{CGT} for details.  

The Sobolev inequality is also used in \cite[\S~8.6.2 and \S~9.6.2]{Joyce}, so to derive uniform estimates, we need to use the uniform Sobolev inequality of \eqref{sing.18}.  As in the proof of Lemma~\ref{sing.19}, this can be achieved by first estimating $u_{0,\epsilon}$, and then $u_{t,\epsilon}-u_{0,\epsilon}$.  With these adjustments, we can obtain for some $\mu_1>0$  a $\rho^{-\mu_1}\cC^0(\cH_{\epsilon}\setminus(\cH_{\epsilon}\cap K))$-bound for $u_{t,\epsilon}$ which is uniform in $t\in [0,1]$ and $\epsilon\in (0,1]$.  Proceeding as in the second half of the proof of Theorem~\ref{mae.4}, this can be improved to the claimed uniform estimates.  
\end{proof}

These estimates can be used to obtain a solution to the complex Monge-Amp\`ere \eqref{sing.15} for $\epsilon=0$ as follows.

\begin{theorem}
For each $t\in [0,1]$, there exists a unique bounded continuous function $u_{t,0}$ on $\cH_0$ which is smooth on $\cH_0\setminus \{0\}$, is a solution of the Monge-Amp\`ere equation  \eqref{sing.15} for $\epsilon=0$ in the sense of Bedford and Taylor and is such that 
$$
       u_{t,0}|_{\cH_0\setminus (\cH_0\cap K)}\in x^{\widetilde{\delta}}\CI_{\nQb}(\cH_0\setminus (\cH_0\cap K)),
$$
where $K\subset \bbC^{m+1}$ is any compact set containing an open neighborhood of the origin and as in Theorem~\ref{mae.4}, $\widetilde{\delta}$ is a positive multiweight satisfying the conditions of Corollary~\ref{w.53}.
\label{sing.22}\end{theorem}
\begin{proof}
For $\epsilon>0$, there is by Theorem~\ref{mae.4} a unique solution $u_{t,\epsilon}\in x^{\widetilde{\delta}}\CI_{\nQb}(\cH_{\epsilon})$ to the complex Monge-Amp\`ere equation \eqref{sing.15}.  By the uniform estimates of Proposition~\ref{sing.21}, as $\epsilon\searrow 0$, we can extract a decreasing sequence $\{\epsilon_i\}\subset (0,1]$ with $\epsilon_i\searrow 0$ such that $u_{t,\epsilon_i}\to u_{t,0}\in x^{\widetilde{\delta}}\CI_{\nQb}(\overline{\cH_0\setminus (\cH_0\cap K)})$ uniformly.  Taking a subsequence, we can in fact assume as well that $u_{t,0}$ is defined on $\cH_0\setminus \{0\}$ and is smooth with $u_{t,\epsilon_i} \to u_{t,0}$ locally uniformly in the smooth topology.  In particular, $u_{t,0}$ is a solution to the complex Monge-Amp\`ere equation \eqref{sing.15} with $\epsilon=0$ on $\cH_0\setminus \{0\}$.  By the uniform upper bound \eqref{sing.19a}, $|u_{t,0}|$ is bounded on $\cH_0\setminus \{0\}$.  Hence, by \cite[Theorem~1.7 and Theorem~2.2]{Demailly}, $u_{t,0}$ uniquely extends to a continuous bounded function on $\cH_0$ which is a solution to \eqref{sing.15} with $\epsilon=0$ in the sense of Bedford and Taylor \cite{Bedford-Taylor}.  To show that such a solution $u_{t,0}$ is unique, we can proceed as in the proof of \cite[Proposition~4.17]{CGT}, but using \cite[Lemma~3.2]{Bedford1982} to justify integration by parts in the sense of currents near the origin.     
\end{proof}

To describe the expected behavior of the metrics of Theorem~\ref{sing.22} near the origin, we need to introduce some notations and definitions.  First, when $t=0$, $\widehat{\omega}_{0,0}= \widetilde{\omega}_0$ coincides with the Calabi-Yau metric $\omega_{\cH_0}$ in some open neighborhood of the origin.  If $\cU\subset \cH_0$ is an open neighborhood of the origin and $E\to \cU\setminus\{0\}$ is a Euclidean vector bundle, consider for $\ell\in \bbN_0$ the space
$$
          \cC_b^{\ell}(\cU\setminus \{0\};E):= \cC^{\ell}_{\frac{g_{\cH_0}}{r_{\cH_0}^2}}(\cU\setminus\{0\};E)
$$
and for $\alpha\in (0,1)$ the corresponding H\"older space
$$
   \cC_b^{\ell,\alpha}(\cU\setminus \{0\};E):= \cC^{\ell,\alpha}_{\frac{g_{\cH_0}}{r_{\cH_0}^2}}(\cU\setminus\{0\};E),
$$
where $r_{\cH_0}$ is the distance function introduced in \eqref{sing.10b}.  Let also $\CI_b(\cU\setminus\{0\};E)= \bigcap_{\ell\in\bbN_0}\cC^{\ell}_b(\cU\setminus\{0\};E)$ be the corresponding Fr\'echet space.  
\begin{definition}
A K\"ahler metric $g$ on $\cH_0\setminus\{0\}$ is \textbf{asymptotic with rate} $\lambda>0$ to $g_{\cH_0}$ at the origin if there exists  open neighborhoods $\cU$ and $\cV$ of the origin in $\cH_0$ and a biholomorphism 
$\Phi: \cU\to \cV$ such that
\begin{equation}
   \Phi^*g-g_{\cH_0}\in r_{\cH_0}^{\lambda}\CI_b(\cU\setminus\{0\}; S^2(T^*\cU))
\label{sing.23a}\end{equation}
with Euclidean metric on the bundle $S^2(T^*\cU)$ induced by $g_{\cH_0}$.  We will also say that $g$ has a \textbf{conical singularity asymptotically modelled on} $g_{\cH_0}$ at the origin if $g$ is asymptotic with rate $\lambda$ to $g_{\cH_0}$ for some $\lambda>0$.  
\label{sing.23}\end{definition} 

We will be interested in the case where such a metric $g$ corresponds to an $\mathfrak{n}$-warped $\QAC$-metric on $\cH_0\setminus K$ for some compact set $K\subset \cH_0$ containing an open neighborhood of the origin.  In fact, our goal will be to prove the following theorem.
\begin{theorem}
The Calabi-Yau metric $\widehat{\omega}_{0,0}= \omega_{0}+\sqrt{-1}\pa\db u_{0,0}$ given by the unique solution $u_{0,0}$ of Theorem~\ref{sing.22} to the complex Monge-Amp\`ere-equation \eqref{sing.15} with $\epsilon=t=0$ has a conical singularity asymptotically modelled on $g_{\cH_0}$ at the origin.
\label{sing.37}\end{theorem}

To prove this theorem, we will apply the continuity method of Hein-Sun \cite{Hein-Sun} to the family of solutions of Theorem~\ref{sing.22}.  Thus, let $\widehat{\omega}_{t,0}=\omega_0+\sqrt{-1}\pa\db u_{t,0}$ be the family of K\"ahler forms of Theorem~\ref{sing.22} and consider the subset $T$ of $[0,1]$ consisting of all $t\in [0,1]$ such that $\widehat{\omega}_{t,0}$ has a conical singularity at the origin asymptotically modelled on $g_{\cH_0}$.  Since the solutions in Theorem~\ref{sing.22} are unique, we know that $\widehat{\omega}_{1,0}=\widetilde{\omega}_0$, so when $t=1$, $\widehat{\omega}_{1,0}$ agrees with $\omega_{\cH_0}$, showing that $1\in T$.  To prove Theorem~\ref{sing.37}, it suffices then to show that $T$ is both open and closed in $[0,1]$.  

\subsection{Openness}

To show that the subset $T$ is open, the main ingredient missing to proceed as in \cite[\S~2.3]{Hein-Sun} is a version of the isomorphism theorem \cite[Theorem~2.11]{Hein-Sun} suitably adapted to our setting.  The idea will be to combine \cite[Theorem~2.11]{Hein-Sun} with Corollary~\ref{w.53}, but this will require a few steps.

Thus, let $g$ be a K\"ahler $\mathfrak{n}$-warped $\QAC$-metric on $\cH_0\setminus \{0\}$ with respect to the compactification $\hX$ and the $\mathfrak{n}$-weighted total boundary defining function $\rho^{-1}=(x_1x_2)^{\frac{1}{1-\nu}}$ and suppose that it is asymptotic at rate $\lambda>0$ to $g_{\cH_0}$ at the origin. Let $\Phi: \cU\to \cV$ be the biholomorphism such that \eqref{sing.23a} holds.  Let $\widetilde{r}\in \CI(\cH_0\setminus\{0\})$ be a positive function equal to a constant outside a compact set containing an open neighborhood of the origin and such that $\Phi^*\widetilde{r}= r_{\cH_0}$ in $\cU\setminus\{0\}$. Consider the metric 
\begin{equation}
    g_{b-\nQb}= \frac{\chi^2}{\widetilde{r}^2}g,
\label{sing.24}\end{equation}
where $\chi= \frac{x_2}{\rho^{\mathfrak{n}}}=\frac{x_2}{(x_1x_2)^{-\frac{\nu}{1-\nu}}}=x_1^{\frac{\nu}{1-\nu}}x_2^{\frac{1}{1-\nu}}=(\rho w)^{-1}$ is the function introduced in \eqref{hol.2b}.  Since $\widetilde{r}$ is constant outside a compact set, this is an $\nQb$-metric outside a compact set.  Similarly, since we assume that  the boundary defining functions  $x_1$ and $x_2$ are constant near the origin, we see that in this region, it is asymptotically cylindrical, a particular case of $b$-metric in the sense of Melrose \cite{MelroseAPS}.  For these reasons, we will say that $g_{b-\nQb}$ is a $b-\nQb$-metric.  Moreover, if $\Delta_g= g^{ij}\nabla_i\nabla_j$ is the scalar Laplacian of $g$, then $\widetilde{r}^2\Delta_g$ is near the origin a $b$-operator of order $2$ which is asymptotically identified with 
\begin{equation}
   r^2_{\cH_0}\Delta_{g_{\cH_0}}= \lrp{ r_{\cH_0}\frac{\pa}{\pa r_{\cH_0}}  }^2+ (2m-2)r_{\cH_0}\frac{\pa}{\pa r_{\cH_0}}+ \Delta_{\kappa},
\label{sing.25}\end{equation}
where $\Delta_{\kappa}$ is the scalar Laplacian of the metric $\kappa$ on the link of the cone $(\cH_0,g_{\cH_0})$, that is, such that
\begin{equation}
 g_{\cH_0}= dr^2_{\cH_0}+ r_{\cH_0}^2\kappa.
\label{sing.26}\end{equation}
Taking the Mellin transform, the indicial family of \eqref{sing.25} is 
\begin{equation}
   I(r^2_{\cH_0}\Delta_{g_{\cH_0}},\lambda)=  \left. \lrp{r_{\cH_0}^{-\lambda}(r^2_{\cH_0} \Delta_{g_{\cH_0}})r_{\cH_0}^{\lambda}} \right|_{r_{\cH_0}=0}= \lambda^2+ (2m-2)\lambda + \Delta_{\kappa}.
\label{sing.27}\end{equation}
If $0=-\lambda_0\ge -\lambda_1\ge \cdots \ge -\lambda_i\ge\cdots$ are the eigenvalues of $\Delta_{\kappa}$, then the indicial roots of $I(r^2_{\cH_0}\Delta_{g_{\cH_0}},\lambda)$, that is, the values of $\lambda$ for which $I(r^2_{\cH_0},\Delta_{g_{\cH_0}},\lambda)$ is not invertible, are given by 
\begin{equation}
   \cI_b(r^{2}_{\cH_0}\Delta_{g_{\cH_0}})= \{ -(m-1)\pm \sqrt{(m-1)^2+\mu} \; | \; -\mu\in \Spec(\Delta_{\kappa})\}.
\label{sing.28}\end{equation}
They are all real, so this set also corresponds to the set $\Crit_b(r^{2}_{\cH_0}\Delta_{g_{\cH_0}})$ of critical weights of the indicial family, which is the set of real parts of indicial roots. 

In particular, for $\mu=\lambda_0=0$, this shows that $-2(m-1)$ and $0$ are critical weights and that there are no critical weights in the interval $(-2(m-1),0)$.  For $\ell\in\bbN_0$, let us consider the Banach space
$$
    \cC^{\ell}_{b-\nQb}(\cH_0\setminus \{0\}):= \cC^{\ell}_{g_{b-\nQb}}(\cH_0\setminus\{0\})
$$
and for $\alpha\in (0,1)$ the associated  H\"older space
$$
    \cC^{\ell,\alpha}_{b-\nQb}(\cH_0\setminus \{0\}):= \cC^{\ell,\alpha}_{g_{b-\nQb}}(\cH_0\setminus\{0\}).
$$
Let also $\CI_{b-\nQb}(\cH_0\setminus\{0\})= \bigcap_{\ell\in\bbN_0}\cC^{\ell}_{b-\nQb}(\cH_0\setminus\{0\})$ be the corresponding Fr\'echet space.  Although we are ultimately mostly interested in mapping properties of the Laplacian on these H\"older spaces, it will also be useful to consider the corresponding $L^2$-Sobolev space of order $\ell\in\bbN_0$, 
$$
        H^{\ell}_{b-\nQb}(\cH_0\setminus\{0\})= H^{\ell}_{g_{b-\nQb}}(\cH_0\setminus\{0\}),
$$
associated to the metric $g_{b-\nQb}$, and by analogy with \eqref{Sob.1}, to consider the weighted version
$$
      H^{\ell}_{b-w}(\cH_0\setminus\{0\}):= \chi^m H^{\ell}_{b-\nQb}(\cH_0\setminus\{0\}).
$$
Using standard parametrices for the Laplacian on a manifold with conical singularities and Corollary~\ref{w.53}, we obtain the following isomorphism theorem for $\Delta_g$.  
\begin{theorem}
Let $g$ be a K\"ahler $\mathfrak{n}$-warped $\QAC$-metric on $\cH_0\setminus\{0\}$ asymptotic with rate $\lambda>0$ to $g_{\cH_0}$ at the origin.  Then for all $\ell\in\bbN_0$, the mapping
\begin{equation}
   \Delta_g: \widetilde{r}^a x^{\delta+\mathfrak{w}}\rho^{\nu m}H^{\ell+2}_{b-w}(\cH_0\setminus\{0\}) \to \chi^2 \widetilde{r}^{a-2} x^{\delta+\mathfrak{w}}\rho^{\nu m}H^{\ell}_{b-w}(\cH_0\setminus\{0\})
\label{sing.29a}\end{equation}
is an isomorphism provided the weights $a\in \bbR$ and $\delta=(\delta_1,\delta_2)\in \bbR^2$ are such that
\begin{equation}
    2-2m<a<0, \quad 0<\delta_2< \frac{2m-2}{1-\nu} \quad \mbox{and} \quad \delta_2+4-2m< \delta_1<\delta_2,
\label{sing.29b}\end{equation}
where $\nu=\frac{m-k}{m-1}$ is specified by Examples~\ref{int.2} and \ref{int.2b}.
\label{sing.29}\end{theorem}
\begin{proof}
By Corollary~\ref{w.53}, we know how to invert $\Delta_g$ up to a compact operator near infinity.  By standard results \cite{MelroseAPS}, see for instance \cite[Theorem~A.3]{CR2021}, we also know how to invert $\Delta_g$ up to a compact error near the origin provided $a$ is not a critical weight of the indicial family \eqref{sing.27}.  Using the fact that $\Delta_g$ is elliptic, these parametrices can be combined to see that the map \eqref{sing.29a} is Fredholm for $a$ and $\delta$ as in \eqref{sing.29b}.  To see that those maps are in fact isomorphisms, let us first consider the map \eqref{sing.29a} with $a=-m+2$, $\delta_1=\frac{\nu(m-1)+1}{1-\nu}$ and $\delta_2=\frac{m}{1-\nu}$, in which case this corresponds to the map
\begin{equation}
    \Delta_g: \widetilde{r}^{2}H^{\ell+2}_{c-w}(\cH_0\setminus \{0\})\to \chi^2H^{\ell}_{c-w}(\cH_0\setminus\{0\}),
\label{sing.30}\end{equation}
where 
$$
    H^{\ell}_{c-w}(\cH_0\setminus\{0\}):= \lrp{\frac{\chi}{\widetilde{r}}}^m H^{\ell}_{b-\nQb}(\cH_0\setminus\{0\})
$$
is such that $H^{0}_{c-w}(\cH_0\setminus \{0\})= L^2(\cH_0\setminus \{0\},g)$ is the $L^2$-space associated to the metric $g$.  Since there is a continuous inclusion
$$
              \widetilde{r}^{2}H^2_{c-w}(\cH_0\setminus)\subset H^2_{g}(\cH_0\setminus\{0\}),
$$ 
this means that
\begin{equation}
  \int_{\cH_0} u \Delta_g u dg= -\int_{\cH_0}|du|_g^2dg
\label{sing.31}\end{equation}
for $u\in \widetilde{r}^{2}H^2_{c-w}(\cH_0\setminus\{0\})$.  Indeed, the equality \eqref{sing.31} follows by integration by parts when $u$ is smooth of compact support on $\cH_0\setminus\{0\}$, so it holds more generally for $u\in \widetilde{r}^{2}H^2_{c-w}(\cH_0\setminus)$ via approximation by smooth compactly supported functions.  Hence, if for such $u$, $\Delta_g u=0$, this implies that $u$ must be constant.  But since $u\in \widetilde{r}^2\cH^2_{c-w}(\cH_0\setminus\{0\})$, $u$ must also decay at infinity so this implies that $u=0$.  In other words, the kernel of \eqref{sing.29a} is trivial when $a\ge 2-m$, $\delta_1\ge \frac{\nu(m-1)+1}{1-\nu}$ and $\delta_2\ge \frac{m}{1-\nu}$. 

On the other hand, with respect to the $L^2$-inner product induced by the metric $g$, the dual of 
$$\chi^2 \widetilde{r}^{a-2} x^{\delta+\mathfrak{w}}\rho^{\nu m}L^2_{b-w}(\cH_0\setminus\{0\})
$$ 
is the space 
$$\widetilde{r}^{a^*} x^{\delta^*+\mathfrak{w}}\rho^{\nu m}L^{2}_{b-w}(\cH_0\setminus\{0\})
$$ 
with $a^*= 2-2m-a$ and $\delta^*=(\delta_1^*,\delta_2^*)$ with
$$
    \delta_2^*= \frac{2m-2}{1-\nu}-\delta_2 \quad \mbox{and} \quad  \delta_1^*= 2\lrp{ \frac{\nu(m-2)+1}{1-\nu}  }-\delta_1. 
$$
A simple computation shows that \eqref{sing.29b} holds for $a$ and $\delta$ if and only if it holds for $a^*$ and $\delta^*$.  Since $\Delta_g$ is formally self-adjoint with respect to the $L^2$-inner product induced by $g$, this shows that when $\ell=0$, the cokernel of \eqref{sing.29a} is identified with the kernel of $\Delta_g$ in $\widetilde{r}^{a^*} x^{\delta^*+\mathfrak{w}}\rho^{\nu m}L^{2}_{b-w}(\cH_0\setminus\{0\})$.  If $u$ is an element of this kernel, then 
$$
 \begin{aligned}
 \Delta_g u=0 \quad &\Longrightarrow \quad \Delta_g \chi^{-1}\widetilde{r}(\widetilde{r}^{-1}\chi u)=0 \\
              & \Longrightarrow \quad  (\widetilde{r}\chi^{-1}\Delta_g \chi^{-1}\widetilde{r})(\widetilde{r}^{-1}\chi u)=0.
 \end{aligned}
$$
Hence, by elliptic regularity with respect to the operator $(\widetilde{r}\chi^{-1}\Delta_g \chi^{-1}\widetilde{r})$, which is uniformly elliptic with respect to the $b-\nQb$-metric $g_{b-\nQb}=\frac{\chi^2 g}{\widetilde{r}^2}$, we see that 
$$
      \widetilde{r}^{-1}\chi u \in \chi\widetilde{r}^{a^*-1} x^{\delta^*+\mathfrak{w}}\rho^{\nu m}H^{\ell}_{b-w}(\cH_0\setminus\{0\}) \quad \forall \ell\in \bbN_0,
$$
that is,
$$
  u \in \widetilde{r}^{a^*} x^{\delta^*+\mathfrak{w}}\rho^{\nu m}H^{\ell}_{b-w}(\cH_0\setminus\{0\}) \quad \forall \ell\in \bbN_0.
$$
Hence, by the injectivity result already obtained, this kernel is trivial when 
$$
     a^*\ge 2-m, \quad \delta_1^*\ge \frac{\nu(m-1)+1}{1-\nu} \quad \mbox{and}  \quad \delta_2^*\ge \frac{m}{1-\nu},
$$
which implies that for $\ell=0$, \eqref{sing.29a} is surjective when 
$$
       a\le -m, \quad \delta_1\le \frac{\nu(m-3)+1}{1-\nu} \quad \mbox{and}  \quad \delta_2\le \frac{m-2}{1-\nu}.
$$
Now, the dimension of the kernel of \eqref{sing.29a} is non-decreasing when the weights $a$, $\delta_1$ or $\delta_2$ decrease, while the dimension of its cokernel is non-increasing.  Thus, when $a, \delta_1$ or $\delta_2$ decrease, the index of \eqref{sing.29a} is non-decreasing.  But conjugating the Laplacian by appropriate weight functions, the operators in \eqref{sing.29a} can be seen as a continuous family of Fredholm operators for $a$ and $\delta$ as in \eqref{sing.29b}, so the index must be constant.  For $\ell=0$, this implies that the dimensions of the kernel and the cokernel of \eqref{sing.29a} do not depend on the choice of $a$ and $\delta$ satisfying \eqref{sing.29b}.  By the discussion above, both must therefore be equal to zero, completing the proof of the theorem when $\ell=0$.  If $\ell\ge 1$, the injectivity of the operator \eqref{sing.29a} immediately follows from the one for $\ell=0$, while the surjectivity follows from the surjectivity when $\ell=0$ and elliptic regularity for the $b-\nQb$-elliptic operator $(\widetilde{r}\chi^{-1}\Delta_g \chi^{-1}\widetilde{r})$.
\end{proof}

This result can be used to deduce a corresponding isomorphism theorem for weighted H\"older spaces.
\begin{corollary}
For all $\ell\in\bbN_0$ and $\alpha\in (0,1)$, the mapping 
\begin{equation}
      \Delta_g: \widetilde{r}^a x^{\delta}\cC^{\ell+2,\alpha}_{b-\nQb}(\cH_0\setminus\{0\})\to \chi^2\widetilde{r}^{a-2} x^{\delta}\cC^{\ell,\alpha}_{b-\nQb}(\cH_0\setminus\{0\})
\label{sing.32a}\end{equation}
is an isomorphism provided $a$ and $\delta=(\delta_1,\delta_2)$ are such that
\begin{equation}
  2-2m<a <0, \quad 0<\delta_2<\frac{2m-2}{1-\nu} \quad \mbox{and} \quad \delta_2+4-2m<\delta_1<\delta_2-1.
\label{sing.32b}\end{equation}
\label{sing.32}\end{corollary}
\begin{proof}
Since 
$$
\widetilde{r}^a x^{\delta}\cC^{\ell+2,\alpha}_{b-\nQb}(\cH_0\setminus\{0\}) \subset \widetilde{r}^{a-\epsilon} x^{\delta+\mathfrak{w}}(x_1x_2)^{-\epsilon}\rho^{\nu m}H^{\ell+2}_{b-w}(\cH_0\setminus\{0\}) 
$$ 
for any $\epsilon>0$, injectivity follows from the injectivity of \eqref{sing.29a}.  For surjectivity, let 
$$
f\in \chi^2\widetilde{r}^{a-2}x^{\delta}\cC^{\ell,\alpha}_{b-\nQb}(\cH_0\setminus\{0\})
$$ 
be given.  We need to find $u\in \widetilde{r}^{a}x^{\delta}\cC^{\ell+2,\alpha}_{b-\nQb}(\cH_0\setminus\{0\})$ such that $\Delta_g u=f$.  Using the invertibility of $\Delta_{g_0'}$ on $\cH_0'$ given by Corollary~\ref{w.53}, we can assume that $f$ is supported in a large ball containing the origin.  Similarly, consider a metric $\hat{g}$ on $\cH_0\setminus \{0\}$ which agrees with $g_{\cH_0}$ near the origin, but also has an isolated conical singularity at infinity modelled on the one of $g_{\cH_0}$ near $0$, so that $(\cH_0\setminus \{0\}, \hat{g})$ is a Riemannian manifold of bounded diameter with two isolated conical singularities.  Using standard mapping properties of the Laplacian $\Delta_{\hat{g}}$, see for instance \cite[Proposition~2.7 and Theorem~2.11]{Hein-Sun}, as well as smooth cut-off functions, we can also reduce to the case where $f$ is supported away from the origin, that is, $f\in \cC_{c}^{\ell,\alpha}(\cH_0\setminus\{0\})$ has compact support.  Since by assumption $\delta_1+1<\delta_2$, we know from Theorem~\ref{sing.29} that there exists a unique  $u\in \widetilde{r}^ax^{\delta+\mathfrak{w}}x_1\rho^{\nu m}H^{\ell+2}_{b-w}(\cH_0\setminus\{0\})$ such that
$$
      \Delta_g u=f.
$$
The reason to apply Theorem~\ref{sing.29} replacing $\delta_1$ by $\delta_1+1$ is that this ensures that in terms of the $L^2$-Sobolev spaces of the metric $g_{b-\nQb}$,
$$
     u\in  \widetilde{r}^a x^{\delta}H^{\ell+2}_{b-\nQb}(\cH_0\setminus\{0\}).
$$
If $\ell+2>m$, this implies by the Sobolev embedding that $\widetilde{r}^{-a}x^{-\delta}u$ is continuous and bounded.  Since the metric $g_{b-\nQb}$ is of bounded geometry, by the Schauder estimates applied to the equation
\begin{equation}
   [(\chi \widetilde{r}^{a-1}x^{\delta})^{-1} (\widetilde{r}\chi^{-1}\Delta_g \chi^{-1}\widetilde{r})(\chi \widetilde{r}^{a-1}x^{\delta})]    ((\widetilde{r}^{-a}x^{-\delta})u)=  \chi^{-2} \widetilde{r}^{2-a}x^{-\delta}f\in \cC^{\ell,\alpha}_{b-\nQb}(\cH_0\setminus\{0\}),
\label{sing.33}\end{equation}
where $[(\chi \widetilde{r}^{a-1}x^{\delta})^{-1} (\widetilde{r}\chi^{-1}\Delta_g \chi^{-1}\widetilde{r})(\chi \widetilde{r}^{a-1}x^{\delta})]$ is uniformly elliptic with respect to the metric $g_{b-\nQb}$, we see that 
$$
    u\in \widetilde{r}^a x^{\delta}\cC^{\ell+2,\alpha}_{b-\nQb}(\cH_0\setminus\{0\}),
$$ 
establishing surjectivity in this case.  If $\ell+2\le m$, then by \cite[p.148]{Gilbarg-Trudinger}, there exists a sequence 
$$\{f_i\}\subset \CI_c(\cH_0\setminus\{0\})$$ bounded in $\chi^2\widetilde{r}^{a-2} x^{\delta}\cC^{\ell,\alpha}_{b-\nQb}(\cH_0\setminus\{0\})$ and converging to $f$ in the topology of $\chi^2\widetilde{r}^{a-2} x^{\delta}\cC^{\ell,\alpha'}_{b-\nQb}(\cH_0\setminus\{0\})$ for some fixed $0<\alpha'<\alpha$.  By the result above, we can find a sequence $\{u_i\}$ in $\widetilde{r}^{a} x^{\delta}\cC^{\ell'+2,\alpha}_{b-\nQb}(\cH_0\setminus\{0\})$ for $\ell'>m-2$ such that 
$$
     \Delta_g u_i= f_i.
$$
In particular, $\{u_i\}$ is a sequence in  $\widetilde{r}^{a} x^{\delta}\cC^{\ell'+2,\alpha}_{b-\nQb}(\cH_0\setminus\{0\})$.  By the boundedness of $\{f_i\}$ and the injectivity of \eqref{sing.32a}, the sequence $\{u_i\}$ is bounded in $\widetilde{r}^{a} x^{\delta}\cC^{\ell'+2,\alpha}_{b-\nQb}(\cH_0\setminus\{0\})$.  Hence, taking a subsequence if needed, we can assume that $\{u_i\}$ converges to some $u\in \widetilde{r}^{a} x^{\delta}\cC^{\ell+2,\alpha}_{b-\nQb}(\cH_0\setminus\{0\})$ in the topology of $\widetilde{r}^{a} x^{\delta}\cC^{\ell+2}_{b-\nQb}(\cH_0\setminus\{0\})$.  In particular, this means that
$$
    \Delta_g u= \lim_{i\to \infty} \Delta_g u_i= f,
$$
establishing surjectivity when $\ell+2\le m$.
\end{proof}

To show openness in the continuity method, we need in fact such an isomorphism for $a>2$.  Since this amounts to taking a smaller H\"older space or Sobolev space, the Laplacian will still be injective.  However, since $0$ is a critical weight of the indicial family \eqref{sing.27}, we know from the relative index theorem of Melrose \cite[\S~6.2]{MelroseAPS} that the Laplacian will no longer be surjective.  However, as long as $a$ is not a critical weight, it will remain Fredholm, so to make it surjective, it will suffices to add a suitable finite dimensional space to the domain.  

To describe this, for $a\in \Crit_b(r^2_{\cH_0}\Delta_{g_{\cH_0}} )$ a critical weight of the indicial family \eqref{sing.27}, let $\cP_{a}$ be the space of harmonic functions homogeneous of degree $a$ with respect to \eqref{sing.25},
\begin{equation}
  \cP_{a}= \{ r_{\cH_0}^{a}u \; | \; u\in \CI(L), \quad r^2_{\cH_0}\Delta_{g_{\cH_0}} r_{\cH_0}^a u=0\},
\label{sing.34}\end{equation}  
where we recall that $(L,\kappa)$ denotes the link of the cone $(\cH_0,g_{\cH_0})$.  Let $\psi\in \CI(\cU\setminus\{0\})$ be a cut-off function identically equal to $1$ near the origin and equal to zero outside a compact set of $\cU$ containing an open neighborhood of the origin.  For an interval $I\subset \bbR$, set 
$$
       \cP_I= \bigcup_{a\in I\cap\Crit_b(r^2_{\cH_0}\Delta_{g_{cH_0}})} \cP_{a} \quad \mbox{and} \quad \psi \cP_I= \{ \psi u\; | \; u\in \cP_I\}.
$$ 
\begin{theorem}
Let $g$ be an $\mathfrak{n}$-warped $\QAC$-metric  asymptotic with rate $\lambda>0$ to $g_{\cH_0}$ near the origin.  
Then for $a>2$ such that $a-2<\min\{\lambda,1\}$ and $(2,a]\cap \Crit_b(r^2_{\cH_0}\Delta_{g_{\cH_0}})=\emptyset$, the mapping 
\begin{equation}
 \Delta_g: \widetilde{r}^a x^{\delta+\mathfrak{w}}\rho^{\nu m}H^{\ell+2}_{b-w}(\cH_0\setminus\{0\})+ \Phi^*(\chi \cP_{[0,2]}) \to \chi^2 \widetilde{r}^{a-2} x^{\delta+\mathfrak{w}}\rho^{\nu m}H^{\ell}_{b-w}(\cH_0\setminus\{0\})
\label{sing.35a}\end{equation}
is an isomorphism for all $\ell\in \bbN_0$ provided the multiweight $\delta=(\delta_1,\delta_2)$ is such that
\begin{equation}
0<\delta_2< \frac{2m-2}{1-\nu} \quad \mbox{and} \quad \delta_2+4-2m< \delta_1<\delta_2.
\label{sing.35b}\end{equation}
Similarly, for $a$ as above, the mapping
\begin{equation}
  \Delta_g: \widetilde{r}^a x^{\delta}\cC^{\ell+2,\alpha}_{b-\nQb}(\cH_0\setminus\{0\}) + \Phi^*(\chi \cP_{[0,2]})  \to \chi^2\widetilde{r}^{a-2} x^{\delta}\cC^{\ell,\alpha}_{b-\nQb}(\cH_0\setminus\{0\})
\label{sing.35c}\end{equation}
is an isomorphism for all $\ell\in \bbN_0$ and $\alpha\in(0,1)$ provided 
\begin{equation}
0<\delta_2< \frac{2m-2}{1-\nu} \quad \mbox{and} \quad \delta_2+4-2m< \delta_1<\delta_2-1.
\label{sing.35d}\end{equation}
\label{sing.35}\end{theorem}
\begin{proof}
For $a=0$, the space $\cP_a$ corresponds to constant functions, while by \cite[Theorem~2.14]{Hein-Sun}, for \linebreak $a\in (0,2)\cap \Crit_b(r^2_{\cH_0}\Delta_{g_{\cH_0}})$, elements in $\cP_a$ are pluriharmonic.  In particular, in these cases, $u\in \cP_a$ is such that 
\begin{equation}
     \Delta_g \Phi_*(\psi u)\in \CI_c(\cH_0\setminus\{0\}).
\label{sing.36}\end{equation}
To see that \eqref{sing.35a} is an isomorphism, we start with $a<0$ where it is an isomorphism, and progressively increase it.  Each time we pass a critical weight $a_0\in [0,2]$, it suffices by the relative index formula of Melrose \cite[\S~6.2]{MelroseAPS} to add $\Phi^*(\psi\cP_{a_0})$ to the domain to still have an isomorphism, noticing that \eqref{sing.36} ensures that the map \eqref{sing.35a} will still be well-defined when we continue to increase the weight $a$.  However, if $2\in\Crit_b(r^2_{\cH_0}\Delta_{g_{\cH_0}})$, we do not know that $u\in \cP_2$ is pluriharmonic, but since $g$ is asymptotic with rate $\lambda>0$ to $g_{\cH_0}$, we know at least that 
$$
   \Delta_g \Phi_*(\psi u)\in \widetilde{r}^{\min\{1,\lambda\}}\CI_{b}(\cH_0\setminus\{0\})
$$ 
with support contained in a small ball containing the origin, so in particular
$$
  \Delta_g \Phi^*(\psi u)\in \chi^2\widetilde{r}^{a-2}x^{\delta+\mathfrak{w}}\rho^{\nu m}H^{\ell}_{b-w}(\cH_0\setminus\{0\})
$$
provided $a-2<\min\{1,\lambda\}$.  Thus, to pass this last critical weight, we can include $\Phi_*(\psi \cP_2)$ in the domain to keep the operator surjective (and injective), showing that the map \eqref{sing.35a} is an isomorphism.  

Knowing this, we can proceed essentially as in the proof of Corollary~\ref{sing.32} to conclude that the mapping \eqref{sing.35c} is also an isomorphism.
\end{proof}

To show that $T$ is open, we can now proceed exactly as in \cite[\S~2.3]{Hein-Sun}, but with the isomorphism theorem \cite[Theorem~2.11]{Hein-Sun} replaced by Theorem~\ref{sing.35}.  Notice moreover that since our potentials decay at infinity, compared to \cite[\S~2.3]{Hein-Sun}, no normalization of the form 
$$
    \int_{\cH_0\setminus\{0\}} u_{t,0} \omega^m_{t,0}=0 \quad \mbox{or}  \int_{\cH_0\setminus\{0\}} f \omega^m_{t,0}=0
$$
must be imposed.  

\subsection{Closedness}

To show that $T$ is closed, we will follow the strategy of \cite[\S~3]{Hein-Sun} suitably adapted to our non-compact setting.  First, fixing a compact set $K\subset \bbC^{m+1}$ containing an open neighborhood of the origin and using the fact that $\widehat{\omega}_{0,\epsilon}$ is Calabi-Yau, we can proceed as in the proof of \cite[Lemma~4.14]{CGT} to conclude that there is a constant $C>0$ such that
\begin{equation}
     \Diam_{\widehat{\omega}_{0,\epsilon}} (K\cap \cH_{\epsilon}) \le C \quad \forall \epsilon\in (0,1],
\label{sing.38}\end{equation}
where $   \Diam_{\widehat{\omega}_{t,\epsilon}} (K\cap \cH_{\epsilon}) $ is the diameter of $K\cap \cH_{\epsilon}$ with respect to the metric $\widehat{\omega}_{t,\epsilon}$.  Indeed, \cite[Lemma~4.14]{CGT} is formulated in terms of metrics that are asymptotically conical at infinity, but this is only used to obtain the following asymptotic behavior of the volume of balls of radius $r$:
\begin{equation}
   \lim_{r\to \infty} \frac{\Vol(B_r(x_i))}{r^{2m}}= \Vol_{g_{C_0}}(L_0)>0,
\label{sing.39}\end{equation}
where $(C_0,g_{C_0})$ is the tangent cone at infinity with link $L_0$.  The main difference in our setting is that the link $L_0$ is now singular, but the asymptotic behavior  \eqref{sing.39} still holds, so the proof \cite[Lemma~4.14]{CGT} applies and yields the uniform bound \eqref{sing.38}.

Unfortunately, the proof of \cite[Lemma~4.14]{CGT} uses in a crucial way  volume comparison for Ricci-flat metrics, so it cannot be applied to the metrics $\widehat{\omega}_{t,\epsilon}$ when $t>0$ since those are not Calabi-Yau.  In this case, we can instead use the argument by contradiction in the proof of \cite[Proposition~3.2]{Hein-Sun}.  To do so, we need the following result.
\begin{lemma}
Let $\{(t_i,\epsilon_i)\}\subset [0,1]\times(0,1]$ be a sequence with $t_i\to t\in [0,1]$ and $\epsilon_i\searrow 0$ such that 
\begin{equation}
     \Diam_{\widehat{\omega}_{t_i,\epsilon_i}} K\cap \cH_{\epsilon_i}\le D \quad \forall i
\label{sing.40a}\end{equation}
for some $D>0$.  Then any subsequence pointed Gromov-Hausdorff limit of $(\cH_{\epsilon_i}, \widehat{\omega}_{t_i,\epsilon_i})$ is naturally homeomorphic to $\cH_0$ and is isometric to the metric completion of $(\cH_0\setminus\{0\}, \widehat{\omega}_{t,0})$.  
\label{sing.40}\end{lemma}
\begin{proof}
  Since $F\in \CI_c(\bbC^{m+1}\setminus\{0\})$ in \eqref{sing.14} is supported away from the origin, the metrics $\widehat{\omega}_{t_i,\epsilon_i}$ are Ricci-flat near the origin, ensuring that there is a uniform lower bound on the Ricci curvature of these metrics.  With the uniform upper bound on the diameter of $K\cap \cH_{\epsilon_i}$ in \eqref{sing.40a}, this ensures that the general convergence theory of \cite{Cheeger-ColdingI, Cheeger-ColdingII, Cheeger-ColdingIII, Cheeger-Colding-Tian} can be applied.  From there, we can argue as in \cite[\S~5 and Proposition~5.3]{Sun-Zhang} or \cite[Propositions~5.2 and 5.5]{CGT} to prove the result.  
\end{proof}

\begin{remark}
By Theorem~\ref{sing.22}, notice that $\widehat{\omega}_{t_i,\epsilon_i}$ converges smoothly locally  to $\widehat{\omega}_{t,0}$ away from the origin.
\label{sing.41}\end{remark}
We can now show that $T$ is closed to complete the proof of Theorem~\ref{sing.37}.

\begin{proof}[Proof of Theorem~\ref{sing.37}]

Combining Lemma~\ref{sing.40} with \eqref{sing.38}, we can now use the contradiction argument of \cite[Proof of Proposition~3.2]{Hein-Sun} to conclude that there exists $D>0$ such that
\begin{equation}
     \Diam_{\widehat{\omega}_{t,\epsilon}} K\cap \cH_{\epsilon}\le D \quad \forall (t,\epsilon)\in[0,1]\times(0,1].
\label{sing.41}\end{equation}
By this bound, Lemma~\ref{sing.40} and Theorem~\ref{sing.22}, we see that $(\cH_0,\widehat{\omega}_{t,0})$ is the pointed Gromov-Hausdorff limit of $(\cH_{\epsilon}, \widehat{\omega}_{t,\epsilon})$ as $\epsilon\searrow 0$ for any $t\in [0,1]$.  Since $\widehat{\omega}_{t,\epsilon}$ is Calabi-Yau near the origin thanks to the fact that $F=0$ there, we can proceed as in \cite[\S~3.2]{Hein-Sun} to establish that the set $T$ is closed.   Within \cite[\S~3]{Hein-Sun}, notice that instead of proceeding as in \cite[\S~3.3]{Hein-Sun} to prove the polynomial rate of convergence of the metric to $g_{\cH_0}$, we can alternatively invoke \cite[Corollary~4.3]{Chiu-Szekelyhidi}.  

\end{proof}

To deduce Corollary~\ref{ucy.1} from Theorem~\ref{sing.37}, we will need the following result.
\begin{lemma}
For $\epsilon>0$, the affine hypersurface $\cH_{\epsilon}$ has non-trivial integral homology classes of degree $m$.  For such a homology class $\mathfrak{c}\in H_m(\cH_{\epsilon};\bbZ)$, 
\begin{equation}
      \inf \{ \area_{g_{\epsilon}}(\Sigma) \; | \: \Sigma \in \mathfrak{c} \; \mbox{is a rectifiable current}\}>0.
\label{nu.1a}\end{equation}
\label{nu.1}\end{lemma}
\begin{proof}
The fact that $H_{m}(\cH_{\epsilon};\bbZ)\ne \{0\}$ is a consequence of a result of Brieskorn-Pham \cite[Corollary~1]{Oka1973}.  On the other hand, to show that the inequality \eqref{nu.1a} holds, we need to suitably adapt standard arguments in geometric measure theory to take into account the fact that $\cH_{\epsilon}$ is not compact.  Thus, fix a representative $\Sigma_0$ of the homology class $\mathfrak{c}$ and let $\{\Sigma_j\}$ be a sequence of rectifiable currents in the homology class $\mathfrak{c}$ such that $\area_{g_{\epsilon}}(\Sigma_j)$ is tending to the infimum in \eqref{nu.1a} as $j\to \infty$.  Since near $\hH_2$, the metric is asymptotically modelled on $(C_0,g_{C_0})$, we can find near $\hH_2$ a smooth vector field $\xi_2$ on $\widehat{\cH}_{\epsilon}$ inward pointing on $\hH_2$ and $\hH_1$ whose flow near $\hH_2$ is a contraction with respect to the metric $g_{\epsilon}$.  Using a smooth cut-off function, we can extend this vector field smoothly to all of $\widehat{\cH}_{\epsilon}$ in such a way that it is zero outside an open neighborhood of $\hH_2\cap \widehat{\cH}_{\epsilon}$ and its flow $\phi_t$ is globally a contraction on $g_{\epsilon}$ for $t\ge 0$.  Applying this flow to the cycles $\Sigma_i$, we can therefore suppose that they all lie outside some fixed open neighborhood of $\hH_2\cap \widehat{\cH}_{\epsilon}$.  Similarly, near $\hH_1$, the metric is asymptotically of the form \eqref{int.5} with $\nu_K=\nu>0$.  This means that we can find another vector field $\xi_1$ inward pointing on $\hH_1\cap\widehat{\cH}_{\epsilon}$ whose flow is a contraction near $\hH_1$.  For instance, in the coordinates of \eqref{int.5}, we can take $\xi_1= -\frac{\pa}{\pa r}$.  Suitably cutting off this vector field, we can assume its flow is globally a contraction on $\cH_{\epsilon}$.  Applying it to the rectifiable currents $\Sigma_j$, we can assume that they are all contained  outside some fixed open neighborhood of $\hH_1\cap \widehat{\cH}_{\epsilon}$.  

In other words, we can assume that the elements of the sequence $\{\Sigma_j\}$ are all contained in some large compact set of $\cH_{\epsilon}$ containing $\Sigma_0$.  From there, we can apply the usual argument in geometric measure theory \cite[Corollary~9.6]{FF1960} to conclude that $\{\Sigma_j\}$ has a subsequence converging in the $\mathcal{F}$-topology to a rectifiable current $\Sigma_{\infty}\in \mathfrak{c}$ whose area is the infimum in \eqref{nu.1a}.  Since $\area_{g_{\epsilon}}(\Sigma_{\infty})>0$, the result follows.
\end{proof}

With this result, the non-uniqueness result of Corollary~\ref{ucy.1} can be deduced as follows.

\begin{proof}[Proof of Corollary~\ref{ucy.1}]
If, except for finitely many values of $\epsilon>0$, the metrics $g_{\epsilon}$ were related by a biholomorphism and a scaling, then by Lemma~\ref{nu.1}, the scaling constant would tend to zero as $\epsilon\searrow 0$ and $(\cH_0,g_0)$ would correspond to a tangent cone at infinity.  But by \cite{Cheeger-Colding1996}, $(\cH_0,g_0)$ would then have to be a metric cone, leading to a contradiction.
\end{proof}





\bibliography{warpedQAC}
\bibliographystyle{amsplain}

\end{document}